\documentclass[11pt,a4paper,reqno]{amsart} 
\usepackage[margin=1in]{geometry}
\usepackage[utf8]{inputenc}
\usepackage[titletoc]{appendix}
\usepackage{amsfonts,amssymb,amscd,graphicx, amsmath,hyperref,subcaption,lscape}
\usepackage{fancybox}
\usepackage{color,float} 
\usepackage{cite} 
\usepackage{setspace}
\usepackage{caption}
\usepackage{mathrsfs}
\DeclareMathOperator{\sech}{sech}

\usepackage[shortlabels]{enumitem}
\usepackage{xcolor}
\usepackage{tikz}
\mathchardef\ordinarycolon\mathcode`\:
\mathcode`\:=\string"8000
\begingroup \catcode`\:=\active
\gdef:{\mathrel{\mathop\ordinarycolon}}
\endgroup
\makeatletter
\renewcommand*\env@matrix[1][\arraystretch]{%
	\edef\arraystretch{#1}%
	\hskip -\arraycolsep
	\let\@ifnextchar\new@ifnextchar
	\array{*\c@MaxMatrixCols c}}
\makeatother
\renewcommand{\vec}[1]{\mathbf{#1}}
\usepackage{graphicx}
\usepackage{amsthm}
\usepackage{rotating}
\newtheorem{theorem}{Theorem}[section]
\newtheorem{proposition}[theorem]{Proposition}

\newtheorem{lemma}[theorem]{Lemma}
\theoremstyle{definition}

\newtheorem{remark}[theorem]{Remark}
\newtheorem{example}[theorem]{Example}
\title[Stability of bimodal planar linear switched systems]{Stability of bimodal planar linear switched systems}
\author[Swapnil Tripathi]{Swapnil Tripathi}
\address{Department of Mathematics\\
	Indian Institute of Science Education and Research Bhopal\\
	Bhopal Bypass Road, Bhauri \\
	Bhopal 462 066, Madhya Pradesh\\
	India}
\email{swapnil93@iiserb.ac.in}
\author[Nikita Agarwal]{Nikita Agarwal}
\address{Department of Mathematics\\
	Indian Institute of Science Education and Research Bhopal\\
	Bhopal Bypass Road, Bhauri \\
	Bhopal 462 066, Madhya Pradesh\\
	India}
\email{nagarwal@iiserb.ac.in}

\date{}

\begin{document} 
	\maketitle

	\begin{abstract} 
		We consider bimodal planar switched linear systems and obtain dwell time bounds which guarantee their asymptotic stability. The dwell time bound obtained is a smooth function of the eigenvectors and eigenvalues of the subsystem matrices. An optimal scaling of the eigenvectors is used to strengthen the dwell time bound. A comparison of our bounds with the dwell time bounds in the existing literature is also presented.  \\
	\end{abstract}
	
	\noindent \textbf{Keywords}: Piecewise continuous dynamical systems, control theory, dwell time, stability, asymptotic stability\\
	\noindent \textbf{2010 Mathematics Subject Classification}: 37N35 (Primary); 93C05, 93D20 (Secondary)
	
	\section{Introduction}
	
	A switched system is a dynamical system which exhibits both discrete and continuous dynamic behaviour. It comprises of a family of continuous time subsystems and a rule describing switching between the subsystems. Such systems are of interest since a large number of physical systems have several modes or states, that is, several dynamical systems are required to describe their behaviour. One of the simplest example of a switched system is a ball bouncing on the floor, see~\cite{liberzon2003switching}. A lot of control engineering applications have switched systems as their underlying model, see~\cite{zhang2010switched,xu2011fuel,belykh2014evolving}. 
	
There are two kinds of switching: state dependent and time dependent. In state dependent switching, the state space of the system is divided into regions. In each region, only one subsystem is active. In this paper, we will focus on time dependent switching between subsystems where a right continuous piecewise constant function, known as a \textit{switching signal}, determines which state is active at which time instant. 

It is interesting to note that a switched system may be unstable even if all the subsystems are stable. Moreover even when some or all the subsystems are unstable, a switched system may be stable. For a time dependent switched system, there are mainly two kinds of stability issues~\cite{liberzon1999basic}: stability under arbitrary switching and stability under constrained switching, such as dwell time constrained and average dwell time constrained. For a linear switched system where all the subsystems are linear, for stability under arbitrary switching, one focuses on finding conditions on the subsystem matrices which guarantee stability of the switched system under any switching strategy. This problem has been extensively explored over the past three decades, see~\cite{narendra1994common,liberzon1999stability,agrachev2001lie,sun2005convergence,mori1997solution,yang2012sufficient,xiang2016necessary}. A survey of results pertaining to the arbitrary switching problem using Lyapunov-based methods is presented in~\cite{lin2009stability}. 

Under constrained switching, the focus is on finding conditions on the switching signal in terms of the subsystems which guarantee stability of the switched system. One of the first works in the direction of stability under constrained switching is discussed in~\cite{morse1996supervisory} in which the concept of \textit{dwell time} was introduced. The authors prove that a switched system with all stable subsystems is stable if large enough time is spent in a subsystem after each switching instance, referred to as \textit{slow switching}. This concept was later generalized to that of \textit{average dwell time} in~\cite{hespanha1999stability} where a similar result was proved. For more results related to constrained switching, we refer to~\cite{margaliot2006stability,morse1996supervisory,geromel2006stability,chesi2010computing}. These results do not provide an explicit expression for the dwell time in terms of the subsystem properties. However, in~\cite{karabacak2009dwell,karabacak2013dwell}, the authors study graph dependent linear switched systems to obtain a dwell time constrained in terms of subsystem matrices, specifically in terms of the distance between eigenvector sets and the cycle ratio of the underlying graph. The theory was further extended and the concept of a \textit{simple loop dwell time} was introduced in~\cite{agarwal2018simple}. This concept allowed for a slow-fast switching mechanism to ensure stability of the switched system. The graph dependent switched systems find use in electrical and power grid systems, where the underlying graph structure varies with time~\cite{belykh2014evolving}.
	
	We refer to~\cite{zhai2001stability,yang2009stabilization,agarwal2019stabilizing} for results about linear switched systems for which not all subsystems are stable. In these works, the primary focus is on stabilizing the switched system. It is noteworthy that the results in~\cite{yang2009stabilization} consider non-linear subsystems as well. For a study of linear switched systems having all unstable subsystems, see~\cite{agarwal2019stabilizing,xiang2014stabilization}. In the latter work, a sufficient condition ensuring stability of switched systems with all unstable subsystems is presented. For all major developments in this rich area, we refer to~\cite{decarlo2000perspectives, sun2006switched} and the references cited therein.
	
	The purpose of this paper is to further improve the dwell time bounds obtained in~\cite{karabacak2009dwell,karabacak2013dwell,agarwal2018simple} for the class of \textit{bimodal planar switched linear systems} -- a switched system on the plane $\mathbb{R}^2$ with two linear subsystems. The dwell time bound obtained in this paper is a smooth function of the eigenvectors and eigenvalues of the subsystem matrices. An optimal scaling of the eigenvectors is used to strengthen the dwell time bound. We also present a comparison of our bounds with the dwell time bounds which exist in the literature. The class of bimodal planar switched linear systems exhibit rich dynamical behaviour. We refer to~\cite{liberzon2003switching,iwatani2006stability} for several examples exhibiting interesting behavior in such systems. Though practical models are usually not linear, studying this class of systems is useful for gaining insight into its non-linear counterpart, which occurs in practice. In~\cite{palejiya2013stability}, an integrated wind turbine and battery system is modeled using a bimodal planar system, and stability issues are discussed using linearization of the subsystems. In the literature, the issues of stability, stabilizability and controllability (in the presence of control) have been studied specifically for bimodal setting. A necessary and sufficient condition for stability under arbitrary switching is given in~\cite{balde2009note} for bimodal planar systems, and in~\cite{eldem2009stability} for bimodal linear systems in $\mathbb{R}^3$. For an overview of results on stabilizability and controllability of bimodal planar linear switched systems, see~\cite{sename2013robust}. We refer to~\cite{iwatani2006stability} for a detailed study of stability and stabilization of multimodal planar linear switched systems and \cite{sun2010stability} for piecewise linear systems. 
	
	\subsection{A switched system}\label{sec:details}
	Given a set of $n\times n$ matrices $\{A_1,\dots,A_p\}$, a \textit{continuous time linear switched system} is defined as
	\begin{eqnarray}\label{eq:system}
		\dot{x}(t)&=&A_{\sigma(t)}x(t), 
	\end{eqnarray}
	where $\sigma:[0,\infty)\rightarrow \{1,\dots,p\}$ is a right continuous piecewise constant function which determines the active subsystem at time instance $t$. The function $\sigma$ is known as a \textit{switching signal}. The flow of the switched system is given by 
	\[
	x(t)=e^{A_{\sigma_i}(t-d_{i-1})}\left(\prod_{k=1}^{i-1}e^{A_{\sigma_k}\Delta_k}\right)x(0),\,\,\,\,\,\, t\in[d_{i-1},d_i),
	\]
	where $d_k$ denotes the $k^{th}$ discontinuity of $\sigma$, $\Delta_k=d_k-d_{k-1}$ denotes the time spent by the switched system in the $k^{th}$ subsystem, and $\sigma_k$ denotes the index of the subsystem active during $[d_{k-1},d_k)$. The switched system~\eqref{eq:system} is said to be 
		\begin{enumerate}
			\item \textit{stable} if for every $\epsilon>0$, there exists $\delta>0$ such that $\| x(0)\|<\delta$ implies $\|x(t)\|<\epsilon$ for all $t>0$;
			\item \textit{asymptotically stable} if it is stable and $\Vert x(t)\Vert\to 0$ as $t\to\infty$ for every initial vector $x(0)\in\mathbb{R}^n$, and
			\item \textit{exponentially stable} if there exist $\alpha,\beta>0$ such that $\|x(t)\|\le \alpha{\rm{e}}^{-\beta t}\|x(0)\|$ for all $t>0$, for every initial vector $x(0)\in\mathbb{R}^n$.
		\end{enumerate} 
	
	For switched linear systems, asymptotic stability and exponential stability are the same notions, refer~\cite[Lemma 1]{hespanha2004uniform}. An asymptotically stable system is stable, but the converse is not true. 
	
	In~\cite{hespanha1999stability}, authors prove that if all the subsystems are asymptotically stable, there is a $\tau>0$ such that if $\Delta_k\ge \tau$, for all $k\ge 1$, then the switched system~\eqref{eq:system} is asymptotically stable. The time $\tau$ that the \textit{switching signal} $\sigma$ spends in each subsystem before switching to another subsystem is known as the \textit{dwell time}. It is an ongoing effort of the researchers in this area to obtain the least dwell time possible, see \cite{morse1996supervisory,geromel2006stability,chesi2010computing,chesi2011nonconservative}.
	
	\subsection{Problem setting}	
	As mentioned earlier, we are going to focus on \textit{planar linear switched systems}~\eqref{eq:system} with only two subsystems ($n=p=2$). Consider two planar matrices $A_1$ and $A_2$ which are both stable (Hurwitz), that is, both the eigenvalues have negative real part. For $i=1,2$, let $J_i$ be the real Jordan form of $A_i$, that is, there exists invertible matrix $P_i$ such that $A_i=P_i J_i P_i^{-1}$. It is enough to consider the switching signal $\sigma$ with $\sigma(0)=1$. Let $(d_j)$ be the set of discontinuities of $\sigma$ with $d_i<d_{i+1}$, for all $i\in\mathbb{N}$. We assume that $d_k\rightarrow \infty$ as $k\rightarrow \infty$ (no zeno behavior). Suppose the switched system spends times $t_j>0$ and $s_j>0$, $j\ge 1$, on consecutive switchings corresponding to subsystem matrices $A_1$ and $A_2$, respectively, as shown in Figure~\ref{fig:signal}.
	\begin{figure}[h!]
		\centering
		\includegraphics[scale=0.3]{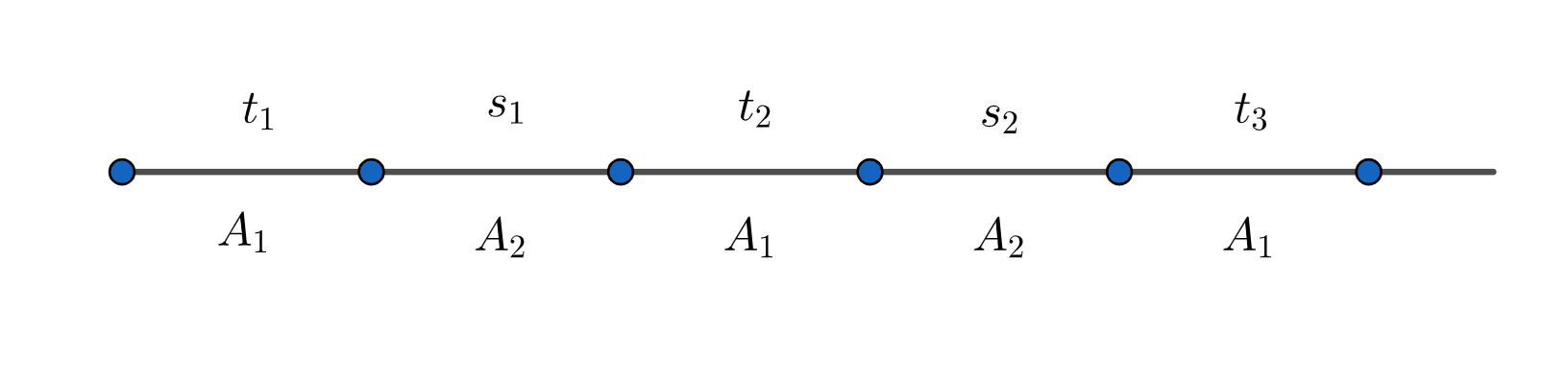}
		\caption{Switching signal $\sigma$ with $\sigma(0)=1$.}
		\label{fig:signal}
	\end{figure}
	
	Consider the following classes of signals:
	\begin{align}\label{eq:signalclass}
			S_{\tau}=\{\, 
			&\sigma:[0,\infty)\rightarrow \{1,2\}\ \vert\  \sigma(0)=1,\ t_j\ge \tau,\ s_j\ge \tau, \ \text{for all } j\in\mathbb{N}\, \},\text{ and}\\
		S_{\tau}'=\{\, 
		&\sigma\in S_{\tau}\ \vert \ \text{ there exists an infinite subset }\Gamma \text{ of } \mathbb{N}\text{ such that }(\tau,\tau) \text{ is not } \nonumber \\
		& \text{a limit point of the set } \{(t_\gamma,s_\gamma)\}_{\gamma\in\Gamma}\, \}. 
	\end{align}
		
		Our goal is to find $\tau>0$ such that the switched system~\eqref{eq:system} is stable for all signals in $S_{\tau}$, and asymptotically stable for all signals in $S'_{\tau}$. Note that $S_\zeta\subset S'_{\tau}$ for all $\zeta>\tau$.
		
	The flow of the switched system~\eqref{eq:system} is given by
	\begin{eqnarray}\label{eq:flow}
	x(t) &=&
	\left\{\begin{aligned}
		& {\rm{e}}^{A_1 (t-d_{2j})}\prod_{i=1}^j \left({\rm{e}}^{A_2 s_i}{\rm{e}}^{A_1 t_i}\right) x(0), & t & \in[d_{2j},d_{2j+1}),\\
		& {\rm{e}}^{A_2(t-d_{2j+1})}{\rm{e}}^{A_1 t_{j+1}}\prod_{i=1}^j \left({\rm{e}}^{A_2 s_i}{\rm{e}}^{A_1 t_i}\right) x(0), &  t & \in[d_{2j+1},d_{2j+2}).
	\end{aligned}\right.
	\end{eqnarray}

	If $\sigma\in S_{\tau}$, then
	\begin{equation}\label{eq:oldform-12}
		\|x(t)\|\le \eta_{1,2}\xi_{1,2} \left(\prod_{i=1}^j \left\|M^{-1}{\rm{e}}^{J_2 s_i}\,M {\rm{e}}^{J_1 t_i} \right\|\right)\|x(0)\|, \ t\in[d_{2j},d_{2j+2}),
	\end{equation}
	where $\xi_{1,2}=\sup_{t\in[\tau,\infty)}\left\{\left\|{\rm e}^{J_1  t}\right\|,\left\|{\rm e}^{J_2 t}\right\|\,\left\|{\rm e}^{J_1  t}\right\|\right\}$, $\eta_{1,2}=\max\{\|P_1\|\|P_1^{-1}\|,\|P_2\|\|P_2^{-1}P_1\|\|P_1^{-1}\|\}$, and $M=P_2^{-1}P_1$. Note that $\xi_{1,2}$ is finite since both $A_1$ and $A_2$ are Hurwitz. 
	
	The terms in~\eqref{eq:flow} can be grouped differently to obtain
	\begin{equation}\label{eq:oldform-21}
		\|x(t)\|\le \eta_{2,1}\xi_{2,1} \left(\prod_{i=1}^j \left\|M{\rm{e}}^{J_1 t_{i+1}}\,M^{-1} {\rm{e}}^{J_2 s_i} \right\|\right)\|x(0)\|, \text{ for }t\in[d_{2j},d_{2j+2}),
	\end{equation}
where $\xi_{2,1}=\sup_{t\in[\tau,\infty)}\left\{\left\|{\rm e}^{J_2  t}\right\|,\left\|{\rm e}^{J_1 t}\right\|\,\left\|{\rm e}^{J_2  t}\right\|\right\}$, $\eta_{2,1}=\max\{\|P_2\|\|P_2^{-1}\|,\|P_1\|\|P_1^{-1}P_2\|\|P_2^{-1}\|\}$.

In \cite{agarwal2018simple,karabacak2013dwell,karabacak2009dwell}, the authors consider $P_1$ and $P_2$ with unit column norms. This condition is restrictive and the dwell time bounds are better when we work with general matrices $P_1$, $P_2$, see~\cite{agarwal2019stabilizing} for instance. Further two Jordan decompositions $PJP^{-1}$ and $P'JP'^{-1}$ of a given matrix $A$ are related through an invertible diagonal matrix. That is, $P'=PD$, for some invertible diagonal matrix $D$. We will call $D$ as a \textit{scaling matrix}. If $J$ is a diagonal matrix, that is, if $A$ is a diagonalizable matrix, then if $A=PJP^{-1}$, then for any diagonal matrix $D$ and $P'=PD$, $A=P'JP'^{-1}$. On the other hand, when $J$ has complex eigenvalues, then any scaling matrix is a constant multiple of the identity matrix. Since both~\eqref{eq:oldform-12} and~\eqref{eq:oldform-21} remain unaltered by changing $P_i$ to a constant multiple of itself, we can choose to work with any fixed choice of $P_i$. When $J$ is defective, that is, the matrix $A$ has repeated eigenvalues with one-dimensional eigenspace, there are several choices for $P$ which arise due to reasons other than scaling, refer to Remark~\ref{nn:optimal}.
	
	The notation $M=P_2^{-1}P_1=\begin{pmatrix}
		a&b\\ c&d
	\end{pmatrix}$ will be fixed throughout the paper and the matrix $M$ will be called the \textit{transition matrix}. For $i=1,2$, the scaling matrix for $P_i$ will be denoted as $D_i=\text{diag}(\lambda_i,\mu_i)$. The matrix $M_{D_1,D_2}=D_2^{-1}P_2^{-1}P_1D_1=\begin{pmatrix}
		a\lambda_1/\lambda_2 & b\mu_1/\lambda_2 \\ c\lambda_1/\mu_2 & d\mu_1/\mu_2
	\end{pmatrix}$ will be called the \textit{scaled transition matrix}, and as we will see, there will be an optimal choice of $D_1$ and $D_2$ among all the possible scaled transition matrices which will give the least dwell time bound. Thus by fixing $P_1$ and $P_2$, thereby $M$,~\eqref{eq:oldform-12} and~\eqref{eq:oldform-21} take the following form:
	\begin{align}
		\|x(t)\| & \le \eta_{1,2}^{(D_1,D_2)}\xi_{1,2} \left(\prod_{i=1}^j \left\|M_{D_1,D_2}^{-1}{\rm{e}}^{J_2 s_i}\,M_{D_1,D_2} {\rm{e}}^{J_1 t_i} \right\|\right)\|x(0)\|, \ t\in[d_{2j},d_{2j+2}), \label{eq:newform-12}\\
		\|x(t)\| & \le \eta_{2,1}^{(D_1,D_2)}\xi_{2,1} \left(\prod_{i=1}^j \left\|M_{D_1,D_2}{\rm{e}}^{J_1 t_{i+1}}\,M_{D_1,D_2}^{-1} {\rm{e}}^{J_2 s_i} \right\|\right)\|x(0)\|, \ t\in[d_{2j},d_{2j+2}). \label{eq:newform-21}
	\end{align} 
	
	Note that $A_1$ and $A_2$ are stable subsystems. We will find an expression for the dwell time $\tau$ in terms of the entries of transition matrix $M$ and the eigenvalues of subsystem matrices $A_1,A_2$. For this, we find $\tau_{1,2}(D_1,D_2)>0$ and $\tau_{2,1}(D_1,D_2)>0$ such that 
	\begin{eqnarray}\label{eq:tau}
		\left\|M_{D_1,D_2}^{-1}{\rm{e}}^{J_2 s}\,M_{D_1,D_2} {\rm{e}}^{J_1 t} \right\| &<&1, \text{ for all } t,s>\tau_{1,2}(D_1,D_2), \nonumber \\
		\left\|M_{D_1,D_2}{\rm{e}}^{J_1 t}\,M_{D_1,D_2}^{-1} {\rm{e}}^{J_2 s} \right\| &<& 1, \text{ for all } t,s>\tau_{2,1}(D_1,D_2).
	\end{eqnarray}

	It is possible to choose matrices $P_1,P_2$ such that the determinant of the transition matrix $M$ is either $1$ or $-1$. Thus, without loss of generality, we will assume that the determinant of $M$ is either $1$ or $-1$ since it does not affect the left hand side of the above inequalities~\eqref{eq:tau}. Moreover, we choose the scaling matrices such that $\det(D_1),\ \det(D_2)=\pm 1$ for the same reason. This choice makes $D_i$ vary over matrices of the form $\text{diag}\left(\lambda_i,\pm 1/\lambda_i \right)$ for $i=1,2$.
	
	\begin{remark}\label{rmk:stab}
		Among all the possible choices of scaling matrices $D_1,D_2$, we will be able to choose optimal scaling matrices giving the least possible values of $\tau_{1,2}(D_1,D_2)$ and $\tau_{2,1}(D_1,D_2)$, denoted as $\tau_{1,2}$ and $\tau_{2,1}$, respectively. 
		\begin{enumerate}
			\item For this optimal choice of $(D_1,D_2)$, we will observe that for $(\ell,k)=(1,2)$ or $(2,1)$, when $\sigma\in S_{\tau_{\ell,k}}$, for all $t>0$, 
				\[
				\|x(t)\|\le \eta_{\ell,k}^{(D_1,D_2)}\xi_{\ell,k}\|x(0)\|,
				\] 
				where $\eta_{\ell,k}^{(D_1,D_2)}=\max\{\|P_\ell D_\ell\|\|D_\ell^{-1}P_\ell^{-1}\|,\|P_k D_k\|\|D_k^{-1}P_k^{-1}P_\ell D_\ell\|\|D_\ell^{-1}P_\ell^{-1}\|\}$. This implies stability of the switched system~\eqref{eq:system} for all $\sigma\in S_{\tau_{\ell,k}}$. 
				\item For this optimal choice of $(D_1,D_2)$, we will observe that the first inequality in~\eqref{eq:tau} becomes non-strict in the region $t,s\ge \tau_{1,2}(D_1,D_2)$ with equality possible only when $t=s=\tau_{1,2}(D_1,D_2)$. Moreover the second inequality in~\eqref{eq:tau} becomes non-strict in the region $t,s\ge \tau_{2,1}(D_1,D_2)$ with equality possible only when $t=s=\tau_{2,1}(D_1,D_2)$. Thus for $\sigma\in S'_{\tau_{\ell,k}}$, $\|x(t)\|$ is bounded above by a scalar multiple of $\rho^r$, for some $\rho<1$, where $r$ is cardinality of the set $\Gamma\cap\{1,\dots,j\}$. This implies asymptotic stability of the switched system~\eqref{eq:system}, since $t\to \infty$ implies $r\to\infty$. Due to this, the switched system~\eqref{eq:system} will be asymptotic stable for all $\sigma \in S'_{\tau_{\ell,k}}$.
			\end{enumerate}
				\end{remark}
			
			\subsection{Definitions and notations}
			
	A $n\times n$ matrix $A$ is called \textit{Schur stable} if its spectral radius $\rho(A)<1$. A planar matrix $A$ is Schur stable if and only if $\lvert\text{tr}(A)\rvert<1+\text{det}(A)$ and $\lvert\text{det}(A)\rvert<1$, see~\cite{fleming1998schur}. These two equivalent conditions for Schur stability will be referred to as \textit{Schur's conditions} in this paper. Further, for any matrix $K$, $\|K\|<1$ if and only if $K^\top K$ is Schur stable. 
	
	For $K=M_{D_1,D_2}^{-1}{\rm{e}}^{J_2 s}\,M_{D_1,D_2} {\rm{e}}^{J_1 t}$, or $M_{D_1,D_2}{\rm{e}}^{J_1 t}\,M_{D_1,D_2}^{-1} {\rm{e}}^{J_2 s}$, using Schur stability, $\|K\|<1$ is equivalent to $\lvert\text{tr}(K^\top K)\rvert<1+\text{det}(K^\top K)$ and $\lvert \det(K^\top K)\rvert<1$ being satisfied simultaneously. Here $K^\top$ denotes the transpose of the matrix $K$. The condition $\lvert \det(K^\top K)\rvert<1$ determines a region which we will refer to as the \textit{feasible region}. Moreover, the zero level curve of $\lvert\text{tr}(K^\top K)\rvert-1-\text{det}(K^\top K)=0$ arising from the other condition will be simplified and studied. This simplified function will be called the \textit{Schur's function}.
	
	We will denote the open first quadrant by $Q_1$; the partial derivative of a function $f$ with respect to $t$ will be denoted by $(\partial/\partial t)f$ or $f_t$.
	
\section{Organization of the paper}\label{sec:organization}
A planar Hurwitz matrix has three possible Jordan forms 
\[
\begin{pmatrix}
	-p & 0 \\ 0&-q
\end{pmatrix}, \ \begin{pmatrix}
	-\alpha & \beta \\ -\beta&-\alpha
\end{pmatrix}, \begin{pmatrix}
	-n & 1 \\ 0&-n
\end{pmatrix},
\]
where $p,q,\alpha,n>0$, $\beta\ne 0$. 

	Let $J_1$ and $J_2$ be the real Jordan forms of $A_1$ and $A_2$, respectively. If $J_i=\text{diag}(-p_i,-q_i)$ (that is, $A_i$ is real diagonalizable), for $i=1,2$, we will assume $p_i\ne q_i$, since otherwise the subsystems commute and hence the switched system~\eqref{eq:system} is stable. 
	
	The sections have been divided according to the forms of the real Jordan forms of the subsystem matrices, as follows:
	\begin{enumerate}
		\item Both subsystem matrices $A_1,A_2$ are real diagonalizable, discussed in Section~\ref{sec:RR}.
		\item Both subsystem matrices $A_1,A_2$ have complex eigenvalues, discussed in Section~\ref{sec:CC}.
		\item One of the subsystem matrices is real diagonalizable and the other has complex eigenvalues, discussed in Section~\ref{sec:RC}. In this section, we assume that $A_1$ is real diagonalizable and $A_2$ has complex eigenvalues.
		\item Both subsystem matrices $A_1, A_2$ are defective, discussed in Section~\ref{sec:NN}.
		\item One of the subsystem matrices is defective and the other has complex eigenvalues, discussed in Section~\ref{sec:NC}. In this section, we assume that $A_1$ is defective and $A_2$ has complex eigenvalues.
		\item One of the subsystem matrices is defective and the other is real diagonalizable, discussed in Section~\ref{sec:NR}. In this section, we assume that $A_1$ is defective and $A_2$ is real diagonalizable.
	\end{enumerate}

Observe that the above list takes care of all possible Jordan form combinations of the subsystem matrices. In each of the sections, we give results to compute $\tau_{1,2}$ and $\tau_{2,1}$. Using~\eqref{eq:newform-12},~\eqref{eq:newform-21},~\eqref{eq:tau}, and Remark~\ref{rmk:stab}, the switched system~\eqref{eq:system} is stable for all signals $\sigma\in S_{\tau_{1,2}}\bigcup S_{\tau_{2,1}}$ and is asymptotically stable for all signals $\sigma\in S_{\tau_{1,2}}'\bigcup S_{\tau_{2,1}}'$. Hence if $\tau=\min\{\tau_{1,2},\tau_{2,1}\}$, then for each $\sigma\in S_\tau$, the switched system~\eqref{eq:system} is stable and for each $\sigma\in S_\tau'$, the switched system~\eqref{eq:system} is asymptotically stable. 
	
The results of~\cite{karabacak2009dwell,karabacak2013dwell,agarwal2018simple} can also be applied to compute dwell time $\tau$ for the system~\eqref{eq:system}. These results make use of \emph{complex} Jordan basis matrices instead of \emph{real} Jordan basis matrices. However it turns out that the the results in~\cite{karabacak2009dwell,karabacak2013dwell,agarwal2018simple} remain unaltered when real Jordan basis matrices are used. Further due to the estimates we use and the introduction of scaling matrices, our results give better dwell time bounds than those in \cite{karabacak2009dwell,karabacak2013dwell,agarwal2018simple}. In fact both $\tau_{1,2}$ and $\tau_{2,1}$ obtained here are lower than the dwell time bounds obtained in the existing literature cited above.	

In Section~\ref{sec:comparison}, we give a comparison between our dwell time bounds with those in the current literature. A particular scenario where the two subsystems share a common eigenvector is discussed in Section~\ref{sec:examples}. In that section, we also generalize our results to the setting of a multimodal planar system in which the switching between the subsystems is governed by a flower-like graph.
	
	\section{Both $A_1$ and $A_2$ are real diagonalizable}\label{sec:RR}

In this section, we assume that both the subsystem matrices $A_1,A_2$ have distinct eigenvalues, in which case, $J_i= \text{diag}(-p_i,-q_i)$, where $0<p_i <q_i$, for $i=1,2$. Recall, for $i=1,2$, $D_i=\text{diag}(\lambda_i,\pm 1/\lambda_i)$, $M=P_2^{-1}P_1$ with $\det(M)=\pm 1$, and $M_{D_1,D_2}=D_2^{-1}MD_1$. This section is divided into two subsections, the first one is devoted to computing $\tau_{1,2}$ and the second one to $\tau_{2,1}$. The results in the second one follow easily from the first one by a simple observation, which is by interchanging the roles of $J_1$ and $J_2$ and by replacing $M_{D_1,D_2}$ by $M_{D_1,D_2}^{-1}$.  	

\subsection{Computing $\tau_{1,2}$}	
Refer to~\eqref{eq:tau}, $\left\|M_{D_1,D_2}^{-1}{\rm{e}}^{J_2 s}\,M_{D_1,D_2} {\rm{e}}^{J_1 t} \right\|<1$ if and only if $f(\lambda_1,t,s)<0$, where 

\begin{eqnarray}\label{mainfunc} f(\lambda_1,t,s)&=&\left(ad\, {\rm{e}}^{-p_1t-p_2s}-bc\, {\rm{e}}^{-p_1t-q_2s}\right)^2+ \frac{b^2d^2}{\lambda_1^4}\left({\rm{e}}^{-p_2s-q_1t}-{\rm{e}}^{-q_1t-q_2s}\right)^2\nonumber \\
	&&+a^2c^2\,\lambda_1^4\left({\rm{e}}^{-p_1t-q_2s}-{\rm{e}}^{-p_1t-p_2s}\right)^2+\left(ad\,{\rm{e}}^{-q_1t-q_2s}-bc\, {\rm{e}}^{-p_2s-q_1t}\right)^2\nonumber\\
	&&-1-{\rm{e}}^{-2(p_1+q_1)t-2(p_2+q_2)s}.
\end{eqnarray}

We will find $\tau_{1,2}$ (depending on $\lambda_1$) such that $f(\lambda_1,t,s)<0$, for all $t,s>\tau_{1,2}$. Also we would like to find an optimal choice of $\lambda_1$ which gives the least bound $\tau_{1,2}$.

\begin{proposition}\label{rr:thm:ad=0,1}
	If one of the entries of $M$ is zero, then the switched system~\eqref{eq:system} is stable for all signals $\sigma$. 
\end{proposition}

\begin{proof}
	We assume that $b=0$, the other cases follow similarly. Since $\det(M)=\pm 1$, $a\ne 0$ and $d=\pm 1/a$. Substituting this in~\eqref{mainfunc}, $f(\lambda_1,t,s)<0$ if and only if
	\begin{eqnarray*}
		a^2c^2\lambda_1^4 &<&\frac{\left(1-{\rm{e}}^{-2q_1 t-2q_2 s}\right)\left(1-{\rm{e}}^{-2p_1 t-2p_2 s}\right)}{{\rm{e}}^{-2p_1 t}\left({\rm{e}}^{-q_2 s}-{\rm{e}}^{-p_2 s}\right)^2}=m(t,s).
	\end{eqnarray*}
	
	For all $t,s>0$, the function $m(t,s)$ is decreasing in $t$. Hence $m(t,s)\ge m(0,s)$, for all $t,\,s>0$. Also the function $m(0,s)$ is increasing in $s$. Hence $m(t,s)\ge \lim_{s\to 0}\,m(0,s)=4p_2q_2/(q_2-p_2)^2$. Thus $f(\lambda_1,t,s)<0$, for all $t,s>0$, for any $\lambda_1$ satisfying 
	\begin{eqnarray*}
		\lambda_1<\sqrt[4]{\frac{1}{a^2 c^2}\frac{4p_2q_2}{(q_2-p_2)^2}}.
	\end{eqnarray*}
	Hence the result follows.		
\end{proof}

\noindent In view of the above proposition, we restrict ourselves to the case when none of the entries of $M$ is zero. Thus $ad\notin\{0,1\}$ hence otherwise one of the entries of $M$ is zero. Define the function 
\begin{eqnarray*}
	k(t,s)=f\left(\sqrt[8]{\frac{b^2d^2}{a^2c^2}\, {\rm{e}}^{-2(q_1-p_1)t}},t,s\right).
\end{eqnarray*}
Observe that, for each $\lambda_1$ and for all $t,s\ge 0$,
\[
k(t,s)\le f(\lambda_1,t,s).
\]
It will turn out that studying the zero set $\mathcal{C}$ of $k$ in the first quadrant $Q_1$ of the $(t,s)$-plane, that is $\mathcal{C}=\{(t,s) \ \vert\ k(t,s)=0\}$, is enough to compute $\tau_{1,2}$.

\begin{lemma}\label{ktsform} 
	Let $M$ have all nonzero entries and $\epsilon=\dfrac{abcd}{|abcd|}$, then $k(t,s)$ can be expressed as
	\[
	\left[ad\left({\rm{e}}^{-p_1t-p_2s}-\epsilon{\rm{e}}^{-q_1t-q_2s}\right)-bc\left({\rm{e}}^{-p_1t-q_2s}-\epsilon{\rm{e}}^{-p_2s-q_1t}\right)\right]^2-\left({\rm{e}}^{-(p_1+q_1)t-(p_2+q_2)s}-\epsilon\right)^2.
	\]
\end{lemma}
\begin{proof}
	Follows by a simple rearrangement of terms. Observe that $\epsilon=-1$ if and only if $0<ad<1$. 
\end{proof}

\begin{lemma}\label{rr:lemma:k0s}
	Let $\mathcal{I}=\left[-\frac{p_2}{q_2-p_2},\frac{q_2}{q_2-p_2}\right]$. If $ad\in\mathcal{I}\setminus\{0,1\}$, then $k(0,s)<0$, for all $s>0$. If $ad\notin\mathcal{I}$, then there is a unique $s_0>0$ such that $k(0,s_0)=0$.
\end{lemma}

\begin{proof}
	Recall the function $m(0,s)$ defined in Proposition~\ref{rr:thm:ad=0,1} and observe that 
	\begin{align*}
		k(0,s)&{}=\left({\rm{e}}^{-p_2 s}-{\rm{e}}^{-q_2 s}\right)^2\left(4ad(ad-1)-m(0,s)\large\right).
	\end{align*}
	The function $m(0,s)$ is increasing in $s$ and $\lim_{s\to 0}\,m(0,s)=4p_2q_2/(q_2-p_2)^2$. The equation $4ad(ad-1)=4p_2q_2/(q_2-p_2)^2$, in $ad$, has two solutions given by $ad=q_2/(q_2-p_2)$ and $ad=-p_2/(q_2-p_2)$. Therefore when $ad\in\mathcal{I}\setminus\{0,1\}$, $k(0,s)<0$, for all $s>0$. Further when $ad\notin\mathcal{I}$, $\lim_{s\to 0}\left(4ad(ad-1)-m(0,s)\right)>0$. Since the function $m(0,s)$ is increasing, there exists a unique $s_0>0$ such that $k(0,s_0)=0$. Also, $k(0,s)>0$, for all $s\in(0,s_0)$ and $k(0,s)<0$, for all $s>s_0$.
\end{proof}

\begin{proposition}\label{rr:thm:adinI}
	If $ad\in\mathcal{I}\setminus\{0,1\}$, then the switched system~\eqref{eq:system} is stable for all signals $\sigma$.
\end{proposition}

\begin{proof}
	Let $\lambda_1=\sqrt[8]{\frac{b^2d^2}{a^2c^2}}$. It follows from the definition of $k$ that $
	f(\lambda_1,0,s)=k(0,s)$. Hence by Lemma~\ref{rr:lemma:k0s}, $f(\lambda_1,0,s)<0$, for all $s>0$. Therefore $\left\Vert M_{D_1,D_2}^{-1} {\rm{e}}^{J_2 s} . \, M_{D_1,D_2}\right\Vert<1$, for all $s>0$. Since $\left\Vert M_{D_1,D_2}^{-1} {\rm{e}}^{J_2 s} . \, M_{D_1,D_2} {\rm{e}}^{J_1t}\right\Vert\le \left\Vert M_{D_1,D_2}^{-1} {\rm{e}}^{J_2 s} . \, M_{D_1,D_2}\right\Vert$, for all $t,s>0$, we obtain $f(\lambda_1,t,s)<0$, for all $t,s>0$. Hence the result follows.
\end{proof}

Since $(0,1)\subset\mathcal{I}$, the case when $\text{sgn}(abcd)<0$ is included in the hypothesis of Proposition~\ref{rr:thm:adinI}. Thus assume that $ad\notin\mathcal{I}$, in which case $\text{sgn}(abcd)>0$. Using Lemma~\ref{ktsform} and the relation $bc=ad-1$, the zero set $\mathcal{C}$ of $k$ is given by $\mathcal{C}=\{(t,s)\colon\, ad=\ell_\pm (t,s)\}$, where the functions $\ell_\pm$ are defined for $s\ne 0$ as
\[
	\ell_\pm(t,s)=\frac{\left(1\pm{\rm e}^{-p_1 t-q_2 s}\right)\left({\rm e}^{-p_2 s-q_1 t}\mp 1\right)}{\left({\rm e}^{-p_1 t}+{\rm e}^{-q_1 t}\right)\left({\rm e}^{-p_2 s}-{\rm e}^{-q_2 s}\right)}.
\]

\noindent Note that $\ell_+(t,s)<0$ and $\ell_-(t,s)>0$, for all $s,t> 0$.

\begin{lemma}\label{rr:lemma:ad:notinI}
	(With $s_0>0$ as obtained in Lemma~\ref{rr:lemma:k0s}) Given $ad\notin\mathcal{I}$, for each $s_1\in[0,s_0]$, there exists a unique $t_1\ge 0$ (depending on $s_1$) such that $k(t_1,s_1)=0$. For $s_1>s_0$, $k(t,s_1)<0$, for all $t\ge 0$.
\end{lemma}

\begin{proof}
	See Appendix~\ref{ap:proof} for a proof. 
\end{proof}

\begin{lemma}\label{rr:zero:IFT}
	(With $s_0>0$ as obtained in Lemma~\ref{rr:lemma:k0s}) Given $ad\notin \mathcal{I}$, there exists a continuous function $\mathcal{O}\colon[0,s_0]\to \mathbb{R}$ such that $\mathcal{C}=\{\left(\mathcal{O}(s),s\right)\,\mid\,s\in[0,s_0]\}$. Moreover $\mathcal{O}$ is continuously differentiable on $(0,s_0)$.
\end{lemma}

\begin{proof}
	At several places in this proof, we will use details from the proof of Lemma~\ref{rr:lemma:ad:notinI} given in Appendix~\ref{ap:proof}. \\
	First assume that $ad>q_2/(q_2-p_2)$. We fix $s_1\in(0,s_0)$ and study $k(t,s_1)$ as a function of $t$. Recall $k(t,s_1)=0$ if and only if $ad=\ell_-(t,s_1)$. Also there exists a unique $t_1\ge 0$ satisfying $ad=\ell_-(t_1,s_1)$. Since $(\partial \ell_-/\partial t)(t_1,s_1)>0$, we apply the implicit function theorem at $(t_1,s_1)\in\mathcal{C}$ to locally parameterize the zero set of $k$ by $\mathcal{O}\colon U_{s_1}\to\mathbb{R}$, where $U_{s_1}$ is an open interval containing $s_1$. That is, $k(\mathcal{O}(s),s)=0$, for all $s\in U_{s_1}$. Repeating the process for each $s_1\in (0,s_0)$, we get a $\mathcal{C}^1$ parametrization $\mathcal{O}\colon\, (0,s_0)\to \mathbb{R}$ of a part of the zero level-set of $k$, details are given in Appendix~\ref{ap:proof}. We can continuously extend this function to the endpoints $0$ and $s_0$. Clearly, the graph of this function covers the zero set $\mathcal{C}$ of $k$ due to Lemma~\ref{rr:lemma:ad:notinI}.
	
	We can use similar arguments when $ad<-p_2/(q_2-p_2)$. In this case, $k(t_1,s_1)=0$ if and only if $ad=\ell_+(t_1,s_1)$. Also we use the fact that $\ell_+(t,s_1)$ is strictly decreasing in $t$.
\end{proof}

\begin{lemma}\label{rr:lemma:uniquebulge}
	(With $s_0>0$ as obtained in Lemma~\ref{rr:lemma:k0s}) Given $ad\notin\mathcal{I}$, for each $t_0\ge 0$, $k(t_0,s)$ has at most two zeros as a function of $s$. As a consequence, the function $\mathcal{O}$ has a unique local maxima at $\tilde{s}$ in the interval $(0,s_0)$.\\	
	(i) When $ad>q_2/(q_2-p_2)$, the tuple $(\mathcal{O}(\tilde{s}),\tilde{s})$ is given by the unique nonzero solution $(t,s)\in\mathcal{C}$ of the equation
	\begin{eqnarray}\label{findingmaxpt}
		\frac{({\rm e}^{p_1 t}-{\rm e}^{-q_2 s})({\rm e}^{q_1 t}+{\rm e}^{-q_2 s})}{({\rm e}^{p_1 t}-{\rm e}^{-p_2 s})({\rm e}^{q_1 t}+{\rm e}^{-p_2 s})}{\rm e}^{(q_2-p_2)s}&=&\frac{q_2	}{p_2}.
	\end{eqnarray}	
	(ii) When $ad<-p_2/(q_2-p_2)$, the tuple $(\mathcal{O}(\tilde{s}),\tilde{s})$ is given by the unique nonzero solution $(t,s)\in\mathcal{C}$ of the equation
	\begin{eqnarray}\label{findingmaxpt2}
		\frac{({\rm e}^{q_1 t}-{\rm e}^{-q_2 s})({\rm e}^{p_1 t}+{\rm e}^{-q_2 s})}{({\rm e}^{q_1 t}-{\rm e}^{-p_2 s})({\rm e}^{p_1 t}+{\rm e}^{-p_2 s})}{\rm e}^{(q_2-p_2)s}&=&\frac{q_2	}{p_2}.
	\end{eqnarray}
\end{lemma}

\begin{proof}
	See Appendix~\ref{ap:proof} for proof. 
\end{proof}

\begin{example}\label{egrr}
	Let $A_1=\begin{pmatrix}
		0.01 & -0.1\\ 0.231 & -0.31
	\end{pmatrix}$ and $A_2=\begin{pmatrix}
		-0.1 & 0\\ 0 & -0.2
	\end{pmatrix}$. Then $M=P_2^{-1}P_1=\begin{pmatrix}
		1 & 1\\ 1.1 & 2.1
	\end{pmatrix}$, hence $ad=2.1$. Thus we use Lemma~\ref{rr:lemma:uniquebulge} (i) to compute $(\mathcal{O}(\tilde{s}),\tilde{s})$. Denoting the expression on the left hand side of~\eqref{findingmaxpt} as \begin{eqnarray*}
		L(t,s)&=&\frac{({\rm e}^{0.1 t}-{\rm e}^{-0.2 s})({\rm e}^{0.2 t}+{\rm e}^{-0.2 s})}{({\rm e}^{0.1 t}-{\rm e}^{-0.1 s})({\rm e}^{0.2 t}+{\rm e}^{-0.1 s})}{\rm e}^{0.1 s},
	\end{eqnarray*}
	the solution of~\eqref{findingmaxpt} is the point of intersection of the curve $L(t,s)=2$ with the zero set $\mathcal{C}$ of $k$, refer Figure~\ref{fig:egrr}. The solution $(0.155351,2.4064)$ was computed in Wolfram Mathematica 11.0 using the command: \verb|FindRoot[{k[t,s]==0,L[t,s]==2},{{t,0.1},{s,2.5}}]|.
	\begin{figure}[h!]
		\centering
		\includegraphics[width=.4\textwidth]{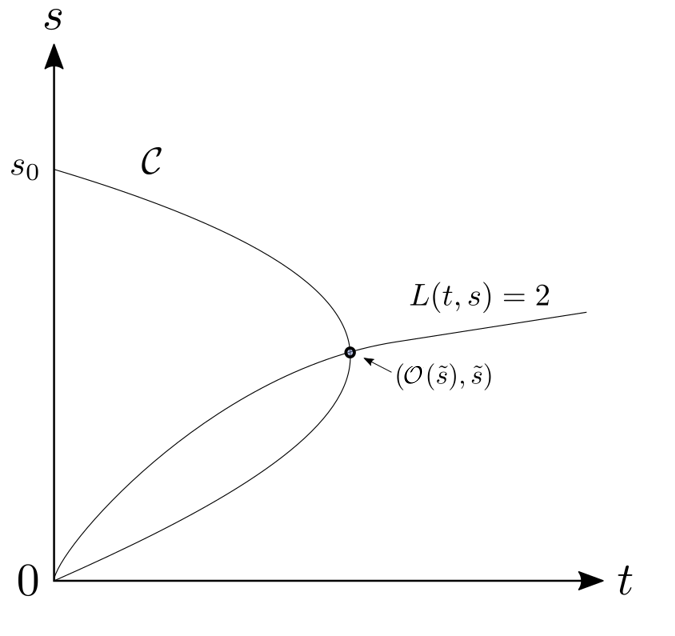}
		\caption{The curves in Example~\ref{egrr}.}
		\label{fig:egrr}
	\end{figure}
	
\end{example}

\begin{remark}\label{zerosetpos} 
	Using the results obtained so far, the following facts about $k(t,s)$ are worth mentioning when $ad\notin \mathcal{I}$.
	\begin{enumerate}[(i)]
		\item By Lemma~\ref{rr:zero:IFT}, the zero set $\mathcal{C}$ is the graph of continuous function $t=\mathcal{O}(s)$ in the domain $[0,s_0]$, where $\mathcal{O}$ is differentiable in $(0,s_0)$. Also $\mathcal{O}(s_0)=\mathcal{O}(0)=0$. 
		\item The zero set $\mathcal{C}$ splits the first quadrant $Q_1$ into $\mathcal{B}\cup \mathcal{U}\cup \mathcal{C}$, where $\mathcal{B}$ and $\mathcal{U}$ are bounded and unbounded regions, respectively. The function $k(t,s)$ is strictly positive on $\mathcal{B}$ and strictly negative on $\mathcal{U}$ by Lemma~\ref{rr:lemma:k0s} and using the continuity of $k$. See Figure~\ref{fig:zerosetpos}.
		
		\begin{figure}[h!]
			\centering
			\includegraphics[scale=0.5]{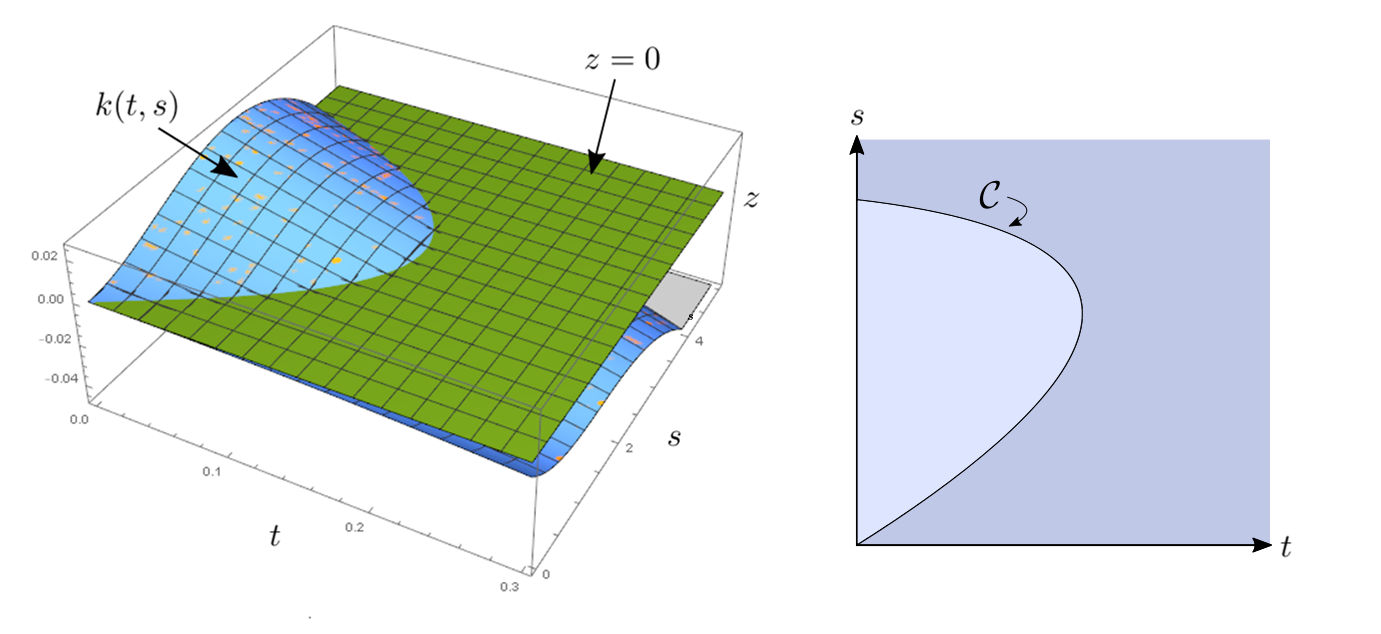}
			\caption{Plot showing $\mathcal{C}=\{(t,s)\in Q_1\ \vert\ k(t,s)=0\}$, for $ad\notin\mathcal{I}$. The first plot corresponds to values in Example~\ref{egrr}.}
			\label{fig:zerosetpos}
		\end{figure}
	\end{enumerate}
\end{remark}

\begin{lemma}\label{rr:direction:k}
	Given $ad>q_2/(q_2-p_2)$ and $\gamma\ge 0$, there is at most one positive solution $t$ of $k(t,\gamma t)=0$. 
\end{lemma}

\begin{proof}
	See Appendix~\ref{ap:proof} for a proof.
\end{proof}

\begin{remark}\label{zeroset}
	Using the results obtained so far, the following observations can be made about the zero set $\mathcal{C}$ of the function $k$, when $ad\notin\mathcal{I}$. 
	\begin{enumerate}
		\item When $ad>q_2/(q_2-p_2)$, as a consequence of Lemma~\ref{rr:lemma:uniquebulge} and Remark~\ref{zerosetpos}, there are three possible classifications of the zero set $\mathcal{C}$, as shown in Figure~\ref{fig:threecases}.
		\begin{figure}[h!]
			\centering
			\includegraphics[width=.9\textwidth]{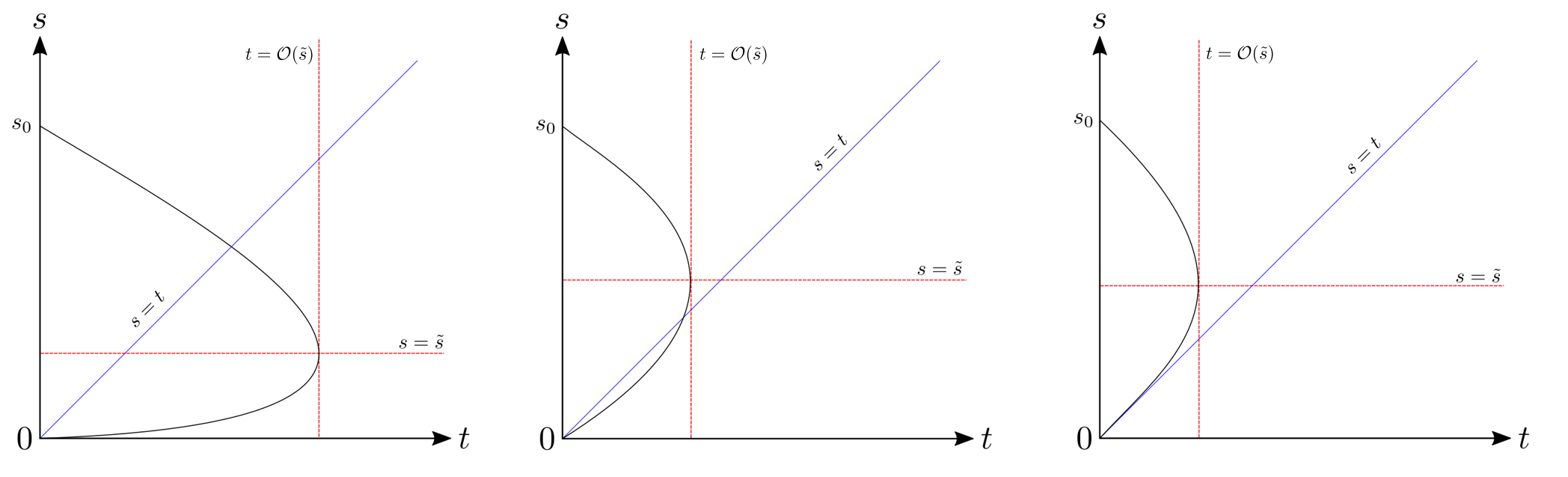}
			\caption{The three cases when $ad>q_2/(q_2-p_2)$.}
			\label{fig:threecases}
		\end{figure}
		The first one corresponds to when $\mathcal{O}(\tilde{s})>\tilde{s}$, and the other two correspond to when $\mathcal{O}(\tilde{s})\le \tilde{s}$.
		\item When $ad<-p_2/(q_2-p_2)$, there is an additional possibility apart from the above three, which is shown in Figure~\ref{fig:addcase}. Here, there are two roots of $k$ on the line $s=t$. Also $\mathcal{O}(\tilde{s})>\tilde{s}$. This possibility does not arise when $ad>q_2/(q_2-p_2)$ due to Lemma~\ref{rr:direction:k}.
		\begin{figure}[h!]
			\centering
			\includegraphics[width=.3\textwidth]{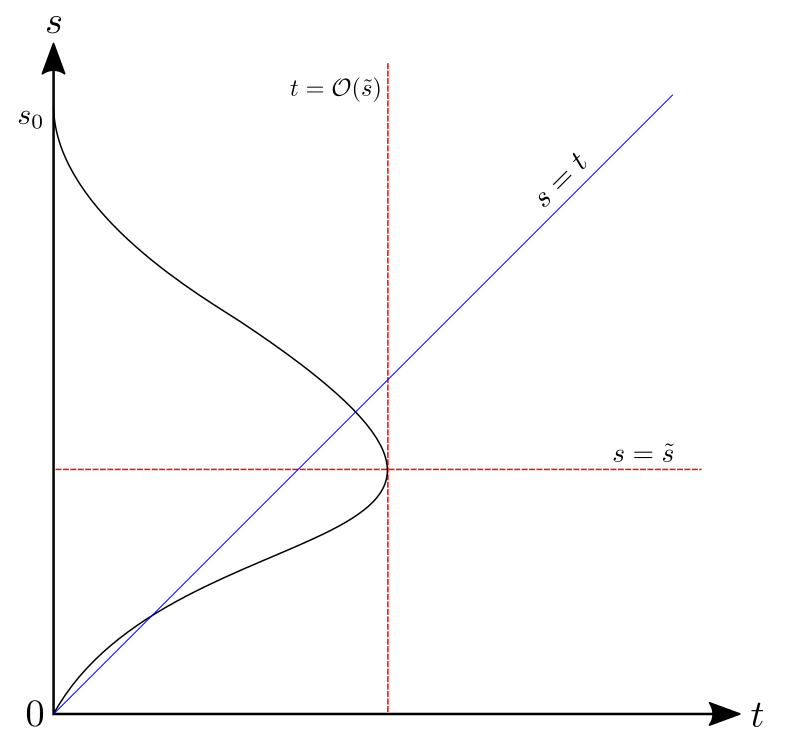}
			\caption{An additional case when $ad<-p_2/(q_2-p_2)$.}
			\label{fig:addcase}
		\end{figure}
		
	\end{enumerate}
\end{remark}

For $t_0>0$, let $\lambda_0=\sqrt[8]{\frac{b^2d^2}{a^2c^2}\, {\rm{e}}^{-2(q_1-p_1)t_0}}$. It is easy to see that $f(\lambda_0,t_0,s)=k(t_0,s)$, for all $s\ge 0$.

Recall that $\left\Vert M_{D_1,D_2}^{-1} {\rm{e}}^{J_2 s} . \, M_{D_1,D_2} {\rm{e}}^{J_1 t}\right\Vert<1$ if and only if $f(\lambda_1,t,s)<0$. For a fixed $s$, the term $\left\Vert M_{D_1,D_2}^{-1} {\rm{e}}^{J_2 s} . \, M_{D_1,D_2} {\rm{e}}^{J_1 t}\right\Vert$ is strictly decreasing in $t$ on $[0,\infty)$. Thus, for each $\lambda_1$ and $s$, $f(\lambda_1,t,s)$ has at most one root as a function of $t$. Analogous to Lemma~\ref{rr:lemma:k0s}, we have the following result.

\begin{lemma}\label{rr:lemma:f}
	Given $ad\notin \mathcal{I}$, there is a unique $\widetilde{s_0}>0$ such that $f(\lambda_0,0,\tilde{s}_0)=0$.
\end{lemma}

\begin{proof}
	See Appendix~\ref{ap:proof} for a proof.
\end{proof}

\begin{lemma}\label{rr:funiqueroot}
	(With $\widetilde{s_0}>0$ as obtained in Lemma~\ref{rr:lemma:f}) Given $ad\notin\mathcal{I}$, for each $s_1\in[0,\widetilde{s_0}]$, there exists a unique $t_1\ge 0$ (depending on $s_1$) such that  $f(\lambda_0,t_1,s_1)=0$. For $s_1>\widetilde{s_0}$, $f(\lambda_0,t,s_1)<0$, for all $t\ge0$.
\end{lemma}

\begin{proof}
	Follows from Lemma~\ref{rr:lemma:f} and arguments preceding it.
\end{proof}

\noindent Now we present the main result of this section when both subsystems are real diagonalizable.

\begin{theorem}\label{rr:thm12}
	(With notations in this section) \\
	i) If $ad\in\mathcal{I}$, then the switched system~\eqref{eq:system} is stable for all signals $\sigma$. Let $\tau_{1,2}=0$.\\	
	ii) If $ad>q_2/(q_2-p_2)$, calculate the unique nonzero solution $(t_{sol},s_{sol})$ satisfying the system
	\begin{eqnarray}\label{eq:2}
		\frac{({\rm e}^{p_1 t}-{\rm e}^{-q_2 s})({\rm e}^{q_1 t}+{\rm e}^{-q_2 s})}{({\rm e}^{p_1 t}-{\rm e}^{-p_2 s})({\rm e}^{q_1 t}+{\rm e}^{-p_2 s})}{\rm e}^{(q_2-p_2)s}&=&\frac{q_2	}{p_2},\\
		\left(ad\left({\rm{e}}^{-p_1t-p_2s}-{\rm{e}}^{-q_1t-q_2s}\right)-bc\left({\rm{e}}^{-p_1t-q_2s}-{\rm{e}}^{-p_2s-q_1t}\right)\right)^2&=&\left({\rm{e}}^{-(p_1+q_1)t-(p_2+q_2)s}-1\right)^2.\nonumber
	\end{eqnarray}
\begin{enumerate}[(a)]
	\item If $s_{sol}\ge t_{sol}$, let $\tau_{1,2}=t_{sol}$. \\
	\item Otherwise if $s_{sol}< t_{sol}$, let $\tau_{1,2}$ be the unique positive solution $t$  of \begin{eqnarray*}
		\left(ad\left({\rm{e}}^{-p_1t-p_2t}-{\rm{e}}^{-q_1t-q_2t}\right)-bc\left({\rm{e}}^{-p_1t-q_2t}-{\rm{e}}^{-p_2t-q_1t}\right)\right)^2&=&\left({\rm{e}}^{-(p_1+q_1+p_2+q_2)t}-1\right)^2.
	\end{eqnarray*}
\end{enumerate}
		iii) If $ad<-p_2/(q_2-p_2)$, calculate the unique nonzero solution $(t_{sol},s_{sol})$ satisfying the system
	\begin{eqnarray}\label{eq:3}
		\frac{ ({\rm e}^{q_2 s+p_1 t}+1)({\rm e}^{q_2 s+q_1 t}-1)}{({\rm e}^{p_2 s+p_1 t}+1)({\rm e}^{p_2 s+q_1 t}-1)}{\rm e}^{-(q_2-p_2)s}&=&\frac{q_2	}{p_2},\\
		\left(ad\left({\rm{e}}^{-p_1t-p_2s}-{\rm{e}}^{-q_1t-q_2s}\right)-bc\left({\rm{e}}^{-p_1t-q_2s}-{\rm{e}}^{-p_2s-q_1t}\right)\right)^2&=&\left({\rm{e}}^{-(p_1+q_1)t-(p_2+q_2)s}-1\right)^2.\nonumber
	\end{eqnarray}
\begin{enumerate}[(a)]
	\item If $s_{sol}\ge t_{sol}$, let $\tau_{1,2}=t_{sol}$. 
	\item Otherwise if $s_{sol}< t_{sol}$, let $\tau_{1,2}$ be the unique solution $t>s_{sol}$ of \begin{eqnarray*}
		\left(ad\left({\rm{e}}^{-p_1t-p_2t}-{\rm{e}}^{-q_1t-q_2t}\right)-bc\left({\rm{e}}^{-p_1t-q_2t}-{\rm{e}}^{-p_2t-q_1t}\right)\right)^2&=&\left({\rm{e}}^{-(p_1+q_1+p_2+q_2)t}-1\right)^2.
	\end{eqnarray*}
\end{enumerate}
	Then the switched system~\eqref{eq:system} is stable for all signals $\sigma\in S_{\tau_{1,2}}$ and asymptotically stable for all signals $\sigma\in S_{\tau_{1,2}}'$.
\end{theorem}

\begin{proof}
	i) If $ad\in\mathcal{I}$, then the result follows by Propositions~\ref{rr:thm:ad=0,1} and~\ref{rr:thm:adinI}. \\~\\
	ii) If $ad>q_2/(q_2-p_2)$, then by~\eqref{findingmaxpt}, the unique nonzero solution of the system~\eqref{eq:2} is $(t_{sol},s_{sol})=(\mathcal{O}(\tilde{s}),\tilde{s})$.
	\begin{enumerate}
		\item[a)] For the last two cases in Remark~\ref{zeroset}(1), $\mathcal{O}(\tilde{s})\le \tilde{s}$. Choose $\lambda_1=\sqrt[8]{\frac{b^2d^2}{a^2c^2}\, {\rm{e}}^{-2(q_1-p_1)\mathcal{O}(\tilde{s})}}$. Then for all $s\ge 0$, $f(\lambda_1,\mathcal{O}(\widetilde{s}),s)=k(\mathcal{O}(\widetilde{s}),s)$. Thus, on the line $t=\mathcal{O}(\widetilde{s})$, the function $f(\lambda_1,\mathcal{O}(\widetilde{s}),s)\le 0$, for all $s\in\mathbb{R}$, where the equality holds only at $s=\widetilde{s}$. Since for each fixed $s$, there is at most one root $t$, $f(\lambda_1,t,s)< 0$, for all $s\in\mathbb{R}$ and $t>\mathcal{O}(\widetilde{s})$. Thus, there is no intersection between the zero set of $f(\lambda_1,t,s)$ and the region $\mathcal{R}=\left\{(t,s)\in\mathbb{R}^2\colon t,s> \mathcal{O}(\widetilde{s})\right\}$. Hence $f(\lambda_1,t,s)<0$ on $\mathcal{R}$ and the result follows. 
		\item[b)] For the first case of Remark~\ref{zeroset}(1), $\mathcal{O}(\tilde{s})> \tilde{s}$. Let $t_u$ be the unique positive root of $k(t,t)$, refer Lemma~\ref{rr:direction:k}. Taking $\lambda_1=\sqrt[8]{\frac{b^2d^2}{a^2c^2}\, {\rm{e}}^{-2(q_1-p_1)t_u}}$ and using the argument preceding Lemma~\ref{rr:lemma:f}, $f(\lambda_1,t,s)$ does not intersect with the region $\mathcal{R}=\left\{(t,s)\in\mathbb{R}^2\colon t,s> t_u\right\}$. Hence the result follows.\\
	\end{enumerate}\vspace{-10pt}
	iii) If $ad<-p_2/(q_2-p_2)$, the result follows using similar arguments as above.
\end{proof}

\subsection{Computing $\tau_{2,1}$}	
Here we obtain a result similar to Theorem~\ref{rr:thm12}. The roles of $J_1$ and $J_2$ are interchanged and the matrix $M_{D_1,D_2}$ is replaced by $M_{D_1,D_2}^{-1}$. More precisely, the vector $(a,b,c,d,p_1,q_1,p_2,q_2)$ is replaced by $(d,-b,-c,a,p_2,q_2,p_1,q_1)$ in the statement of Theorem~\ref{rr:thm12}. It is noteworthy that $\tau_{2,1}=0$ when $ad\in \left[-\dfrac{p_1}{q_1-p_1},\dfrac{q_1}{q_1-p_1}\right]$. Thus we have the following result.

\begin{proposition}\label{rr:dt0cond}
	The switched system~\eqref{eq:system} is stable for all signals $\sigma$ if 
	\[
	ad\in \left[-\dfrac{p_2/q_2}{1-p_2/q_2},\dfrac{1}{1-p_2/q_2}\right]\bigcup \left[-\dfrac{p_1/q_1}{1-p_1/q_1},\dfrac{1}{1-p_1/q_1}\right].
	\]
\end{proposition}

\section{Both subsystem matrices $A_1, A_2$ have complex eigenvalues}\label{sec:CC}
	
	Suppose $A_1$ and $A_2$ are planar Hurwitz matrices with both having a pair of complex conjugate eigenvalues. For $i=1,2$, let the Jordan form of $A_i$ be $J_i= \begin{pmatrix} -\alpha_i & \beta_i\\ -\beta_i & -\alpha_i\end{pmatrix}$, where $\alpha_i>0$ and $\beta_i\ne 0$. Since $\det(M)=\pm 1$, $D_1$ and $D_2$ are both identity.
	\subsection{Computing $\tau_{1,2}$} 
	In this section, we will compute $\tau_{1,2}\ge 0$ such that for all $t,s>\tau_{1,2}$, the matrix $K=M_{D_1,D_2}^{-1}{\rm{e}}^{J_2 s}\,M_{D_1,D_2}{\rm{e}}^{J_1 t}$ is Schur stable, see~\eqref{eq:tau}. This will ensure stability of the switched system~\eqref{eq:system} using~\eqref{eq:newform-12}. Recall that the feasible region is the set of values of $t$ for which $|\text{det}(K^TK)|<1$. It is straightforward to check that the first quadrant $Q_1$ is contained in the feasible region. Hence using Schur's stability, $\left\|M^{-1}{\rm{e}}^{J_2 s}\,M {\rm{e}}^{J_1 t} \right\|<1$ if and only if
		\[
		\frac{{\rm e}^{-2(\alpha_2 s+\alpha_1 t)}}{2}\left[\left(a^2+b^2+c^2+d^2\right)^2-\{\left(a^2+b^2+c^2+d^2\right)^2-4\}\cos 2\beta_2 s\right]<1+{\rm{e}}^{-4(\alpha_2 s+\alpha_1 t)},
	\]
	which can be simplified to Schur's function form $f(t,s)<0$, where $f(t,s)$ equals
		\[\left(a^2+b^2+c^2+d^2\right)^2-\left\{\left(a^2+b^2+c^2+d^2\right)^2-4 \right\}\cos 2\beta_2 s-4\cosh\,(2\alpha_1 t+2\alpha_2 s).
	\]
It should be noted that $\left(a^2+b^2+c^2+d^2\right)^2\ge 4$ with equality only when $M=\begin{pmatrix} a & b\\ -b & a\end{pmatrix}$ or $M=\begin{pmatrix} a & b\\ b & -a\end{pmatrix}$.
	
	\begin{remark}\label{cc:observation}
		Let $L(s)=\left(a^2+b^2+c^2+d^2\right)^2-\left\{\left(a^2+b^2+c^2+d^2\right)^2-4 \right\}\cos 2\beta_2 s$. The following observations can be made. 
		\begin{enumerate}
			\item $L(s)$ is a periodic function of period $\pi/\beta_2$.
			\item The zero set $f(t,s)=0$ satisfies: for each $s$, the value $t_s$ such that $f(t_s,s)=0$ is given by the formula \[
				t_s=\frac{-2\alpha_2 s+\cosh^{-1} L(s)}{2\alpha_1}=-\frac{\alpha_2}{\alpha_1}s+\frac{\cosh^{-1} L(s)}{2\alpha_1},
			\]
			where $\cosh^{-1}$ is assumed to return only positive values since we want to characterise the zero set in the first quadrant. Also, it is clear that for large enough $s$, no non-negative solution $t_s$ exists, since the second term in the expression of $t_s$ is bounded. Denote the zero set of $f(t,s)$ by $Z=\left\{(t_s,s)\colon\,s\in\mathbb{R}\right\}\cap Q_1$.
			\item Periodicity of $L$ gives $t_{s+\pi/\beta_2}=t_s-(\alpha_2/\alpha_1)(\pi/\beta_2)$.
			Thus, the zero set $Z$ shifts by $\alpha_2\pi/\alpha_1\beta_2$ units to the left for $\pi/\beta_2$ units distance covered along the $s-$axis.
		\end{enumerate}
	\end{remark}
	
	\begin{theorem}\label{cc:thm12}
		(With notations in this section) Let \begin{eqnarray}
			\tau_{1,2}&=&\frac{1}{\alpha_1+\alpha_2}\cosh^{-1}\left(\frac{a^2+b^2+c^2+d^2}{2}\right).
		\end{eqnarray} 
	Then the switched system~\eqref{eq:system} is stable for all $\sigma\in S_{\tau_{1,2}}$ and asymptotically stable for all signals $\sigma\in S_{\tau_{1,2}}'$.
	\end{theorem}
	
	\begin{proof}
		We will show that the zero set $Z$ of the function $f$ lies below the line $s=-(\alpha_1/\alpha_2)t+c_0$ where $c_0=\frac{1}{\alpha_2}\cosh^{-1}\left(\frac{a^2+b^2+c^2+d^2}{2}\right)$. To prove this, we first claim that the zero set $Z$ lies below the line $s=-(\alpha_1/\alpha_2)t+c_0$, where \[
			c_0=\max_{s\in\left[0,\frac{\pi}{\beta_2}\right]}s+\frac{\alpha_1}{\alpha_2}t_s.
		\]
		As a consequence of Remark~\ref{cc:observation}(1,2), the proof of the claim will follow if we show that $\left\{(t_s,s)\colon\,s\in\left[0,\pi/\beta_2\right]\right\}$ lies below the line $s=-(\alpha_1/\alpha_2)t+c_0$. Thus, we want to show that $t_s\le-(\alpha_2/\alpha_1)s+(\alpha_2/\alpha_1)c_0$ for all $s\in\left[0,\pi/\beta_2\right]$. Note that this inequality is true if and only if $s+(\alpha_1/\alpha_2)t_s\le c_0$ for all $s\in\left[0,\pi/\beta_2\right]$, which is clearly true by the choice of $c_0$ we made. Now, substituting the value of $t_s$ from Remark~\ref{cc:observation}, we have $s+(\alpha_1/\alpha_2)t_s=(1/2\alpha_2)\cosh^{-1} L(s)$ which attains its maximum value at $s_0=\pi/2\beta_2$. Thus, we get the value of $c_0$ equal to the one stated before.
		
		The zero set $Z$ lies below the line $s=-(\alpha_1/\alpha_2)t+c_0$ and does not intersect with the region 
		\[
			\mathcal{R}=\left\{(s,t)\in\mathbb{R}^2\colon\,t,s>\frac{c_0}{\frac{\alpha_1}{\alpha_2}+1}\right\},
		\] refer the figure on the right in Figure~\ref{fig:cc}. Also note that $f(t,s)$ is negative on $\mathcal{R}$ due to continuity and it being negative for large values of $t$ and $s$. Hence
		\[
			\tau_{1,2}=\frac{\alpha_2c_0}{\alpha_1+\alpha_2}=\frac{1}{\alpha_1+\alpha_2}\cosh^{-1}\left(\frac{a^2+b^2+c^2+d^2}{2}\right).
		\]
		\begin{figure}[h!]
			\centering
			\includegraphics[scale=0.45]{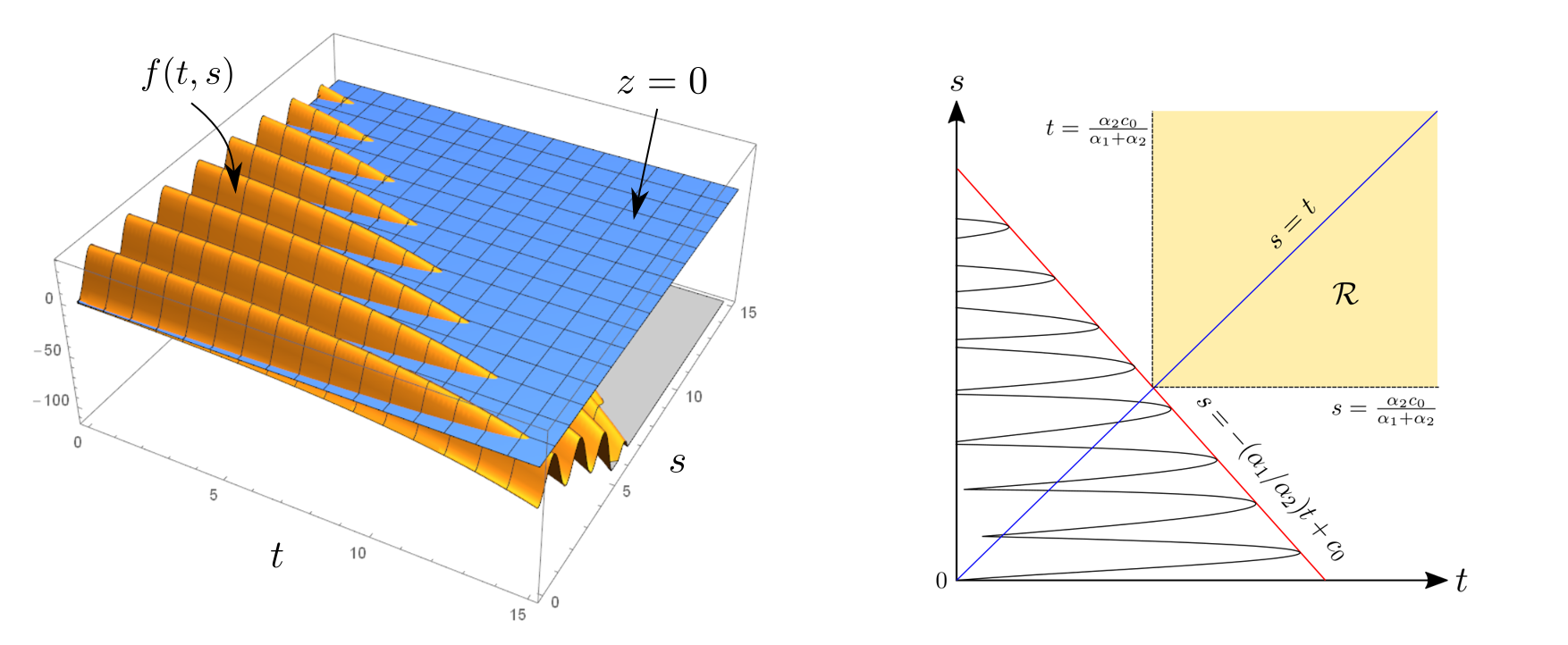}\caption{The zero set of $f$ lies below $s=-(\alpha_1/\alpha_2)t+c_0$. The left graph corresponds to the values $\alpha_1=\alpha_2=0.1$, $\beta_1=2.75$, $\beta_2=1.7$, $a=0.3\sqrt{6}$, $b=0.4\sqrt{6}$, $c=\sqrt{8/3}$ and $d=\sqrt{2/3}$.}\label{fig:cc}
		\end{figure}
	\end{proof}
	
	\subsection{Computing $\tau_{2,1}$} 
	
	Considering the matrix $M {\rm{e}}^{J_1 t} \,M^{-1}{\rm{e}}^{J_2 s}$, we get the following result. 
	
	\begin{theorem}\label{cc:thm21}
		(With notations in this section) Let $\tau_{2,1}=\tau_{1,2}$. The switched system~\eqref{eq:system} is stable for all $\sigma\in S_{\tau_{2,1}}$ and asymptotically stable for all signals $\sigma\in S_{\tau_{2,1}}'$.
	\end{theorem}
		
	\section{$A_1$ is real diagonalizable and $A_2$ has complex eigenvalues}\label{sec:RC}
	
	Suppose $A_1$ is a real diagonalizable matrix with canonical form $J_1=\begin{pmatrix} -p_1 & 0\\ 0 & -q_1\end{pmatrix}$ where $0<p_1<q_1$ and $A_2$ has complex eigenvalues with canonical form $J_2= \begin{pmatrix} -\alpha_2 & \beta_2\\ -\beta_2 & -\alpha_2\end{pmatrix}$ where $\alpha_2>0$ and $\beta_2\ne 0$. Choose the matrices $P_1$ and $P_2$ such that determinant of the transition matrix $M=P_2^{-1}P_1$ has determinant $1$. We vary $D_1$ over matrices of the form $\text{diag}(\lambda_1,1/\lambda_1)$; and $D_2$ is identity.
	
	\subsection{Computing $\tau_{1,2}$} 
	
	\begin{theorem}\label{rc:thm12}
		(With notations in this section) Let $\tau_{1,2}$ be  the unique fixed point of \[
		S(t)=-\left(\frac{q_1+p_1}{2\alpha_2}\right)t+\frac{1}{\alpha_2}\sinh^{-1}\left(\lvert ab+cd\rvert\cosh\left(\frac{q_1-p_1}{2}\right)t+\sinh\left(\frac{q_1-p_1}{2}\right)t\right).
		\] 
	Then the switched system~\eqref{eq:system} is asymptotically stable for all $\sigma\in S_{\tau_{1,2}}$.
	\end{theorem}

\textit{Note}: Unlike other results for computing $\tau_{1,2}$ and $\tau_{2,1}$ where the switched system~\eqref{eq:system} is stable for all $\sigma\in S_{\tau_{1,2}}$ and asymptotically stable for all $\sigma\in S_{\tau_{1,2}}'$, in the above result, the switched system is asymptotically stable for all $\sigma\in S_{\tau_{1,2}}$. This is because the curve given by the function $S$ which bounds the zero set $\mathcal{C}$ of $k$ is strictly above the set $\mathcal{C}$ unless $p_1=q_1$.
	
	\begin{proof}
		We have $\left\|M_{D_1,D_2}^{-1}{\rm{e}}^{J_2 s}\, M_{D_1,D_2} {\rm{e}}^{J_1 t} \right\|<1$ if and only if $f(\lambda_1,t,s)<0$ where $f(\lambda_1,t,s)$ equals \begin{align*}
			 \frac{1}{2}\left(\frac{(b^2+d^2)^2}{\lambda_1^4}{\rm{e}}^{-(q_1-p_1) t}+(a^2+c^2)^2\lambda_1^4 {\rm{e}}^{(q_1-p_1)t}\right)\sin^2\beta_2 s+R \sinh (q_1-p_1) t\, \sin 2\beta_2 s\\
		 +\left(\cos^2\beta_2 s+R^2\sin^2\beta_2 s\right)\, \cosh (q_1-p_1) t-\cosh\left((q_1+p_1) t+2\alpha_2 s\right),
		\end{align*}
		where $R=ab+cd$. Define $k(t,s)=f\left(\sqrt[4]{\frac{b^2+d^2}{a^2+c^2}\, {\rm{e}}^{-(q_1-p_1)t}},t,s\right)$. Then,
		\begin{eqnarray}\label{eq:rc:opt}
		k(t,s)&=&(R^2+1) \sin^2 \beta_2 s+R \sinh (q_1-p_1) t\, \sin 2\beta_2 s\nonumber \\
		& &+\left(\cos^2\beta_2 s+R^2\sin^2\beta_2 s\right)\, \cosh (q_1-p_1) t-\cosh\left((q_1+p_1) t+2\alpha_2 s\right).
		\end{eqnarray}
	
		Let $\mathcal{C}=\{(t,s)\in Q_1\colon\ k(t,s)=0\}$. Let $\Lambda^0=(\pi/\beta_2)\mathbb{N}$. For all $s_0\in\Lambda^0$ and $t>0$, $k(t,s_0)<0$. Thus, the zero set $\mathcal{C}$ can be rewritten as $\mathcal{C}=\{(t,s)\in Q_1\colon\ R=\ell_\pm(t,s)\}$ where \begin{eqnarray*}
			\ell_\pm(t,s)&=&\frac{-\cos\beta_2 s\sinh\left(\frac{q_1-p_1}{2}\right)t\pm\sinh\left(\alpha_2 s+\left(\frac{q_1+p_1}{2}\right)t\right)}{\sin\beta_2 s\cosh\left(\frac{q_1-p_1}{2}\right)t},\ s\notin \Lambda^0.
		\end{eqnarray*}
		The following observations can be made about the functions $\ell_\pm$, refer Figure~\ref{fig:rc12}. Let $\Lambda^1=\cup_{k\in\mathbb{N}}\left((2k-2)\pi/\beta_2,(2k-1)\pi/\beta_2\right)$ and $\Lambda^2=\cup_{k\in\mathbb{N}}\left((2k-1)\pi/\beta_2,2k\pi/\beta_2\right)$\begin{enumerate}
			\item When $s_0\in \Lambda^1$,\begin{enumerate}
				\item $\ell_+(t,s_0)$ is positive and increasing in $t$, and
				\item $\ell_-(t,s_0)$ is negative and decreasing in $t$.
			\end{enumerate}
		\item When $s_0\in \Lambda^2$,\begin{enumerate}
			\item $\ell_+(t,s_0)$ is negative and decreasing in $t$, and
			\item $\ell_-(t,s_0)$ is positive and increasing in $t$.
		\end{enumerate}
		\end{enumerate}
	
			\begin{figure}[h!]
		\centering
		\captionsetup{justification=centering}
		\includegraphics[scale=0.7]{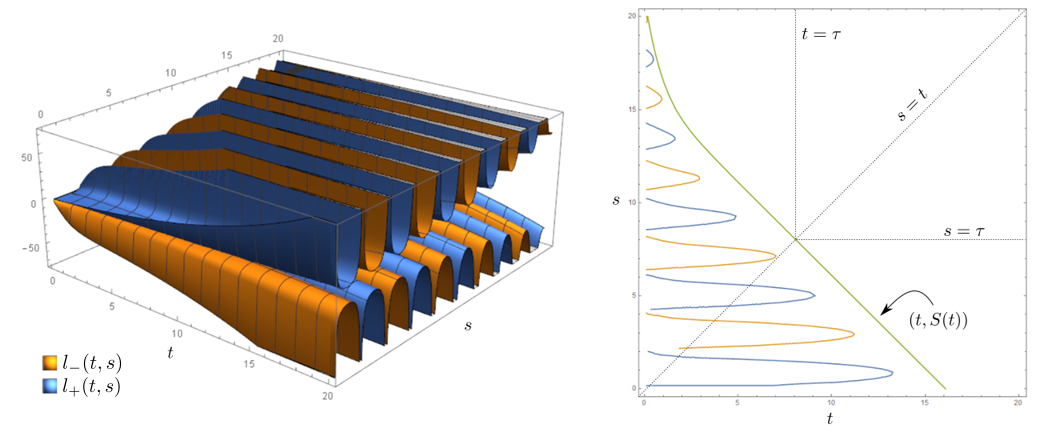}\caption{Plots of $\ell_\pm(t,s)$ and $\mathcal{C}$ when $R>0$. Graphs correspond to the values $\alpha_2=0.1$, $\beta_2=1.5$, $p=0.1$, $q=1.2$ and $R=4$.}\label{fig:rc12}
	\end{figure}

	We will now work towards finding a curve lying above $\mathcal{C}$ as in the proof of Theorem~\ref{cc:thm12}. This curve, in this case, will not be a straight line and will help us find an expression for $\tau_{1,2}$.
	 
	We find this curve for the case when $R\ge 0$, the other case will follow similarly. From the above observations, it is clear that for a given positive $s_0\notin \Lambda^0$ and $t>0$, $\ell_+(t,s_0)$ and $\ell_-(t,s_0)$ have opposite signs. Thus, we have $\mathcal{C}=\left\{(t,s)\in Q_1\colon\ \max\{\ell_+(t,s),\ell_-(t,s)\}=R \right\}$. Moreover $R<\max\{\ell_+(t,s),\ell_-(t,s)\}$ if and only if $k(t,s)<0$. Now,
	\begin{eqnarray*}
		\max\{\ell_+(t,s),\ell_-(t,s)\}&=& 
	\left\{\begin{aligned}
		& \frac{\sinh\left(\alpha_2 s+\left(\frac{q_1+p_1}{2}\right)t\right)-\cos\beta_2 s\sinh\left(\frac{q_1-p_1}{2}\right)t}{\lvert\sin\beta_2 s\rvert\cosh\left(\frac{q_1-p_1}{2}\right)t}, & s & \in\Lambda^1\\
		& \frac{\sinh\left(\alpha_2 s+\left(\frac{q_1+p_1}{2}\right)t\right)+\cos\beta_2 s\sinh\left(\frac{q_1-p_1}{2}\right)t}{\lvert\sin\beta_2 s\rvert\cosh\left(\frac{q_1-p_1}{2}\right)t}, & s & \in\Lambda^2
	\end{aligned}\right.\\
		&\ge& \frac{\sinh\left(\alpha_2 s+\left(\frac{q_1+p_1}{2}\right)t\right)-\sinh\left(\frac{q_1-p_1}{2}\right)t}{\cosh\left(\frac{q_1-p_1}{2}\right)t}.
	\end{eqnarray*}
	The last inequality follows since $\cos\beta_2s$ takes all values in the interval $(-1,1)$ and $\lvert\sin\beta_2 s\rvert$ takes all the values in the interval $(0,1]$, for $s\in\Lambda^1\cup \Lambda^2$. Moreover, the inequality is strict for $t\ne 0$ since $q_1\ne p_1$.\\	
	Hence $\mathcal{C}$ lies below the curve 
	\[
	R\cosh\left(\frac{q_1-p_1}{2}\right)t+\sinh\left(\frac{q_1-p_1}{2}\right)t=\sinh\left(\alpha_2 s+\left(\frac{q_1+p_1}{2}\right)t\right).
	\]
	
	When $R<0$, $k(t,s)<0$ if and only if $R>\min \{\ell_+(t,s),\ell_-(t,s)\}$. Hence it can be shown $\mathcal{C}$ lies below the curve
	\[
	\lvert R\rvert\cosh\left(\frac{q_1-p_1}{2}\right)t+\sinh\left(\frac{q_1-p_1}{2}\right)t=\sinh\left(\alpha_2 s+\left(\frac{q_1+p_1}{2}\right)t\right).
	\]
	Thus, for a general $R$, the equation of the curve bounding $\mathcal{C}$ obtained above is given by \[
	S(t)=-\left(\frac{q_1+p_1}{2\alpha_2}\right)t+\frac{1}{\alpha_2}\sinh^{-1}\left(\lvert R\rvert\cosh\left(\frac{q_1-p_1}{2}\right)t+\sinh\left(\frac{q_1-p_1}{2}\right)t\right).
	\]
 	A straightforward calculation shows that $S$ is a decreasing and concave up function. Thus, if $\tau$ is the unique fixed point of $S$, $k(t,s)\le 0$ for all $t,s\ge \tau$ with the possible equality only at $(\tau,\tau)$. Moreover choosing $\lambda_1=\sqrt[4]{\frac{b^2+d^2}{a^2+c^2}\, {\rm{e}}^{-(q_1-p_1)\tau}}$, as in the proof of Theorem~\ref{rr:thm12}, we have $f(\lambda_1,t,s)<0$ in the region $\mathcal{R}_\tau=\{(t,s)\in Q_1\colon\ t,s\ge \tau\}$. This implies that for some choice $D_1$ corresponding to $\lambda_1$, we have $\left\|M_{D_1,D_2}^{-1}{\rm{e}}^{J_2 s}\, M_{D_1,D_2} {\rm{e}}^{J_1 t} \right\|<\rho$ for all $t,s\ge \tau$, for some $0<\rho<1$. Hence, the result.
\end{proof}

	\subsection{Computing $\tau_{2,1}$}\label{sub:rc21}
	
	We have $\left\|M_{D_1,D_2} {\rm{e}}^{J_1 t}\, M_{D_1,D_2}^{-1}{\rm{e}}^{J_2 s} \right\|<1$ if and only if $f(t,s)<0$, where
	\begin{eqnarray*}
		f(t,s)=\left(a^2+c^2\right)\left(b^2+d^2\right)\sinh^2\left(\frac{q_1-p_1}{2}\right)t-\sinh^2\left(\left(\frac{q_1+p_1}{2}\right)t+\alpha_2 s\right)<0.
	\end{eqnarray*}
	
	\begin{remark}\label{rem:sec5}
		The following observations can be made. Let $K=\left(a^2+c^2\right)\left(b^2+d^2\right)$, which is at least 1. 
		\begin{enumerate}
			\item For $s=0$, $f(t,0)<0$ if and only if \begin{eqnarray*}
				\sqrt{K}&<&\frac{\sinh\left(\left(q_1+p_1\right)t/2\right)}{\sinh\left(\left(q_1-p_1\right)t/2\right)},\ t\ne 0.
			\end{eqnarray*}
			Since the function on the right attains the minimum value $(q_1+p_1)/(q_1-p_1)$ at $t=0$, and is increasing in $t$, the function $f(t,0)$ has at most two roots. Firstly $t=0$ is a root. If $K\le (q_1+p_1)/(q_1-p_1)$, $f(t,0)<0$ for all $t>0$. Else if $K> (q_1+p_1)/(q_1-p_1)$, there exists a unique $\tilde{t}>0$ such that $f(\tilde{t},0)=0$.
			\item For each $t$, there is a unique $s_t$ such that $f(t,s_t)=0$ with \begin{eqnarray}\label{zerosetcr}
				\alpha_2\, s_t=-\left(\frac{q_1+p_1}{2}\right)t+\left|\sinh^{-1}\left(\sqrt{K}\sinh\left(\frac{q_1-p_1}{2}\right)t\right)\right|.
			\end{eqnarray}
			Since $f_s(t,s)<0$, the existence of $s_t\ge 0$ for each $t\in[0,\tilde{t}\,]$ is guaranteed. 
		\end{enumerate}
	\end{remark}
	
	\begin{theorem}\label{rc:thm21}
		(With notations in this section) \\
		i) If $1\le K\le\left((q_1+p_1)/(q_1-p_1)\right)^2$, the switched system~\eqref{eq:system} is stable for all signals $\sigma$. Let $\tau_{2,1}=0$. \\
		ii) If $K>\left((q_1+p_1)/(q_1-p_1)\right)^2$, let $t=t_0\ge 0$ be the unique solution of 
		\[
			\tanh^2\left(\frac{q_1-p_1}{2}\right)t=\frac{1}{K-1}\left(K\left(\frac{q_1-p_1}{q_1+p_1}\right)^2-1 \right),
		\]
		and		
		\[
			S(t_0)=-\left(\frac{q_1+p_1}{2\alpha_2}\right)t_0+\frac{1}{\alpha_2}\sinh^{-1}\sqrt{\frac{K(q_1-p_1)^2-(q_1+p_1)^2}{(q_1+p_1)^2-(q_1-p_1)^2}}.
		\]
		\begin{enumerate}[(a)]
			\item If $S(t_0)\le t_0$, let $\tau_{2,1}=S(t_0)$.
			\item If $S(t_0)> t_0$, let $\tau_{2,1}$ be the unique positive solution $t$ of \[
			K\sinh^2\left(\frac{q_1-p_1}{2}\right)t=\sinh^2\left(\frac{q_1+p_1+2\alpha_2}{2}\right)t.
			\]
		\end{enumerate} 
		Then the switched system~\eqref{eq:system} is stable for all $\sigma\in S_{\tau_{2,1}}$ and asymptotically stable for all $\sigma\in S_{\tau_{2,1}}'$.
	\end{theorem}
	
	\begin{proof}
		If $K=1$,~\eqref{zerosetcr} becomes $2\alpha_2 s_t=-2p_1t$, and hence there is no positive solution $s_t$. \\
		Let us denote the expression on the right hand side of~\eqref{zerosetcr} as the function $g(t)$.\\ 
		If $K> 1$, we will prove that the zero set, $Z=\{(t,s)\colon\ f(t,s)=0\}\cap Q_1$, of Schur's function $f(t,s)$ takes one of the three forms as shown in Figure~\ref{fig:rc21}. Now $g'(t)<0$ if and only if \[
			\tanh^2\left(\frac{q_1-p_1}{2}\right)t > \frac{1}{K-1}\left(K\left(\frac{q_1-p_1}{q_1+p_1}\right)^2-1 \right).
		\]
		If $1<K\le\left((q_1+p_1)/(q_1-p_1)\right)^2$, the above inequality holds true for all $t>0$. Thus $g'(t)<0$ for all $t>0$. Since $g(0)=0$, $g(t)<0$ for all $t>0$. \\		
		If $K>\left((q_1+p_1)/(q_1-p_1)\right)^2$, $\widetilde{t}>0$ by Remark~\ref{rem:sec5}. Also there exists a unique $t_0>0$ such that \begin{eqnarray*}
			\tanh^2\left(\frac{q_1-p_1}{2}\right)t_0&=&\frac{1}{K-1}\left(K\left(\frac{q_1-p_1}{q_1+p_1}\right)^2-1 \right).
		\end{eqnarray*}
		Thus the function $g(t)$ increases on $[0,t_0)$ and decreases on $(t_0,\infty)$.Hence, $S(t)=s_t=(1/\alpha_2)g(t)$ increases on $[0,t_0)$ and then decreases to zero as $t$ approaches $\tilde{t}$. Thus, $Z=\{\left(t,S(t)\right)\colon\  t\in(0,\tilde{t})\}$. Since $S''(t)<0$ for all $t>0$, the graph of the zero set of $f(t,s)$ is concave down and can be classified into the following three cases. The first two cases correspond to when $S(t_0)\le t_0$. It is clear that in these cases, the zero set of $f(t,s)$ does not intersect with the region $\mathcal{R}_{S(t_0)}=\{(t,s)\in\mathbb{R}^2\colon\, s,t>S(t_0)\}$.
		
		\begin{figure}[h!]
			\centering
			\includegraphics[scale=0.25]{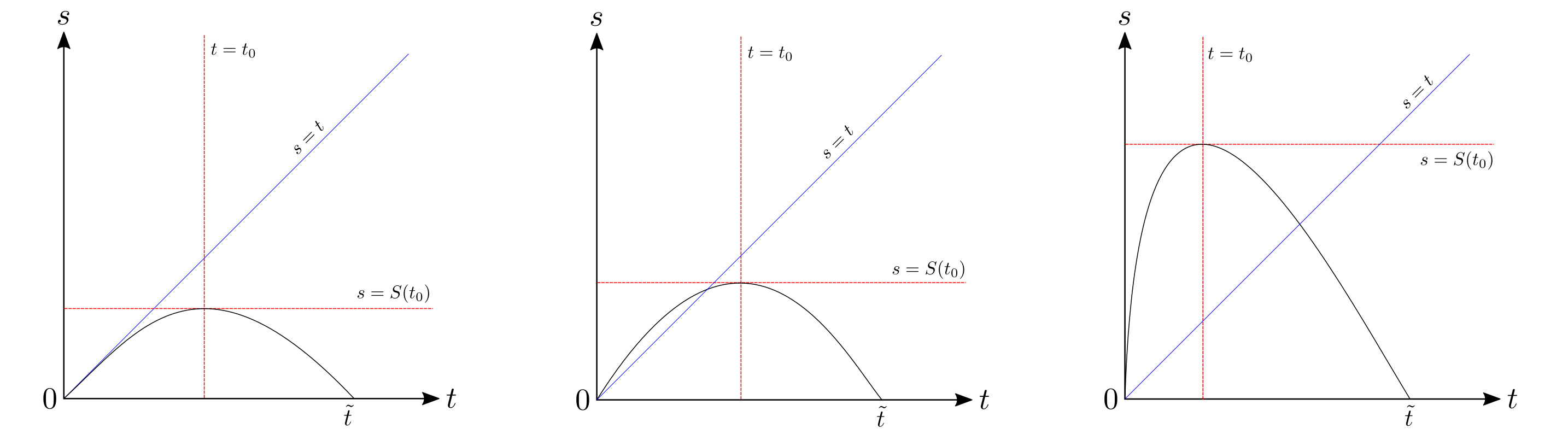}\caption{Zero set of $f(t,s)$: the first two correspond to $S(t_0)\le t_0$ and the right most corresponds to $S(t_0)> t_0$.}\label{fig:rc21}
		\end{figure}
		
		The third and the final case corresponds to $S(t_0)>t_0$. There is a unique positive root of $f$ in the direction $s=t$, call it $\tau$. Then the zero set does not intersect with the region $\mathcal{R}_\tau=\{(t,s)\in\mathbb{R}^2\colon\, s,t>\tau\}$ and hence $\tau_{1,2}=\tau$. Moreover, $\tau$ is the unique positive root of $f(t,t)=0$ and hence the result follows.
	\end{proof}
	
	\begin{remark}
		Using notations in the section, $\tau_{2,1}\le\tau_{1,2}$. To observe this, we consider the functions \begin{eqnarray*}
		S_{1,2}(t)&=&-\left(\frac{q_1+p_1}{2\alpha_2}\right)t+\frac{1}{\alpha_2}\sinh^{-1}\left(\lvert ab+cd\rvert\cosh\left(\frac{q_1-p_1}{2}\right)t+\sinh\left(\frac{q_1-p_1}{2}\right)t\right),\\
		 S_{2,1}(t)&=&-\left(\frac{q_1+p_1}{2\alpha_2}\right)t+\frac{1}{\alpha_2}\sinh^{-1}\left(\sqrt{\left(a^2+c^2\right)\left(b^2+d^2\right)}\sinh\left(\frac{q_1-p_1}{2}\right)t\right),
		\end{eqnarray*}
	where $S_{1,2}$ is the curve bounding the zero set in Theorem~\ref{rc:thm12} and $S_{2,1}$ is the parametrization of the zero set in the feasible region in Theorem~\ref{rc:thm21}. Now, $S_{2,1}(t)< S_{1,2}(t)$ if and only if \[
	\tanh\left(\frac{q_1-p_1}{2}\right)t< \frac{1}{\lvert ab+cd\rvert}+\sqrt{1+\frac{1}{(ab+cd)^2}}.
	\]
	The above equivalence is proved using $\left(a^2+c^2\right)\left(b^2+d^2\right)-1=(ab+cd)^2$, and is always true since $\tanh$ is bounded above by $1$. This implies that, $\left\|M_{D_1,D_2} {\rm{e}}^{J_1 t}\, M_{D_1,D_2}^{-1}{\rm{e}}^{J_2 s} \right\|<1$, for all $t,s> \tau_{1,2}$. Since $\xi=\tau_{2,1}$ is the smallest value satisfying $\left\|M_{D_1,D_2} {\rm{e}}^{J_1 t}\, M_{D_1,D_2}^{-1}{\rm{e}}^{J_2 s} \right\|<1$, for all $t,s>\xi$, we have $\tau_{1,2}\ge \tau_{2,1}$.
	\end{remark}
	\section{Both $A_1$ and $A_2$ are defective}\label{sec:NN}
	For $i=1,2$, let $A_i$ be a defective matrix with canonical form $J_i=\begin{pmatrix}
		-n_i & 1\\0 & -n_i
	\end{pmatrix}$, for some $n_i>0$. 
	
	\subsection{Computing $\tau_{2,1}$} \label{sec:nnt21}
	Note that
		\begin{eqnarray*}
		\left\|M {\rm{e}}^{J_1 t}\, M^{-1}{\rm{e}}^{J_2 s}\right\| &\le & \left\|M {\rm{e}}^{J_1 t}\, M^{-1}{\rm{e}}^{-n_2 s} \right\|\theta(s),
	\end{eqnarray*}
	where $\theta(s)=\begin{Vmatrix}\begin{pmatrix}
		1 & s\\ 0 & 1
		\end{pmatrix}
	\end{Vmatrix}=\sqrt{1+\frac{s^2}{2}+s\sqrt{1+\frac{s^2}{4}}}$. \\
Observe that $\left\|M {\rm{e}}^{J_1 t}\, M^{-1}{\rm{e}}^{-n_2 s} \right\|\theta(s)<\nobreak 1$ if and only if 
		\begin{itemize}
		\item[(C1)] $(t,s)$ lies in the feasible region $n_2 s-\ln\theta (s)+n_1 t>0$, and
		\item[(C2)] $(t,s)$ satisfies the inequality
		\begin{eqnarray*}
			2+(a^2+c^2)^2\, t^2\,\theta(s)^2 {\rm e}^{-2 (n_2 s + n_1 t)}<1+\theta(s)^4{\rm e}^{-4 (n_2 s + n_1 t)},
		\end{eqnarray*}
		which can be simplified to the form $f(t,s)<0$, where \begin{eqnarray}\label{nn:fts}
			f(t,s)&=&\left(\frac{a^2+c^2}{2}\right) t-\left\lvert\,\sinh \left(n_1t+n_2 s-\ln\theta(s)\,\right)\,\right\rvert.
		\end{eqnarray}
		\end{itemize}
	
	The goal to find the smallest $\tau=\tau_{1,2}\ge 0$ such that both (C1) and (C2) are satisfied in the region $\mathcal{R}_\tau=\{(t,s)\colon\, t,s>\tau\}$. We will find the smallest $\tau\ge 0$ such that~\eqref{nn:fts} is satisfied in the region $\mathcal{R}_\tau$, and then observe that the first inequality $n_2 s-\ln\theta (s)+n_1 t>0$ is also satisfied in the same region.
	
	\begin{remark}\label{funcgtheta}
		Denote $g(s)=n_2 s-\ln \theta(s)$. We can make the following observations.\begin{enumerate}
			\item $g''(s)> 0$, for all $s>0$.
			\item When $n_2<1/2$, the function $g$ decreases on $(0,s_0)$, increases on $(s_0,\infty)$ and is concave up, where  $s_0=\sqrt{(1/n_2^2)-4}$. The function attains minimum value
			\begin{eqnarray*}
				g(s_0)=\sqrt{1-4n_2^2}-\ln\sqrt{\frac{1}{2n_2^2}-1+\frac{\sqrt{1-4n_2^2}}{2n_2^2}}=-g_{min}<0.
			\end{eqnarray*}
		 	The function $g$ has a unique positive root, $\tilde{s}>s_0$.
			\item When $n_2\ge 1/2$, $g$ is strictly increasing, hence monotone. Thus, $g^{-1}\colon [0,\infty)\to[0,\infty)$ exists and is strictly increasing as well.
		\end{enumerate}
	The possible graphs of the function $g$ are shown in Figure~\ref{fig:g}.
			\begin{figure}[h!]
				\centering
				\includegraphics[width=.8\textwidth]{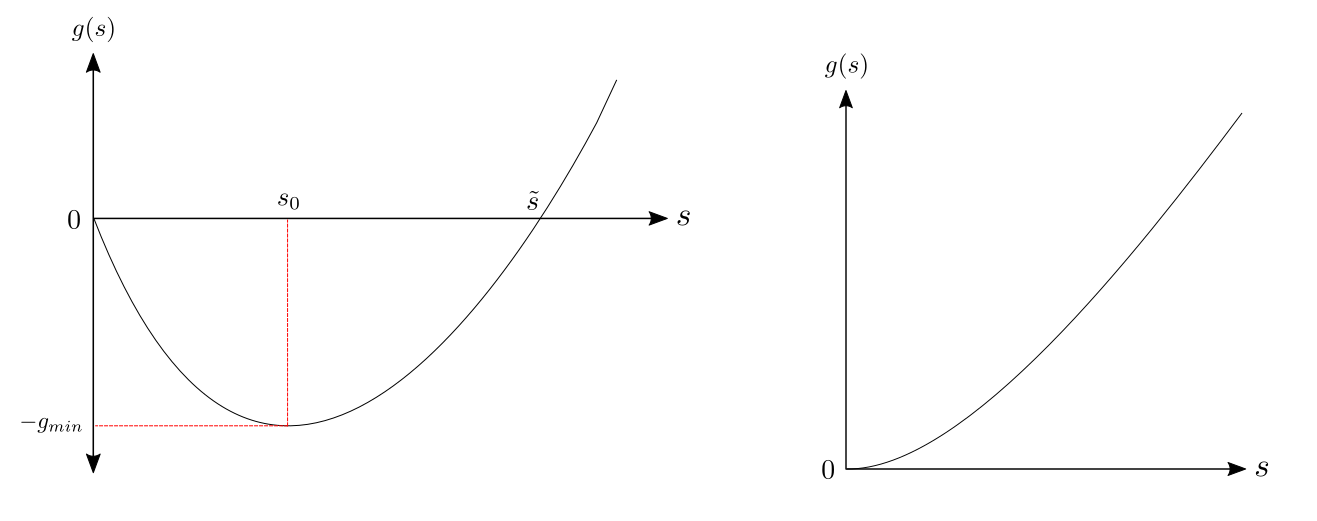}
				\caption{Possible graphs of the function $g$ for $n_2<1/2$ (left) and $n_2\ge 1/2$ (right).}
				\label{fig:g}
			\end{figure}
	\end{remark}
	
	We now present two main results, for the two cases $n_2\ge 1/2$ and $n_2<1/2$ separately.
	
	\begin{theorem}\label{nn:thm21>1/2}
	(With notations in this section) Suppose $n_2\ge 1/2$. Let $K=\left(a^2+c^2\right)/2$. \\
		i) The switched system~\eqref{eq:system} is stable for all signals $\sigma$ if $K\le n_1$. Let $\tau_{2,1}=0$.\\
		ii) If $K>n_1$, let $t_0=\sqrt{\frac{1}{n_1^2}-\frac{1}{K^2}}$ and let $s=S(t_0)$ be the unique positive solution of
	\begin{eqnarray*}
			n_2 s-\ln \sqrt{1+\frac{s^2}{2}+s\sqrt{1+\frac{s^2}{4}}} &=&-n_1\, t_0 + \sinh^{-1}Kt_0.
	\end{eqnarray*}	
\begin{enumerate}[(a)]
		\item If $S(t_0)\le t_0$, let $\tau_{2,1}=S(t_0)$.
		\item If $S(t_0)\ge t_0$, let $\tau_{2,1}$ be the unique positive solution of 
		\begin{eqnarray*}
			n_2 t-\ln \sqrt{1+\frac{t^2}{2}+t\sqrt{1+\frac{t^2}{4}}} &=&-n_1\, t + \sinh^{-1}Kt.
		\end{eqnarray*}
	\end{enumerate}
	Then the switched system~\eqref{eq:system} is stable for all $\sigma\in S_{\tau_{2,1}}$ and asymptotically stable for all $\sigma\in S_{\tau_{2,1}}'$. Moreover when $c\ne 0$, a lower $\tau_{2,1}$ can be obtained using $K_{opt}=c^2/2$ instead of $K$ as defined above.
	\end{theorem}
	
	\begin{proof}
		Note that the equation $n_2 s-\ln\theta (s)+n_1 t>0$ is satisfied for all $t,s>0$ since $n_2\ge 1/2$. Thus, we just need to study the function $f(t,s)$. Observe that $f(0,0)=0$ and $f(t,0)=0$ if and only if $\sinh n_1t=Kt$. Moreover $t^{-1}\sinh n_1t$ is increasing function of $t$ and attains minimum value $n_1$ at $t=0$. When $K\le n_1$, $f(t,0)<0$, for all $t>0$. Since $n_2\ge 1/2$, $g$ is an increasing function and hence $f_s(t,s)<0$. Therefore $f(t,s)<0$ for all $t,s>0$ and the result follows.\\
		When $K>n_1$, there exists a unique $\tilde{t}>0$ such that $f(\tilde{t},0)=0$. Also, $f(t,0)>0$ for $t\in(0,\tilde{t})$, and $f(t,0)<0$ for $t\in(\tilde{t},\infty)$. For each $t\in[0,\tilde{t}\,]$, there exists a unique $s_t\ge 0$ such that $f(t,s_t)=0$ and it is given by
		\[
			s_t=g^{-1}\left(-n_1\, t + \sinh^{-1}Kt\right)=g^{-1}(\xi(t)).
		\]
		Note that $\xi'(t)=-n_1+K\left(K^2t^2+1\right)^{-1/2}$. Since $K> n_1$, $\lim_{t\to 0}\,\xi'(t)>0$ and $\xi''(t)<0$ for all $t>0$, hence $\xi$ increases initially and is decreasing later. The function $g^{-1}$ is differentiable and monotonically increasing, hence the differentiable function $S\colon [0,\tilde{t}\,]\to\mathbb{R}$ defined as $S(t)=s_t$ is increasing if and only if $\xi$ is increasing. Thus, $S$ achieves its maximum value at $t_0$ satisfying $\xi'(t_0)=0$, that is, $t_0=\sqrt{(1/n_1^2)-(1/K^2)}$. Also, $S(t)$ parameterizes the zero set $Z=\{(t,s)\in\mathbb{R}^2\colon f(t,s)=0\}\cap Q_1$.
		
		Hence the zero set $Z$ can be classified into three cases as shown in Figure~\ref{fig:rc21}. The first two cases correspond to when $S(t_0)\le t_0$. It is clear that in these two cases, $Z$ does not intersect with the region $\mathcal{R}_{S(t_0)}=\{(t,s)\in\mathbb{R}^2\colon\, s,t>S(t_0)\}$ and hence the result follows.
		
		The third case is when $S(t_0)>t_0$. There is a unique zero of $f(t,s)$ in the direction $s=t$ since $S$ is concave down, denote this zero by $\tau$. Then the zero set $Z$ does not intersect with the region $\mathcal{R}_\tau=\{(t,s)\in\mathbb{R}^2\colon\, s,t>\tau\}$ and the result follows. Moreover, $\tau$ can be computed as the unique positive solution of the equation $\sinh \left(n_1t+g(t)\right)=Kt$. \\
		The case when $c\ne 0$ will be discussed later in Remark~\ref{nn:optimal}.
	\end{proof}
	
	The above result is applicable when at least one subsystem has eigenvalue lying in the interval $(-\infty,-1/2]$, by making suitable relabeling of matrices $A_1,A_2$. The following result is applicable when at least one of the subsystems has eigenvalue in the interval $(-1/2,0)$.
	
	\begin{theorem}\label{nn:thm21<1/2}
		(With notations in this section) Suppose $n_2< 1/2$. Let $K=\left(a^2+c^2\right)/2$ and $L=\sqrt{1-4n_2^2}-\ln\sqrt{\frac{1}{2n_2^2}-1+\frac{\sqrt{1-4n_2^2}}{2n_2^2}}$.\\
		i) If $K\le n_1$, there exists a unique pair $(t_{max},s_{max})$ satisfying the following equation for $t=t_{max}$ and $s=s_{max}$ 
		\begin{eqnarray*}
			-n_1t+\sinh^{-1}Kt&=&n_2s-\ln \sqrt{1+\frac{s^2}{2}+s\sqrt{1+\frac{s^2}{4}}}=L.
		\end{eqnarray*}
	\begin{enumerate}[(a)]
		\item If $s_{max}\ge t_{max}$, let $\tau_{2,1}=t_{max}$. 
		\item If $s_{max}<t_{max}$, let $\tau_{2,1}$ be the unique positive solution of 
		\begin{eqnarray*}
			-n_1t+\sinh^{-1}Kt&=&n_2t-\ln \sqrt{1+\frac{t^2}{2}+t\sqrt{1+\frac{t^2}{4}}}.
		\end{eqnarray*}
	\end{enumerate}
		ii) If $K>n$, let $\tau_{2,1}$ be the unique positive solution of the above equation.\\
		Then the switched system~\eqref{eq:system} is stable for all $\sigma\in S_{\tau_{2,1}}$ and asymptotically stable for all $\sigma\in S_{\tau_{2,1}}'$. Moreover, when $c\ne 0$, a lower $\tau_{2,1}$ can be obtained using $K_{opt}=c^2/2$ instead of $K$ as defined above.
	\end{theorem}
	
	\begin{proof}
		We split the proof into two parts: when $K\le n_1$ and when $K>n_1$, as these two cases characterize the zero set $Z$ of the function $f(t,s)$. Using the graph of zero set $Z$, we will derive expression for $\tau_{2,1}$ in each of these cases.\\~\\		
		\textit{Case 1:} $K\le n_1$\\
		We will prove that the zero set $Z$ looks like as in Figure~\ref{nn:case1}. Note that $f(t,0)<0$ for all $t>0$. Also for some fixed value of $t$, there might be more than one solutions $s_t$ such that $f(t,s_t)=0$. This happens because for $n_2<1/2$, the function $g(s)$ and hence the expression $n_1t+g(s)$ can assume negative values. Thus, the zero set $Z$ is described by set of equations
		\begin{eqnarray}\label{zerochar2}
			g(s)&=&F_\pm(t),
		\end{eqnarray}
	where $F_\pm(t)=-n_1t\pm\sinh^{-1}Kt$.\\
		As shown in Figure~\ref{fplusandfminusm}, it is easy to check that $F_+(0)=0=F_-(0)$, $F_+$ is concave down and decreasing, and $F_-$ is concave up and decreasing. Also $F_-(t)<F_+(t)<0$, for all $t>0$.\\
			\begin{figure}[h!]
			\centering
			\includegraphics[scale=0.5]{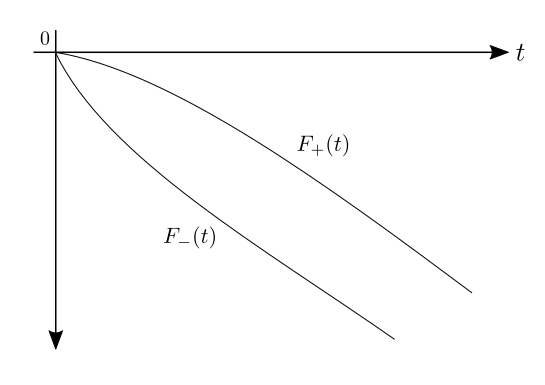}\hspace{50pt}
			\includegraphics[scale=0.5]{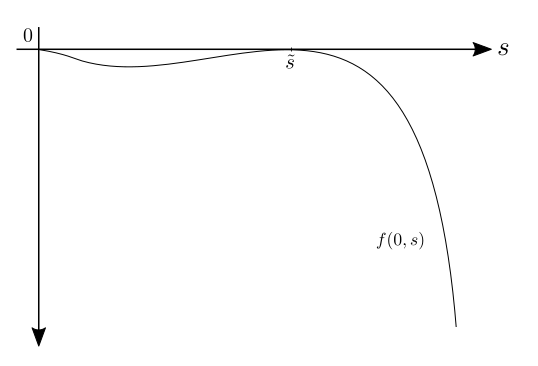}
			\caption{Graphs of $F_+$, $F_-$ and $f(0,s)$}\label{fplusandfminusm}
		\end{figure}
	
		Now, $f(0,s)=0$ if and only it $g(s)=0$ if and only if $s\in\{0,\tilde{s}\}$ ($\tilde{s}$ as in Remark~\ref{funcgtheta}). Also, for $s\in\{0,\tilde{s}\}$, $t=0$ is the only root of $f(t,s)$ since $F_+$ and $F_-$ are both decreasing. For the same reason, the zero set $Z$ does not intersect with the region $\mathcal{R}=\{(t,s)\in\mathbb{R}^2\colon\, s>\tilde{s}\}$. For each $s\in[0,\tilde{s}]$, $g(s)\le 0$ implying the existence of unique $t_s^\pm\ge 0$ such that $g(s)=F_\pm(t_s^\pm)$. Note that $t_s^-<t_s^+$.
		
		Now, for each $s\in(0,\tilde{s})$, $g(s)-F_+(t_s^+)=0$ and $F_+'(t_s^+)<0$. Therefore by the implicit function theorem, the zero set arising from equation $g(s)=F_+(t)$ in~\eqref{zerochar2} can be parametrized by the function $Z^+\colon\,[0,\tilde{s}]\to\mathbb{R}$ defined as $Z^+(s)=t_s^+$. Similarly, the zero set arising from equation $g(s)=F_+(t)$ in~\eqref{zerochar2} can be parametrized by the function $Z^-\colon\,[0,\tilde{s}]\to\mathbb{R}$ given by $Z^-(s)=t_s^-$. Thus, the zero set $Z$ of $f(t,s)$ is union of the graphs of $Z^+$ and $Z^-$.
		
		We claim next that for every $t$, there are at most two choices of $s$ such that $F_+(t)=g(s)$. It follows from the observations below:
		\begin{enumerate}
			\item if $F_+(t)\in(-g_{min},0]$, there are exactly two positive values of $s$ such that $F_+(t)=g(s)$,
			\item if $F_+(t)=-g_{min}$, there is a unique positive value of $s$ such that $F_+(t)=g(s)$, and
			\item if $F_+(t)<-g_{min}$, then there is no positive value of $s$ such that $F_+(t)=g(s)$,
		\end{enumerate}
		where $g_{min}$ is as defined in Remark~\ref{funcgtheta}. 
		
		Similarly, for every $t$, there are at most two positive values of $s$ such that $F_-(t)=g(s)$. Hence both $Z^+$ and $Z^-$ have a unique maxima. Thus, the zero set $Z$ looks like as in Figure~\ref{nn:case1}. Note that the solution curve of $n_2 s-\ln\theta (s)+n_1 t=0$ lies in the region bounded by the curves $(Z^+(s),s)$ and $(Z^-(s),s)$. This is due to the following two observations: any point $(t,s)$ on the solution curve of $n_2 s-\ln\theta (s)+n_1 t=0$ satisfies $f(t,s)\ge 0$ by~\eqref{nn:fts}; and $f$ is non-negative only in the region bounded by the curves $(Z^+(s),s)$ and $(Z^-(s),s)$ (including the boundary).
		\begin{figure}[h!]
			\centering
			\includegraphics[scale=0.4]{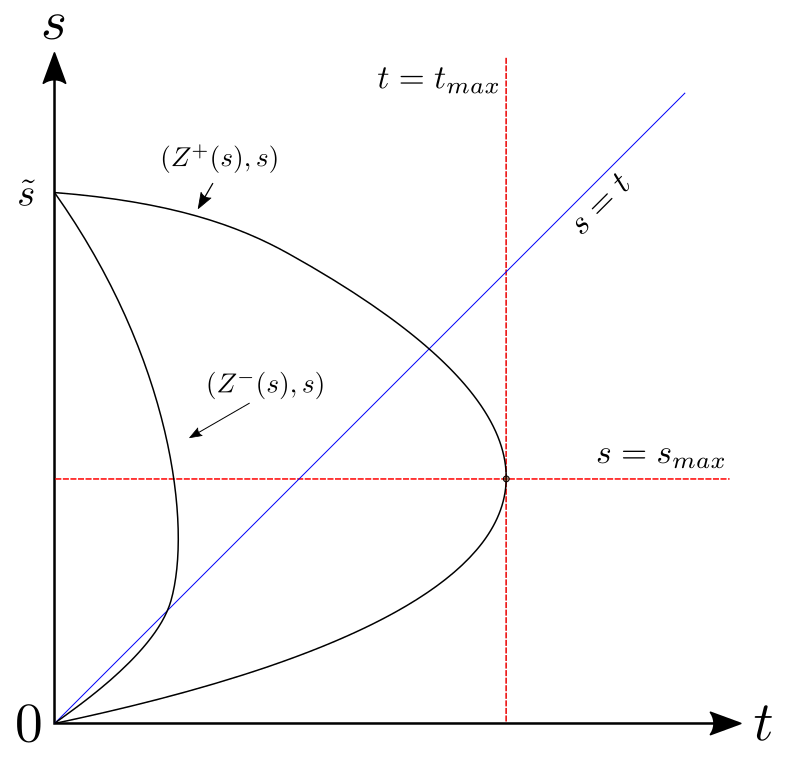}\hfill
			\includegraphics[scale=0.4]{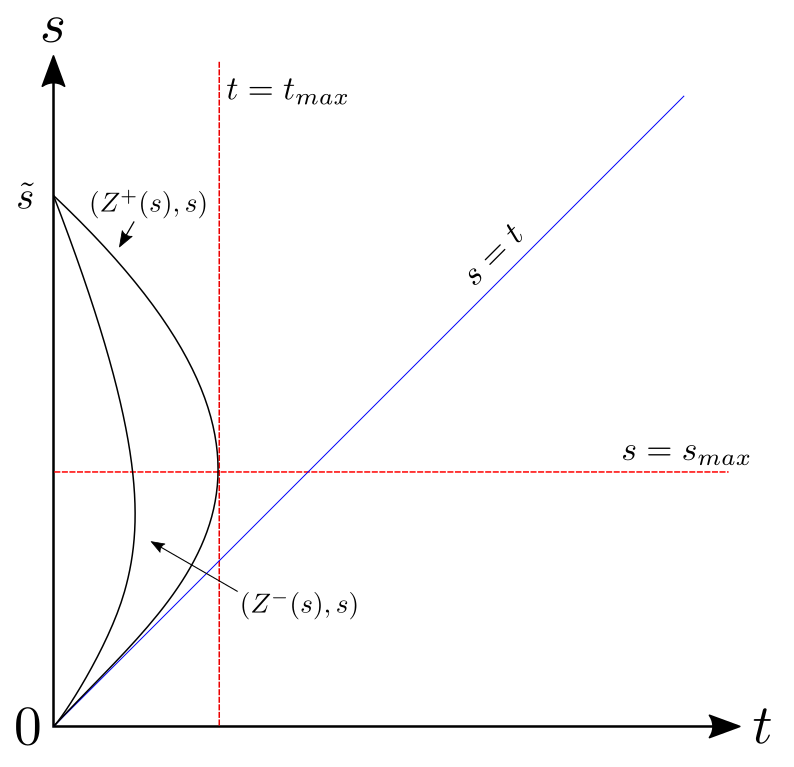}
			\caption{Zero set $Z$ when $K\le n_1$: $t_{max}\ge s_{max}$ (left) and $t_{max}< s_{max}$ (right).}
			\label{nn:case1}
		\end{figure}
		To compute $\tau_{2,1}$, we will find the smallest $\tau$ such that the region $\mathcal{R}_\tau=\{(t,s)\in\mathbb{R}^2\colon\,t,s>\tau\}$ does not intersect the zero set $Z$ of $f(t,s)$. We just need to consider the graph of $Z^+$ for this purpose. Also, $n_2 s-\ln\theta (s)+n_1 t>0$ and $f(t,s)<0$ in this $\mathcal{R}_\tau$ by continuity. The graph of $Z^+$ lies to the left of the line $t=t_{max}$ and intersects it at $(t_{max},s_{max})$ which is given by the unique solutions of the equations $F_+(t_{max})=-g_{min}=g(s_{max})$.
		
		As before, if $s_{max}\ge t_{max}$, $\tau_{2,1}=t_{max}$. On the other hand, if $s_{max}<t_{max}$, $\tau_{2,1}$ is given by the unique solution of $g(t)=F_+(t)$. The uniqueness follows from the fact that the graph of $Z^+$ is concave down since $(F_+^{-1})''\,(t)<0$ for all $t>0$.\\~\\		
		\textit{Case 2:} $K> n_1$\\
		Recall that the zero set $Z$ can be described by~\eqref{zerochar2}. The function $f(t,0)$ has a positive derivative at $t=0$, is concave up initially and then concave down, and hence it has a unique positive zero $\tilde{t}$, that is, $f(\tilde{t},0)=0$. Also, $f(0,s)$ has $\tilde{s}$ as its unique positive root, where $\tilde{s}$ is as in Remark~\ref{funcgtheta}. Moreover, $F_-$ is concave up decreasing function, $F_+$ is concave down, $(F_+)'(0)>0$ and $F_+(t)$ attains maximum (say $F_+^{max}$) at $t=\sqrt{(1/n_1^2)-(1/K^2)}$. We make the following observations for the part of zero set of $f(t,s)$ in $Q_1$ arising from $F_+(t)=g(s)$:
		\begin{figure}[h!]
			\centering
			\includegraphics[scale=0.4]{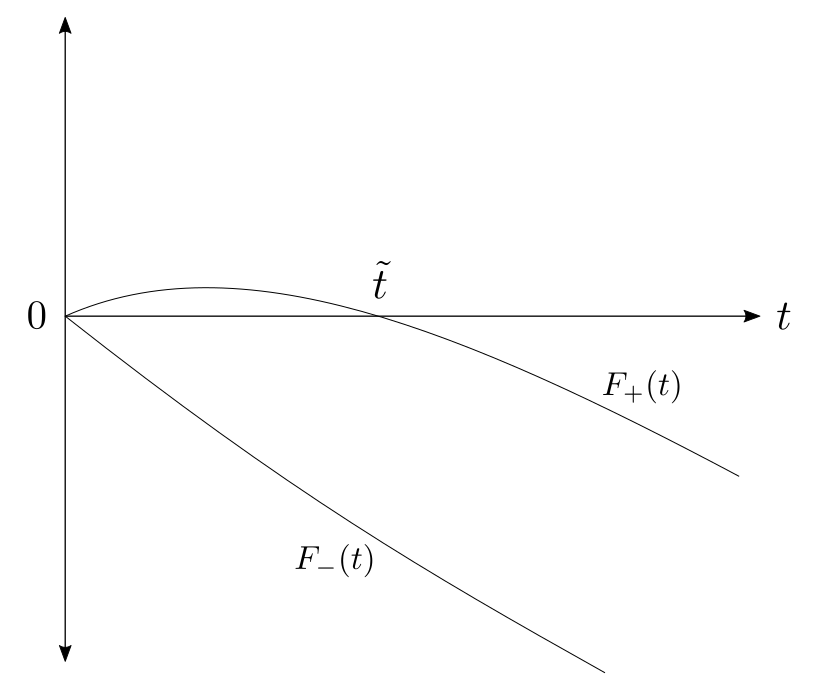}\hfill
			\includegraphics[scale=0.4]{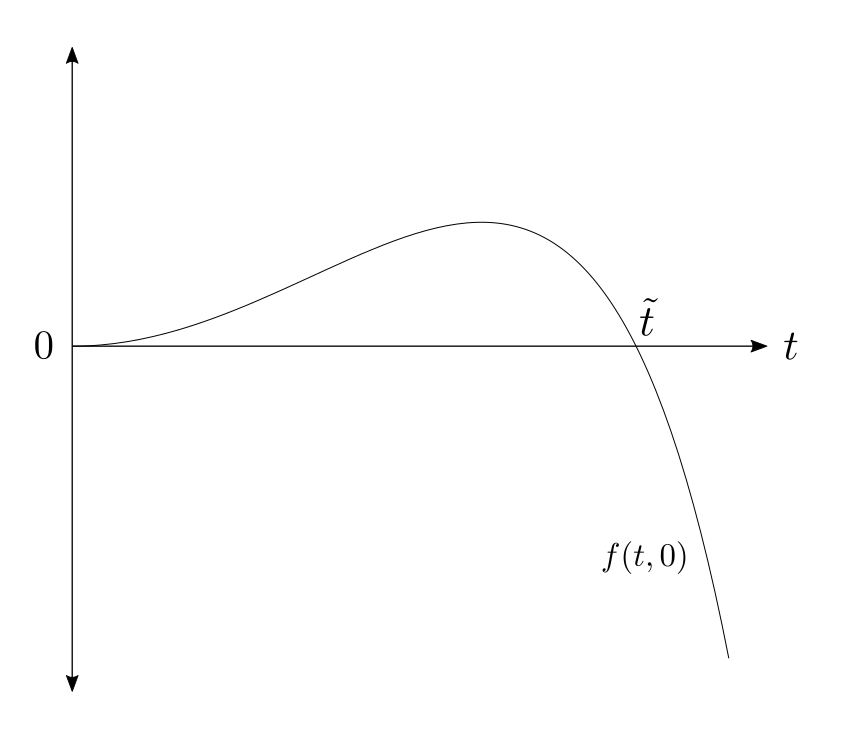}
			\caption{Graphs of $F_+$, $F_-$ and $f(t,0)$.}\label{fpandfmnl1/2}
		\end{figure}
		\begin{enumerate}
			\item For each $0<t<\tilde{t}$, $0=g(0)<F_+(t)$. Since $g$ first decreases and then increases, there exists a unique $s_t$ such that $g(s_t)=F_+(t)$. 
			\item For $t=\tilde{t}$, the only values $s$ satisfying $0=F_+(\tilde{t})=g(s)$ are $0$ and $\tilde{s}$. 
			\item For each $t>\tilde{t}$, there are at most two zeros of $F_+(t)=g(s)$ depending on value $t$ as follows:
			\begin{enumerate}
				\item if $F_+(t)\in(-g_{min},0]$, there are two values of $t$ satisfying  $F_+(t)=g(s)$,
				\item if $F_+(t)=-g_{min}$, there is exactly one value of $t$ satisfying  $F_+(t)=g(s)$, and
				\item if $F_+(t)<-g_{min}$, there is no value of $t$ satisfying  $F_+(t)=g(s)$.
			\end{enumerate}
			\item For each $s<\tilde{s}$, $g(s)<0<F_+(0)$. Since $F_+$ first increases and then decreases, there is a unique value $t_s$ such that $F_+(t_s)=g(s)$.
			\item For $s=\tilde{s}$, the only values $t$ satisfying $F_+(t)=g(\tilde{s})=0$ are $0$ and $\tilde{t}$.
			\item For each $s>\tilde{s}$, $g(s)>0>F_+(0)$ and there are at most two zeros of $F_+(t)=g(s)$ depending on the value of $s$ as follows:
			\begin{enumerate}
				\item if $g(s)>F_+^{max}$, there are no values $t$ such that $F_+(t)=g(s)$,
				\item if $g(s)=F_+^{max}$, there is exactly one value of $t$ satisfying  $F_+(t)=g(s)$, and
				\item if $0<g(s)<F_+^{max}$, there are two values of $t$ satisfying  $F_+(t)=g(s)$.
			\end{enumerate}
		\end{enumerate}
	In addition to the points (1) and (4), the following inequalities are satisfied
	\[
	\frac{\partial}{\partial t} \left(g(s_t)-F_+(t)\right)<0,\ t\in[0,\tilde{t}];\text{ and }\ \frac{\partial}{\partial s} \left(g(s)-F_+(t_s)\right)<0,\ s\in[0,\tilde{s}].
	\]
		
	Thus we can apply the implicit function theorem to obtain the functions $Z_1\colon[0,\tilde{t}]\to\mathbb{R}$ and $Z_2\colon[0,\tilde{s}]\to\mathbb{R}$ defined as $Z_1(t)=s_t$ and $Z_2(s)=t_s$. Moreover, both $Z_1$ and $Z_2$ have a unique maxima as a consequence of the points (3) and (6), and their graphs intersect at $(\tilde{t},\tilde{s})$. Thus the solution curve of $g(s)=F_+(t)$ behaves like Figure~\ref{zerosetnn2}.\\
		\begin{figure}[h!]
			\centering
			\includegraphics[scale=0.4]{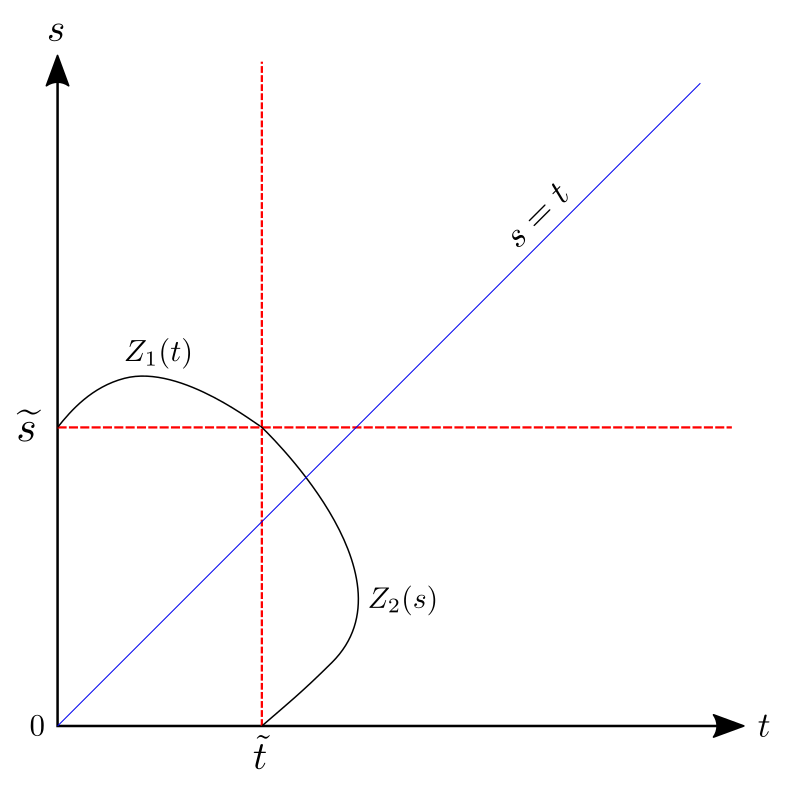}
			\caption{Zero set $Z$ when $K>n_1$.}
			\label{zerosetnn2}
		\end{figure}
	The part of $Z$ arising from $F_+(t)=g(s)$ splits the first quadrant into two regions. As in the previous case, the part of $Z$ arising from the $F_-(t)=g(s)$ lies in the bounded region and so does the curve $n_2 s-\ln\theta (s)+n_1 t=0$. Hence $\tau_{2,1}$ can be computed by studying the curve $F_+(t)=g(s)$. Thus $\tau_{2,1}$ is given by the unique positive solution of the equation $g(t)=F_+(t)$. Uniqueness of the solution is due to the following facts: $g(0)=F_+(0)=0$, $g'(0)<0<F_+'(0)$, $g$ is concave up, and $F_+$ is concave down.
	\end{proof}

	\subsection{Computing $\tau_{1,2}$} 
	
	\begin{theorem}\label{nn:thm12>1/2}
		(With notations in this section) Suppose $n_1\ge 1/2$. Let $R=\left(c^2+d^2\right)/2$. \\
		i) If $R\le n_2$, the switched system~\eqref{eq:system} is stable for all signals $\sigma$. Let $\tau_{1,2}=0$. \\
		ii) If $R>n_2$, let $s_0=\sqrt{\frac{1}{n_2^2}-\frac{1}{R^2}}$ and let $t=T(s_0)$ be given by the unique solution of 
		\begin{eqnarray*}
			n_1 t-\ln \sqrt{1+\frac{t^2}{2}+t\sqrt{1+\frac{t^2}{4}}} &=&-n_2\, s_0 + \sinh^{-1}(Rs_0).
		\end{eqnarray*}
		\begin{enumerate}[(a)]
			\item If $T(s_0)\le s_0$, let $\tau_{1,2}=T(s_0)$.
		\item If $T(s_0)\ge s_0$, let $\tau_{1,2}$ be given by the unique positive solution of \begin{eqnarray*}
			n_1 t-\ln \sqrt{1+\frac{t^2}{2}+t\sqrt{1+\frac{t^2}{4}}} &=&-n_2\, t + \sinh^{-1}(Rt).
		\end{eqnarray*}
	\end{enumerate}
			Then the switched system~\eqref{eq:system} is stable for all $\sigma\in S_{\tau_{1,2}}$ and asymptotically stable for all $\sigma\in S_{\tau_{1,2}}'$. When $c\ne 0$, a lower $\tau_{1,2}$ can be obtained using $R_{opt}=c^2/2$ instead of $R$.
	\end{theorem}
	\begin{proof}
		Analogous to the proof of Theorem~\ref{nn:thm21>1/2}.
	\end{proof}
		
	\begin{theorem}\label{nn:thm12<1/2}
		(With notations in this section) Suppose $n_1< 1/2$. Let $R=\left(c^2+d^2\right)/2$ and $L=\sqrt{1-4n_1^2}-\ln\sqrt{\frac{1}{2n_1^2}-1+\frac{\sqrt{1-4n_1^2}}{2n_1^2}}$.\\
		i) If $R\le n_2$, there exists a unique pair $(t_{max},s_{max})$ satisfying the following equations 
		\[
		n_1t_{max}-\ln \sqrt{1+\frac{t_{max}^2}{2}+t_{max}\sqrt{1+\frac{t_{max}^2}{4}}}=-n_2s_{max}+\sinh^{-1}Rs_{max}=L.
		\]
		\begin{enumerate}[(a)]
		\item If $s_{max}\le t_{max}$, let $\tau_{1,2}=s_{max}$. 
		\item If $s_{max}>t_{max}$, let $\tau_{1,2}$ be the unique positive solution of \begin{eqnarray*}
			n_1t-\ln \sqrt{1+\frac{t^2}{2}+t\sqrt{1+\frac{t^2}{4}}}&=&-n_2t+\sinh^{-1}(Rt).
		\end{eqnarray*}		
	\end{enumerate}
		ii) If $R>n_2$, let $\tau_{1,2}$ be the unique positive solution of the above equation.\\
		Then the switched system~\eqref{eq:system} is stable for all $\sigma\in S_{\tau_{1,2}}$ and asymptotically stable for all $\sigma\in S_{\tau_{1,2}}'$. When $c\ne 0$, a lower $\tau_{1,2}$ can be obtained using $R_{opt}=c^2/2$ instead of $R$.
	\end{theorem}
	
	\begin{proof}
		Analogous to the proof of Theorem~\ref{nn:thm21<1/2}.
	\end{proof}
	
	\begin{remark}[Optimal choice of Jordan basis matrices $P_1$ and $P_2$]\label{nn:optimal}
		When $A$ is a planar defective with eigenvalue $-n<0$, the Jordan basis matrix corresponding to the Jordan form $P=\begin{pmatrix}
			-n&1\\ 0&-n
		\end{pmatrix}$ consists of an eigenvector and a generalized eigenvector. Suppose we make a choice $\vec{x}$ of an eigenvector and $\vec{j}$ of a generalized eigenvector such that $(A+nI)\vec{j}=\vec{x}$. Then for any $\epsilon$, $\vec{j}+\epsilon \mathbf{x}$ is also a generalized eigenvector. Hence choices for $P$ are $\begin{bmatrix} \vec{x} & \vec{j}+\epsilon\vec{x}	\end{bmatrix}$ and all its nonzero scalar multiples. 
	
	Coming back the two subsystems as was discussed in this section. For $i=1,2$, suppose $A_i$ be a planar defective matrix with Jordan form $\begin{pmatrix}
		-n_i &1\\ 0& -n_i
	\end{pmatrix}$. Suppose $P_i=\begin{bmatrix} \vec{x}_i & \vec{j}_i	\end{bmatrix}$ with $\det(P_i)=\pm 1$; $\widetilde{P}_i^{(\epsilon_i)}=\begin{bmatrix} \vec{x}_i & \vec{j}_i+\epsilon_i\vec{x}_i	\end{bmatrix}$, $M=P_2^{-1}P_1=\begin{pmatrix}
	a&b\\c&d
\end{pmatrix}$, and $\widetilde{M}(\epsilon_1,\epsilon_2)=\left(\widetilde{P}_2^{(\epsilon_2)}\right)^{-1}\widetilde{P}_1^{(\epsilon_1)}$. Calculating $\widetilde{K}(\epsilon_1,\epsilon_2)$ and $\widetilde{R}(\epsilon_1,\epsilon_2)$ corresponding to $\widetilde{M}(\epsilon_1,\epsilon_2)$ as in Theorems~\ref{nn:thm21>1/2},~\ref{nn:thm21<1/2},~\ref{nn:thm12>1/2}, and~\ref{nn:thm12<1/2}, one can observe that $\widetilde{K}$ is independent of $\epsilon_1$ and $\widetilde{R}$ is independent of $\epsilon_2$. Note that $c=0$ if and only if $\vec{x}_1\in \textup{Span}(\vec{x}_2)$. In this case, $\widetilde{K}$ and $\widetilde{R}$ are constant functions. 

The functions $\widetilde{K}$ and $\widetilde{R}$ attain the minimum value $K_{opt}$ and $R_{opt}$, respectively, where 
\[
K_{opt}=\begin{cases}
					a^2/2, & \vec{x}_1\in \textup{Span}(\vec{x}_2)\\ 
					c^2/2, & \text{otherwise} 
				\end{cases}, \ \ R_{opt}=\begin{cases}
				d^2/2, & \vec{x}_1\in \textup{Span}(\vec{x}_2)\\ 
				c^2/2, & \text{otherwise} 
			\end{cases}.
			\]
		  These values of $K_{opt}, R_{opt}$ will lower the bounds $\tau_{2,1}, \tau_{1,2}$, as the case may be, in Theorems~\ref{nn:thm21>1/2},~\ref{nn:thm21<1/2},~\ref{nn:thm12>1/2}, and~\ref{nn:thm12<1/2}.

		When $c\ne 0$, these optimal values $K_{opt}$ and $R_{opt}$ are attained at $\epsilon_1=-\det\begin{bmatrix}\vec{j}_1 & \vec{x}_2\end{bmatrix}/\det\begin{bmatrix} \vec{x}_1 & \vec{x}_2\end{bmatrix}$ and $\epsilon_2=-\det\begin{bmatrix} \vec{x}_1 & \vec{j}_2\end{bmatrix}/\det\begin{bmatrix} \vec{x}_1 & \vec{x}_2\end{bmatrix}$. For these choices of $\epsilon_i$, $\widetilde{P}_i^{(\epsilon_i)}$ are the optimal choices of Jordan basis matrices for subsystems $i=1,2$; that means, these give the least values for $\tau_{1,2}$ and $\tau_{2,1}$ by following the above procedure. Moreover $K_{opt}$ and $R_{opt}$ are independent of the choices of $P_1$ and $P_2$. 
	\end{remark}

	\section{$A_1$ is defective and $A_2$ has complex eigenvalues}\label{sec:NC}
	
	Suppose $A_1$ is a defective matrix with canonical form $J_1=\begin{pmatrix}
		-n_1 & 1\\0 & -n_1
	\end{pmatrix}$, for some $n_1>0$, and $A_2$ has complex eigenvalues with canonical form $J_2= \begin{pmatrix} -\alpha_2 & \beta_2\\ -\beta_2 & -\alpha_2\end{pmatrix}$, where $\alpha_2>0$ and $\beta_2\ne 0$. In this case, both $D_1$ and $D_2$ are identity.
	
	\subsection{Computing $\tau_{1,2}$}
	Note that
	\begin{eqnarray*}
		\left\|M^{-1} {\rm{e}}^{J_2 s}\, M{\rm{e}}^{J_1 t}\right\| &\le & \left\|M^{-1} {\rm{e}}^{J_2 s}\, M{\rm{e}}^{-n_1 t} \right\|\theta(t),
	\end{eqnarray*}
	where $\theta(t)=\begin{Vmatrix}\begin{pmatrix}
			1 & t\\ 0 & 1
		\end{pmatrix}
	\end{Vmatrix}=\sqrt{1+\frac{t^2}{2}+t\sqrt{1+\frac{t^2}{4}}}$ (the function $\theta$ was also used in Section~\ref{sec:NN}).

\noindent Observe that $\left\|M^{-1}{\rm{e}}^{J_2 s}\, M {\rm{e}}^{-n_1 t} \right\|\theta(t)<1$ if and only if 
	
	\begin{enumerate}
		\item[(C1)] $(t,s)$ lies in the feasible region $\alpha_2 s + n_1 t-\ln\theta(t)>0$, and
				\item[(C2)] $(t,s)$ satisfies the following inequality
				\begin{eqnarray*}
			\left(2\cos^2\beta_2 s +L\sin^2\beta_2 s \right)\,\theta(t)^2 {\rm e}^{-2 (\alpha_2 s + n_1 t)}<1+\theta(t)^4{\rm e}^{-4 (\alpha_2 s + n_1 t)},
		\end{eqnarray*}
		where $L=(a^2+b^2)^2+2(ac+bd)^2+(c^2+d^2)^2$. This inequality can be rewritten in a simplified form $f(t,s)<0$ where $f(t,s)$ equals
		\[
		\sqrt{\frac{(b^2-c^2)^2+(a^2-d^2)^2+2(ac+bd)^2+2(ab+cd)^2}{4}}\left\lvert \sin \beta_2 s\right\rvert-\sinh \left(\alpha_2 s+n_1t-\ln\theta(t)\right).
		\]
	\end{enumerate} 
	
	Recall the function $g(t)=n_1t-\ln\theta(t)$ defined in Section~\ref{sec:NN} and refer to Remark~\ref{funcgtheta} for properties of $g$.
	
	\begin{theorem}\label{nc:thm12}
		(With notations in this section) Let \begin{eqnarray*}
			K=\sqrt{\frac{(b^2-c^2)^2+(a^2-d^2)^2+2(ac+bd)^2+2(ab+cd)^2}{4}}.
		\end{eqnarray*}
		i) If $n_1\ge 1/2$, let $\tau_{1,2}$ be the unique solution of the equation $n_1t-\ln\theta(t)=-\alpha_2 t+\sinh^{-1}K$. \\
		ii) If $n_1<1/2$, let $t_0=\sqrt{(1/n_1^2)-4}$. 
		\begin{enumerate}[(a)]
	\item If $n_1t_0-\ln\theta(t_0)\le -\alpha_2 t_0+\sinh^{-1}K$, let $\tau_{1,2}$ be the unique solution of $n_1t-\ln\theta(t)=-\alpha_2 t+\sinh^{-1}K$. 
	\item Otherwise, let $\tau_{1,2}=-(1/\alpha_2)\left(n_1t_0-\ln\theta(t_0)-\sinh^{-1}K\right)$.
\end{enumerate}
	Then the switched system~\eqref{eq:system} is stable for all $\sigma\in S_{\tau_{1,2}}$ and asymptotically stable for all $\sigma\in S_{\tau_{1,2}}'$.
	\end{theorem}
	
	\begin{proof}
		First consider the case when $n_1\ge 1/2$. We will find a function $T$ whose graph bounds the zero set $Z$ of the function $f(t,s)$ as shown in Figure~\ref{fig:nc12n>1/2}. The function $g$ is strictly increasing and hence $\alpha_2 s + n_1 t-\ln\theta(t)>0$, for all  $t,s>0$. This shows that (C1) is satisfied by all points of $Q_1$. Thus, we just need to look at (C2), that is, $f(t,s)<0$. Note that for each fixed $s$, there is a unique $t_s$ such that $g(t_s)=-\alpha_2 s+\sinh^{-1}(K\lvert\sin\beta_2 s\rvert)$.  Since the function $g$ is strictly increasing (from Remark~\ref{funcgtheta}), the graph of $t_s$ in the first quadrant lies below the graph of the function $T$ given by
		\begin{eqnarray*}
			T(s)&=&g^{-1}\left(-\alpha_2 s+\sinh^{-1}K\right).
		\end{eqnarray*} 
		Further $T$ is concave down and decreasing with $s>0$. Let $t_0$ be the unique fixed point of the function $T$. Then, the zero set $Z$ does not intersect with $\mathcal{R}_{t_0}=\{(t,s)\in Q_1\,\colon t,s>t_0\}$. Moreover, (C2) is satisfied in $\mathcal{R}_{t_0}$ and hence the result follows.
		\begin{figure}[h!]
			\centering
			\includegraphics[scale=0.2]{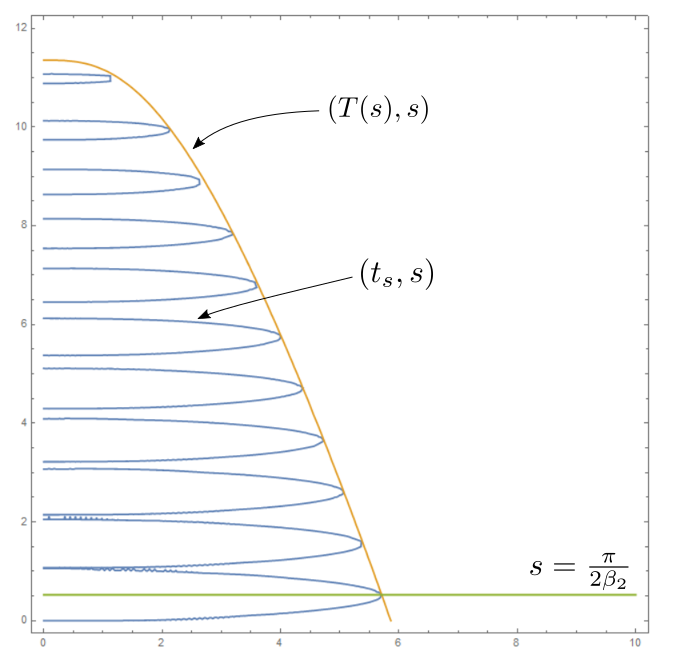}
			\caption{The zero set $Z$ when $n_1\ge 1/2$. The graph corresponds to the values $n_1=1/2$, $\alpha_2=0.1$, $\beta_2=3$ and $K=1.4$.}
			\label{fig:nc12n>1/2}
		\end{figure}
		
		Now, consider the case when $n_1<1/2$. We will prove that the zero set $Z$ can be bounded as shown in Figure~\ref{fig:nc12n<1/2}. Since $n_1< 1/2$, by Remark~\ref{funcgtheta}, $g(t)$ is concave up, decreasing on $(0,t_0)$ and increasing on $(t_0,\infty)$, where $t_0=\sqrt{(1/n_1^2)-4}$ (this $t_0$ is different from that in the previous case). Denote $g_{min}=-g(t_0)$. The zero set $Z$ lies in the region bounded by the graphs of $g(t)=F_+(s)$ and $g(t)=F_-(s)$, where $F_\pm(s)=-\alpha_2 s\pm \sinh^{-1} K$.
		
		For each $t$, there exists at most one solution $s_t^+\ge 0$ and at most one solution $s_t^-\ge 0$ such that $g(t)=F_\pm (s_t^\pm)$. Also note that the graph of $g(t)=F_+(s)$ lies above that of $g(t)=F_-(s)$. The following two situations arise depending on the value of $-g_{min}+\sinh^{-1} K$. If $-g_{min}+\sinh^{-1} K>0$, then the graph of $g(t)=F_-(s)$ does not lie in the first quadrant.
		
		\begin{figure}[h!]
			\centering
			\includegraphics[scale=0.35]{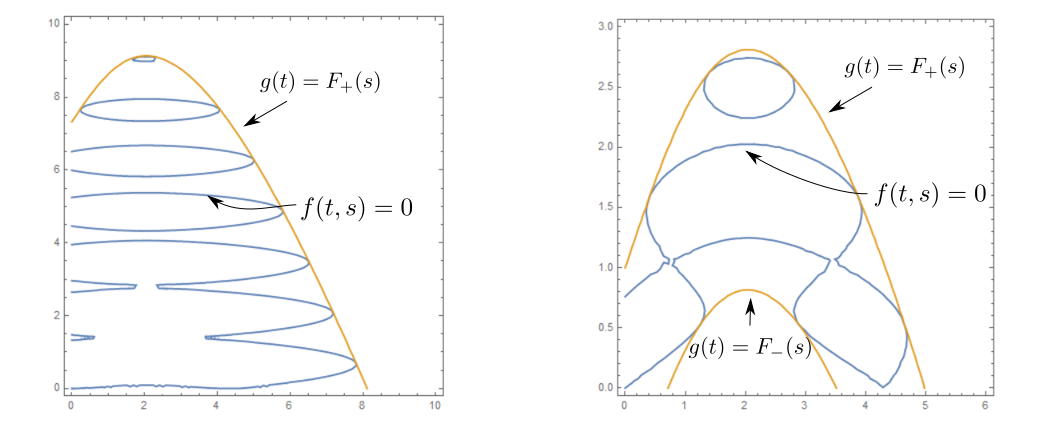}
			\caption{When $-g_{min}+\sinh^{-1}K>0$ (left) and when $-g_{min}+\sinh^{-1}K<0$ (right). The graphs correspond to the values $\alpha_2=0.1$, $n_1=0.35$ in both; $\beta_2=2.25$ and $K=0.8$ in the left plot; $\beta_2=3$ and $K=0.1$ in the right plot.}
			\label{fig:nc12n<1/2}
		\end{figure}
		
		In both of the cases, $f(t,s)<0$ on the unbounded region of the first quadrant partitioned by $g(t)=F_+(t)$. For any point $(t,s)$ lying on the curve $\alpha_2 s+n_1t-\ln\theta(t)=0$, $f(t,s)>0$. Hence the curve $\alpha_2 s+n_1t-\ln\theta(t)=0$ does not intersect with the unbounded region of the first quadrant bounded by the curve $g(t)=F_+(t)$. Thus if we find a region $\mathcal{R}_\tau=\{(t,s)\in Q_1\colon\ t,s>\tau\}$ where (C2) is satisfied, (C1) will automatically be satisfied.
				
		Thus, in both of the cases, $\tau_{1,2}$ can be computed by studying the solution set of $g(t)=F_+(s)$, which can be simplified as $s=-(1/\alpha_2)\left(g(t)-\sinh^{-1}K\right)$. The graph of $-(1/\alpha_2)\left(g(t)-\sinh^{-1}K\right)$ is concave down and attains its maximum at $t_0=\sqrt{(1/n_1^2)-4}$ (follows from Remark~\ref{funcgtheta}). Hence the result follows.
	\end{proof}
	
	\subsection{Computing $\tau_{2,1}$}
	\begin{theorem}\label{nc:thm21}
		(With notations in this section) Let $R=\left(a^2+c^2\right)/2$. \\
		i) The switched system~\eqref{eq:system} is stable for all signals $\sigma$ if $R\le n_1$. Let $\tau_{2,1}=0$.\\
		ii) If $R>n_1$, let $t_0=\sqrt{\frac{1}{n_1^2}-\frac{1}{R^2}}$ and $S(t_0)=-\frac{n_1}{\alpha_2}\, t_0 + \frac{1}{\alpha_2}\,\sinh^{-1}Rt_0$. 
		\begin{enumerate}[(a)]
		\item If $S(t_0)\le t_0$, let $\tau_{2,1}=S(t_0)$.
		\item If $S(t_0)\ge t_0$, let $\tau_{2,1}$ be the unique positive solution of $\sinh \left(n_1t+\alpha_2 t\right)=R t$.
	\end{enumerate}
		Then the switched system~\eqref{eq:system} is stable for all $\sigma\in S_{\tau_{2,1}}$ and asymptotically stable for all $\sigma\in S_{\tau_{2,1}}'$.
	\end{theorem}
	
	\begin{proof}
	Observe that $\left\|M {\rm{e}}^{J_1 t\,}M^{-1}{\rm{e}}^{J_2 s} \right\|<1$ if and only if 
	\begin{eqnarray*}
			\left(2+(a^2+c^2)^2\, t^2\right) {\rm e}^{-2 (\alpha_2 s + n_1 t)}<1+{\rm e}^{-4 (\alpha_2 s + n_1 t)},
		\end{eqnarray*}
	which can be simplified as $f(t,s)<0$, where $f(t,s)=Rt-\sinh \left(n_1t+\alpha_2 s\right)$.
	Clearly $f(0,0)=0$. When $s=0$, then $f(t,0)<0$ if and only if $\sinh n_1t>Rt$.
	The function $(\sinh n_1t)/t$ is increasing in $t$ and attains its minimum value $n_1$ at $t=0$. When $R\le n_1$, $f(t,0)<0$ for all $t>0$. Since $f_s(t,s)<0$, $f(t,s)<0$ for all $t,s>0$ and hence the result follows.
		
		When $R>n_1$, we will prove that the zero set $Z$ of $f(t,s)$ can be classified into the three cases as shown in Figure~\ref{fig:rc21}. Further, there exists a unique $\tilde{t}>0$ such that $f(\tilde{t},0)=0$. Also, $f(t,0)>0$ for $t\in (0,\tilde{t})$ and $f(t,0)<0$ for $t\in(\tilde{t},\infty)$. For each $t$, there exists at most one $s_t\ge 0$ such that $f(t,s_t)=0$ and it is given by $s_t=-(n_1/\alpha_2) t + (1/\alpha_2)\,\sinh^{-1}\left(Rt\right)$.
		Define $S\colon [0,\tilde{t}]\to\mathbb{R}$ as $S(t)=s_t$. Note $S(0)=S\left(\tilde{t}\right)=0$. This function parameterizes the zero set $Z$ in the first quadrant. We have analyzed a similar function~\eqref{zerochar2} earlier. For $t\ne 0$, $S''(t)<0$ and $S'(t)=-\frac{n_1}{\alpha_2}+\frac{R}{\alpha_2}\left(R^2t^2+1\right)^{-1/2}$. Hence, $\lim_{t\to 0}\,S'(t)>0$ and there exists a unique $t_0$ such that $S'(t_0)=0$. Thus the zero set $Z$ can be classified into the three cases as shown in Figure~\ref{fig:rc21}. The first two cases correspond to when $S(t_0)\le t_0$. In these cases, $Z$ does not intersect with the region $\mathcal{R}_{S(t_0)}=\{(t,s)\in\mathbb{R}^2\colon\, s,t>S(t_0)\}$ and hence the result follows. The third case corresponds to the case when $S(t_0)>t_0$. There is a unique root, say $\tau$, in the direction $s=t$ due to $S$ being concave down. Hence the zero set $Z$ does not intersect with the region $\mathcal{R}_\tau=\{(t,s)\in\mathbb{R}^2\colon\, s,t>\tau\}$ and the result follows. Moreover, $\tau$ is the unique positive solution of $\sinh \left(n_1t+\alpha_2 t\right)=Rt$.
	\end{proof}
	
	\begin{remark}[Further optimization of $\tau_{1,2}$ and $\tau_{2,1}$]\label{nc:optimization}
		We refer to Remark~\ref{nn:optimal} for notations. Suppose $P_1=\begin{bmatrix} \vec{x}_1 & \vec{j}_1	\end{bmatrix}$ and $\widetilde{P}_1^{(\epsilon_1)}=\begin{bmatrix} \vec{x}_1 & \vec{j}_1+\epsilon_1\vec{x}_1	\end{bmatrix}$. Then $\widetilde{M}(\epsilon_1)= P_2^{-1}\widetilde{P}_1^{(\epsilon_1)}=\begin{pmatrix}a &b+\epsilon_1 a\\c & d+\epsilon_1 c\end{pmatrix}$. Clearly $R$ is invariant under choice of $\epsilon_1$. For finding $K_{opt}$ (for Theorem~\ref{nc:thm12}), we can minimize the function $\widetilde{K}(\epsilon_1)$ corresponding to $\widetilde{M}(\epsilon_1)$. By using $K_{opt}$ instead of $K$ in Theorem~\ref{nc:thm12}, a lower $\tau_{1,2}$ can be obtained.
	\end{remark}
	
	\section{$A_1$ is defective and $A_2$ is real diagonalizable}\label{sec:NR}
	
	Suppose $A_1$ is a defective matrix with canonical form $J_1=\begin{pmatrix}
		-n_1 & 1\\0 & -n_1
	\end{pmatrix}$, for some $n_1>0$ and $A_2$ is a real diagonalizable matrix with canonical form $J_2= \begin{pmatrix} -p_2 & 0\\ 0 & -q_2\end{pmatrix}$, where $0<p_2<q_2$. We make a choice of matrices $P_1$ and $P_2$ such that the determinant of the transition matrix $M=P_2^{-1}P_1$ is $1$. Let $D_1,D_2$ be scaling matrices. In this section, we will assume that $D_1$ is the identity matrix; the general case will be discussed in Remark~\ref{nr:optimization}. We vary $D_2$ over matrices of the form $\text{diag}(\lambda_2,1/\lambda_2)$.
	
	\subsection{Computing $\tau_{1,2}$}
		Note that
	\begin{eqnarray*}
		\left\|M_{D_1,D_2}^{-1} {\rm{e}}^{J_2 s}\, M_{D_1,D_2}{\rm{e}}^{J_1 t}\right\| &\le & \left\|M_{D_1,D_2}^{-1} {\rm{e}}^{J_2 s}\, M_{D_1,D_2}{\rm{e}}^{-n_1 t} \right\|\theta(t),
	\end{eqnarray*}
	where $\theta(t)=\begin{Vmatrix}\begin{pmatrix}
			1 & t\\ 0 & 1
		\end{pmatrix}
	\end{Vmatrix}=\sqrt{1+\frac{t^2}{2}+t\sqrt{1+\frac{t^2}{4}}}$ (the function $\theta$ has also been used in earlier sections). Observe that $\left\|M_{D_1,D_2}^{-1}{\rm{e}}^{J_2 s}\, M_{D_1,D_2} {\rm{e}}^{-n_1 t} \right\|\theta(t)<1$ if and only if 
	\begin{enumerate}
		\item[(C1)] $(t,s)$ is contained in the feasible region $(p_2+q_2)s/2+g(t)>0$, and
		\item[(C2)] $(t,s)$ satisfies the inequality $f(t,s)<0$, where
		\begin{eqnarray*}
		f(t,s) &=& 	(a^2+b^2)(c^2+d^2)\,\sinh^2\left(\frac{q_2-p_2}{2}\right)s-\sinh^2 \left(\left(\frac{p_2+q_2}{2}\right)s+g(t) \right),
		\end{eqnarray*} 
		where $g(t)=n_1t-\ln\theta(t)$ (also used in Section~\ref{sec:NN}, the properties of $g$ are given in Remark~\ref{funcgtheta}).
	\end{enumerate}
	
	In the following two results, we compute $\tau_{1,2}$ separately when $0<n_1<1/2$ and when $n_2\ge 1/2$.
	
	\begin{theorem}\label{nr:thm12}
		(With notations in this section) Suppose $n_1\ge 1/2$. Let $K=(a^2+b^2)(c^2+d^2)$. \\
		i) If $1\le K\le ((q_2+p_2)/(q_2-p_2))^2$, the switched system~\eqref{eq:system} is stable for all switching signals $\sigma$. Let $\tau_{1,2}=0$.\\
		ii) If $K> ((q_2+p_2)/(q_2-p_2))^2$, calculate $s_0$ and $t=t_{s_0}$ given by the unique solutions of 
		\begin{eqnarray*}
			\left(1-\frac{1}{K}\right)\tanh^2\left(\frac{q_2-p_2}{2}\right)s_0&=&\left(\frac{q_2-p_2}{q_2+p_2}\right)^2-\frac{1}{K},\\
			n_1t-\ln \sqrt{1+\frac{t^2}{2}+t\sqrt{1+\frac{t^2}{4}}}&=&-\left(\frac{p_2+q_2}{2}\right)s_0+\sinh^{-1}\left(\sqrt{K}\sinh\left(\frac{q_2-p_2}{2}\right)s_0\right).
		\end{eqnarray*}
	\begin{enumerate}[(a)]
		\item If $t_{s_0}\le s_0$, let $\tau_{1,2}=t_{s_0}$. 
		\item Otherwise, let $\tau_{1,2}$ be the unique solution $t$ of 
		\begin{eqnarray*}
			n_1t-\ln \sqrt{1+\frac{t^2}{2}+t\sqrt{1+\frac{t^2}{4}}}&=&-\left(\frac{p_2+q_2}{2}\right)t+\sinh^{-1}\left(\sqrt{K}\sinh\left(\frac{q_2-p_2}{2}\right)t\right).
		\end{eqnarray*}
	\end{enumerate}
	Then the switched system~\eqref{eq:system} is stable for all $\sigma\in S_{\tau_{1,2}}$ and asymptotically stable for all $\sigma\in S_{\tau_{1,2}}'$. Moreover, a lower $\tau_{1,2}$ can be obtained on replacing $K$ by $K_{opt}$ as defined in Remark~\ref{nr:optimization}.
	\end{theorem}
	
	\begin{proof}
		The function $g(t)$ is increasing since $n_1\ge 1/2$, hence (C1) is satisfied for all $t,s>0$. Thus we need to just look at (C2). Observe that $f(0,0)=0$ and for $s\ne 0$, $f(0,s)=0$ if and only if \begin{eqnarray}\label{rnzerot=0}
			\sqrt{K}&=&\frac{\sinh((q_2+p_2)s/2)}{\sinh((q_2-p_2)s/2)}.
		\end{eqnarray}
		The function on the right is increasing in $s$ on the interval $[0,\infty)$ and attains its minimum value $(q_2+\nobreak p_2)/(q_2-\nobreak p_2)$ at $t=0$. Clearly, when $K\le ((q_2+p_2)/(q_2-p_2))^2$,~\eqref{rnzerot=0} has no nonzero solution. Hence $f(0,s)<0$ for all $s>0$, which further implies $f(t,s)<0$ for all $t,s>0$ since $f_t(t,s)<0$. Hence the result follows.
		
		When $K> ((q_2+p_2)/(q_2-p_2))^2$, we  will show that the zero set of $f(t,s)$ can take one of the three forms as shown in Figure~\ref{fig:nr12>}. First,~\eqref{rnzerot=0} has a unique positive solution, say $\tilde{s}$. Also, $f(0,s)<0$ for all $s>\tilde{s}$, which implies $f(t,s)<0$ for all $s>\tilde{s}$ and $t>0$. Further for each $s\in[0,\tilde{s}]$, there is a unique $t_s\ge 0$ such that $f(t_s,s)=0$, which is equivalent to
		\begin{eqnarray}\label{eq:gts}
			g(t_s)&=&-\left(\frac{p_2+q_2}{2}\right)s+\sinh^{-1}\left(\sqrt{K}\sinh\left(\frac{q_2-p_2}{2}\right)s\right).
		\end{eqnarray}
		We will now show that $t_s$ (as a function of $s$) on the interval $[0,\tilde{s}]$ is concave down and has a unique maxima. Operating both sides by $g^{-1}$ (which exists since $g$ is strictly increasing), note that $t_s$ is a differentiable function of $s$. An easy computation shows that the function on the right in~\eqref{eq:gts} is concave down on the interval $[0,\tilde{s}]$ and has a positive derivative at 0. Upon partially differentiating~\eqref{eq:gts} with respect to $s$ and using the fact that $g$ is increasing implies that $t_s$ has a maxima at $s_0\in(0,\tilde{s})$ given by the unique solution of 
		\begin{eqnarray*}
			\left(1-\frac{1}{K}\right)\tanh^2\left(\frac{q_2-p_2}{2}\right)s+\frac{1}{K}&=&\left(\frac{q_2-p_2}{q_2+p_2}\right)^2.
		\end{eqnarray*}
	Hence the result follows.
		\begin{figure}[h!]
			\centering
			\includegraphics[scale=0.15]{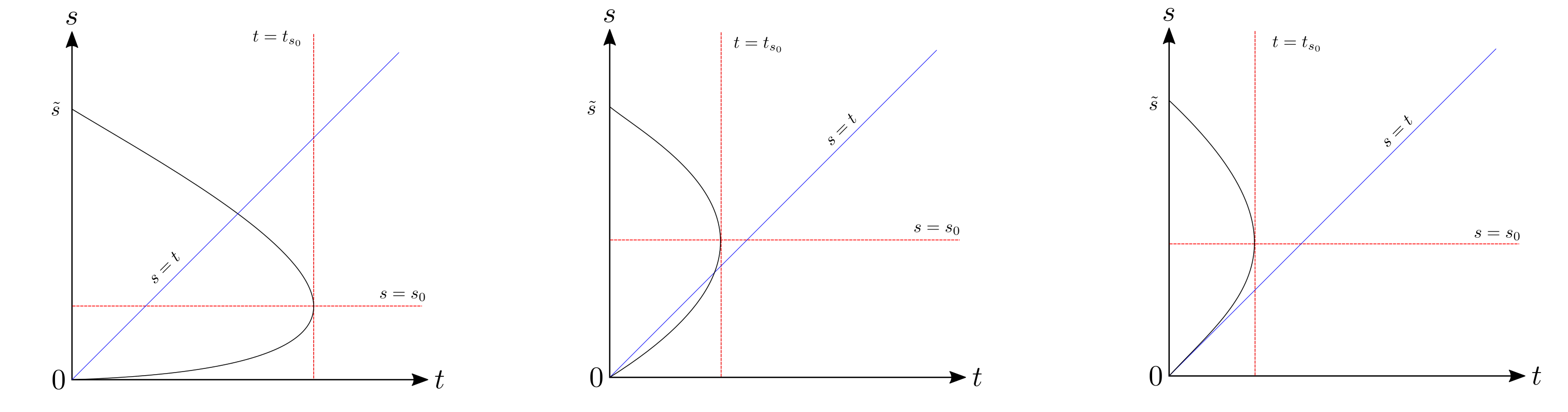}
			\caption{Zero set of $f(t,s)$: the first one corresponds to $s_0\le t_{s_0}$; the second and the third one correspond to $s_0>t_{s_0}$.}\label{fig:nr12>}
		\end{figure}
	\end{proof}

	\begin{theorem}\label{nr:thm12<1/2}
		(With notations in this section) Suppose $n_1< 1/2$. Let $K=(a^2+b^2)(c^2+d^2)$. \\
		i) If $K\le ((q_2+p_2)/(q_2-p_2))^2$, let $t_0=\sqrt{(1/n_1^2)-4}$ and let $s_{t_0}$ be the unique solution of 
		\begin{eqnarray*}
			n_1t_0-\ln \sqrt{1+\frac{t_0^2}{2}+t_0\sqrt{1+\frac{t_0^2}{4}}}&=&-\left(\frac{q_2+p_2}{2}\right)s_{t_0}+\sinh^{-1}\left(\sqrt{K}\sinh\left(\frac{q_2-p_2}{2}\right)s_{t_0}\right).
		\end{eqnarray*}
	\begin{enumerate}[(a)]
		\item If $s_{t_0}\le t_0$, let $\tau_{1,2}=s_{t_0}$.
		\item Otherwise let $\tau_{1,2}$ be the unique solution $t$ of 
		\begin{eqnarray}\label{eq:NR2}
			n_1t-\ln \sqrt{1+\frac{t^2}{2}+t\sqrt{1+\frac{t^2}{4}}}&=&-\left(\frac{q_2+p_2}{2}\right)t+\sinh^{-1}\left(\sqrt{K}\sinh\left(\frac{q_2-p_2}{2}\right)t\right).
		\end{eqnarray}
	\end{enumerate}
		ii) If $K> ((q_2+p_2)/(q_2-p_2))^2$, let $\tau_{1,2}$ be the unique solution of~\eqref{eq:NR2}.\\	Then the switched system~\eqref{eq:system} is stable for all $\sigma\in S_{\tau_{1,2}}$ and asymptotically stable for all $\sigma\in S_{\tau_{1,2}}'$. Moreover, a lower $\tau_{1,2}$ can be obtained on replacing $K$ by $K_{opt}$ as defined in Remark~\ref{nr:optimization}.
	\end{theorem}
	
	\begin{proof}
		From the second condition (C2), the zero set of $f(t,s)$ can be simplified as $g(t)=F_\pm(s)$, where 
		\begin{eqnarray*}
			F_\pm(s)&=& -(p_2+q_2)s/2\pm\sinh^{-1}\left(\sqrt{K}\sinh\left((q_2-p_2)s/2\right)\right).
		\end{eqnarray*}
		Notice that $g(t)$ can assume both positive and negative values since $n_1<1/2$ (Remark~\ref{funcgtheta}). 
		
		If $K< ((q_2+p_2)/(q_2-p_2))^2$, $F_+$ is concave down and since $F_+'(0)<0$, $F_+'(s)<0$ on $(0,\infty)$. Similarly, $F_-$ is concave up and $\lim_{s\to\infty}F_-'(s)=-q_2$, hence $F_-'(s)<0$ on $(0,\infty)$. Also $F_-(s)<F_+(s)$. Refer to Figure~\ref{NRfpfm}.
		
				\begin{figure}[h!]
			\centering
			\includegraphics[scale=0.6]{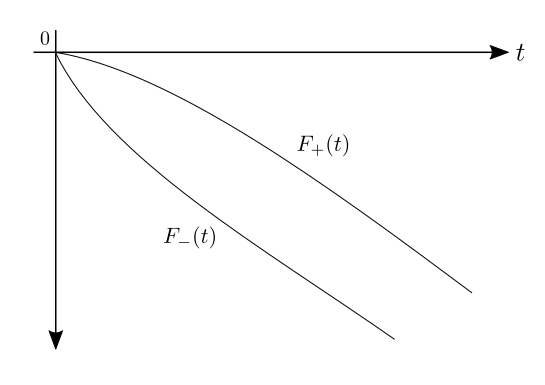}\hfill
			\includegraphics[scale=0.4]{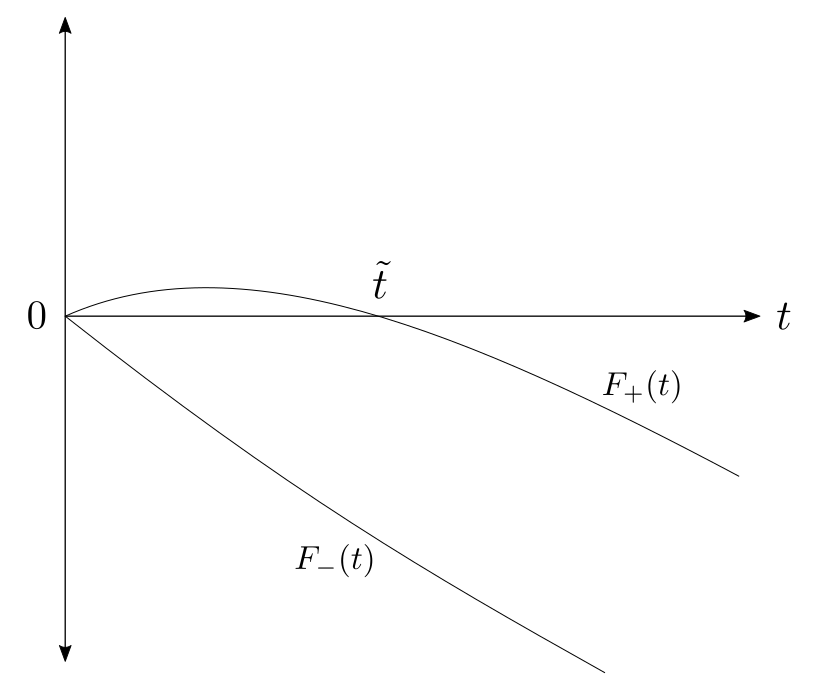}
			\caption{Graphs of $F_+$, $F_-$ when $K< ((q_2+p_2)/(q_2-p_2))^2$ (left) and when $K> ((q_2+p_2)/(q_2-p_2))^2$ (right).}\label{NRfpfm}
		\end{figure}
		
		Arguing as in the proof of Theorem~\ref{nn:thm21<1/2}, $\tau_{1,2}$ can be obtained by studying the graph of $g(t)=F_+(s)$ since the curves $g(t)=F_-(s)$ and $((p_2+q_2)/2) s+g(t)$ do not intersect with the unbounded region of the first quadrant partitioned by the curve $g(t)=F_+(t)$. For $s=0$, $f(t,0)=0$ if and only if $g(t)=F_+(0)=0$ if and only if $t\in\{0,\tilde{t}\}$ where $\tilde{t}$ denotes the largest positive zero of $g$. For $t>\tilde{t}$, $g(t)>F_+(0)$, hence $g(t)>F_+(s)$ for all $s\ge 0$ (since $F_+$ is decreasing) and thus $f(t,s)<0$ for all $t>\tilde{t}$ and $s\ge 0$. For $t\in [0,\tilde{t}]$, there exists unique $s_t^+\ge 0$ such that $g(t)=F_+(s_t^+)$. Since $F_+'(s)<0$ for all $s>0$, it has a $\mathcal{C}^1$ inverse. Thus $s_t^+$ is differentiable in $t\in[0,\tilde{t}]$. Further $s_t^+$ attains maximum at $t_0=\sqrt{(1/n_1^2)-4}$ and is concave down (since $g$ is concave up) and $F_+'(s)<0$ for all $s\ge 0$. Hence the result follows.
		
		If $K> ((q_2+p_2)/(q_2-p_2))^2$, then $F_-(s)$ behaves as in the previous case, but $F_+(s)$ is concave down and has a unique maxima, refer to Figure~\ref{NRfpfm}. Arguing as in the proof of Theorem~\ref{nn:thm21<1/2}, we obtain the result. 
	\end{proof}
		
	\subsection{Computing $\tau_{2,1}$}
	\begin{theorem}\label{nr:thm21}
		(With notations in this section)\\ 
		i) If $|ac|\le n_1$, the switched system~\eqref{eq:system} is stable for all signals $\sigma$. Let $\tau_{2,1}=0$.\\
		ii) If $\lvert ac\rvert >n_1$ calculate the unique nonzero solution $(t_0,S(t_0))$ satisfying the system \begin{eqnarray*}
		act \cosh\left(\left(\frac{q_2-p_2}{2}\right)s\right)&=&\sinh\left(\frac{q_2-p_2}{2}\right)s+ sgn(ac)\, \sinh\left(\left(\frac{q_2+p_2}{2}\right)s+n_1 t\right),\\
		t &=& -\left(\frac{q_2+p_2}{2n_1}\right)s+\frac{1}{n_1}\cosh^{-1}\left(\frac{\lvert ac\rvert}{n_1}\, \cosh\left(\left(\frac{q_2-p_2}{2}\right)s\right)\right).
	\end{eqnarray*}
	\begin{enumerate}[(a)]
	\item If $S(t_0)\le t_0$, let $\tau_{2,1}=S(t_0)$.
	\item If $S(t_0)> t_0$, let $\tau_{2,1}$ be the unique solution $t\in(t_0,\infty)$ of \[
	act \cosh\left(\left(\frac{q_2-p_2}{2}\right)t\right)=\sinh\left(\frac{q_2-p_2}{2}\right)t+ sgn(ac)\, \sinh\left(\left(\frac{q_2+p_2}{2}+n_1\right)t\right).
	\]
	\end{enumerate}
		Then the switched system~\eqref{eq:system} is stable for all $\sigma\in S_{\tau_{2,1}}$ and asymptotically stable for all $\sigma\in S_{\tau_{2,1}}'$.
	\end{theorem}
	
	\begin{proof}
		We have $\left\|M_{D_1,D_2} {\rm{e}}^{J_1 t\,}M_{D_1,D_2}^{-1}{\rm{e}}^{J_2 s} \right\|<1$ if and only if the Schur's function $f(\lambda_2,t,s)<0$ where $f(\lambda_2,t,s)$ equals
		\begin{align*}
		\frac{1}{2}\left(c^4\lambda_2^4{\rm{e}}^{(q_2-p_2) s}+\frac{a^4}{\lambda_2^4} {\rm{e}}^{-(q_2-p_2)s}\right) t^2+\left(1+a^2c^2t^2\right) \cosh((q_2-p_2) s)\\
		-2act \sinh((q_2-p_2) s)-\cosh((p_2+q_2) s+2 n_1 t).
		\end{align*}
		
		When $\lvert a c\rvert=0$, the switched system~\eqref{eq:system} is stable for all signals $\sigma$. This can be observed by taking $a=0$ without loss of generality. The condition $f(\lambda_2,t,s)<0$ reduces to \[
		\lambda_2^4<\frac{4\, \sinh (q_2s+n_1t)\, \sinh (p_2s+n_1t)}{c^4\, {\rm{e}}^{(q_2-p_2) s}}=m(t,s).
		\]
		Since $(\partial m/\partial s)(t,s)>0$ for $t,s>0$. The above expression is equivalent to $\lambda_2^4<m(t,0)=4c^{-4}t^{-2}\sinh^2n_1 t$, for all $t>0$. Since $m(t,0)$ is increasing in $t$ with $\lim_{t\to 0}m(t,0)=4n_1^2c^{-4}$. Taking any $\lambda_2<\sqrt{4n_1^2c^{-4}}$, we have $f(\lambda_2,t,s)<0$ for all $t,s>0$. Thus we have proved the claim.\\		
		When both $a$ and $c$ are nonzero, we define $k(t,s)=f\left(\sqrt[4]{\frac{a^2}{c^2}{\rm{e}}^{-(q_2-p_2)s}},t,s\right)$. Then \[
		k(t,s)=a^2c^2t^2+\left(1+a^2c^2t^2\right)\, \cosh((q_2-p_2) s)-2act\, \sinh((q_2-p_2) s)-\cosh((p_2+q_2) s+2 n_1 t).
		\] 	
		We will prove that the zero set $\mathcal{C}=\{(t,s)\in Q_1\colon\ k(t,s)=0\}$ can be classified into four cases, as shown in Figures~\ref{fig:nrcase12} and~\ref{fig:nrcase34}. We will then use arguments similar to those in Remark~\ref{zeroset} to compute $\tau_{2,1}$. First notice that $\mathcal{C}=\{(t,s)\in Q_1\colon\ ac=R_\pm(t,s)\}$, where \[
		R_\pm(t,s)=\frac{1}{t} \sech\left(\left(\frac{q_2-p_2}{2}\right)s\right)\left[\sinh\left(\frac{q_2-p_2}{2}\right)s\pm \sinh\left(\left(\frac{q_2+p_2}{2}\right)s+n_1 t\right)\right].
		\]
		The following observations can be made about the functions $R_\pm$: \begin{enumerate}
			\item $R_+(t,s)>0$ and $R_-(t,s)<0$ for all $t,s>0$. Thus, if $ac$ is positive (negative, respectively) the zero set of $k$ is given by the points $(t,s)\in Q_1$ satisfying $ac=R_+(t,s)$ ($ac=R_-(t,s)$, respectively).
			\item $R_+(t,0)$ is an increasing function of $t$, thus there exists a unique $\tilde{t}>0$ such that $R_+ (\tilde{t},0)=ac$ whenever $a c>n_1$. When $0<ac\le n_1$, we have $ac<R_+(t,0)$ for all $t>0$.\\ 
			Similarly $R_-(t,0)$ is decreasing in $t$, thus there exists a unique $\tilde{t}>0$ such that $R_- (\tilde{t},0)=ac$ whenever $a c<-n_1$. If $-n_1\le ac<0$, then $R_-(t,0)<ac$ for all $t>0$. 
			
			\item $(\partial R_+/\partial s)(t,s)>0$ and $(\partial R_-/\partial s)(t,s)<0$ for all $t,s>0$. Combining with the previous observation, we can conclude that when $\lvert ac\rvert\le n_1$, $R_-(t,s)<ac<R_+(t,s)$ for all $t,s>0$. Thus, $\tau_{2,1}=0$. When $\lvert ac\rvert>n_1$, the zero set of $k$ lies entirely in the strip $0\le t\le \tilde{t}$.\\
			Similar to the proof of Lemma~\ref{rr:zero:IFT}, we can use the above observation to parameterize the zero set of $k$ in the closed first quadrant as the graph of a $\mathcal{C}^1$ function $S:[0,\tilde{t}]\to \mathbb{R}$ using the implicit function theorem.
			\item For each fixed $s_0>0$, there is a unique $t_0^\pm>0$ such that $(\partial R_\pm/\partial t)(t_0^\pm,s_0)=0$. Moreover,
			\begin{enumerate}
				\item $R_+(t,s_0)$ is decreasing on $(0,t_0^+)$ and increasing on $(t_0^+,\infty)$ with $\lim_{t\to 0}R_+(t,s_0)=\lim_{t\to\infty}R_+(t,s_0)=\infty$.
				\item $R_-(t,s_0)$ is increasing on $(0,t_0^-)$ and decreasing on $(t_0^-,\infty)$ with $\lim_{t\to 0}R_-(t,s_0)=\lim_{t\to\infty}R_-(t,s_0)=-\infty$.
			\end{enumerate}
		Thus, for $ac>n_1$ ($ac<-n_1$, respectively), for each $s_0>0$, there are at most two positive solutions of $R_+(t,s_0)=ac$ ($R_-(t,s_0)=ac$, respectively).
		\end{enumerate}
				\begin{figure}[h!]
			\centering
			\includegraphics[scale=0.3]{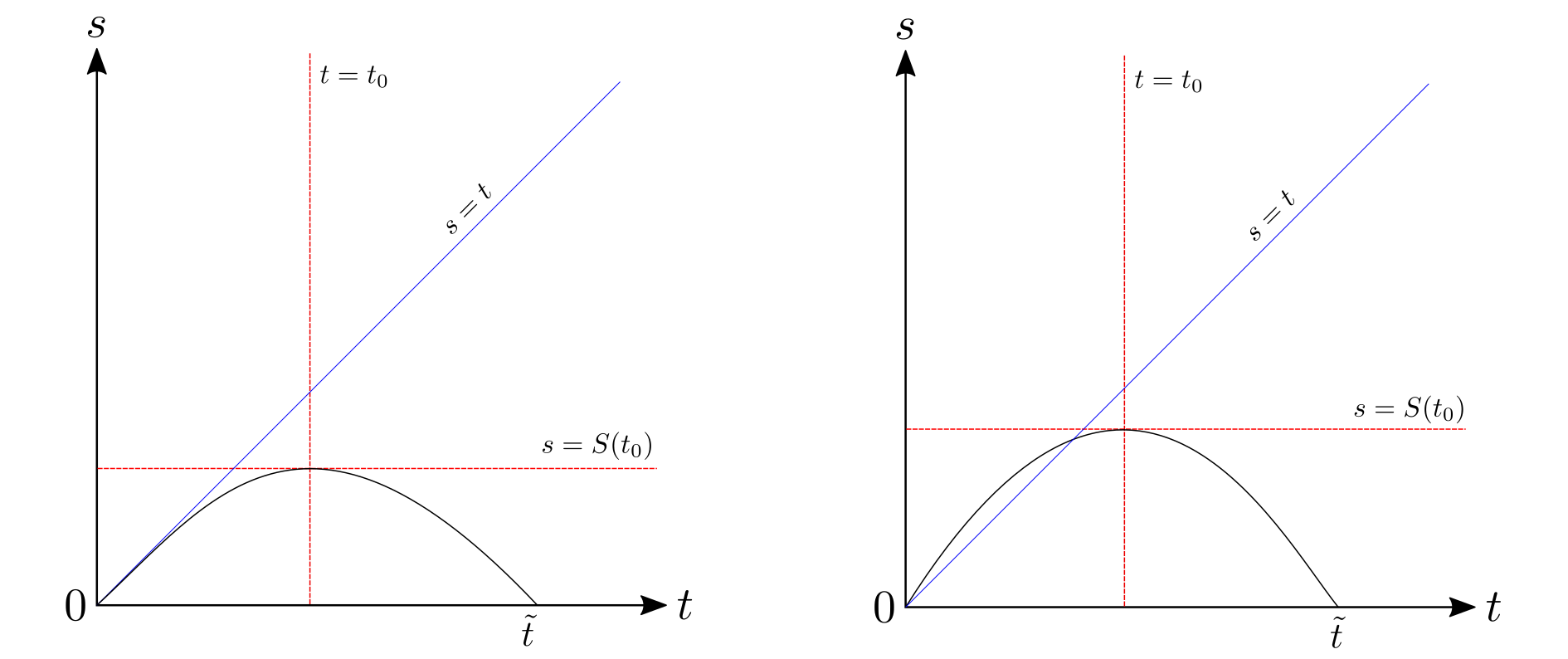}\caption{Zero set of $k(t,s)$ when $S(t_0)\le t_0$.}\label{fig:nrcase12}
		\end{figure}
	
		\begin{figure}[h!]
			\centering
			\includegraphics[scale=0.3]{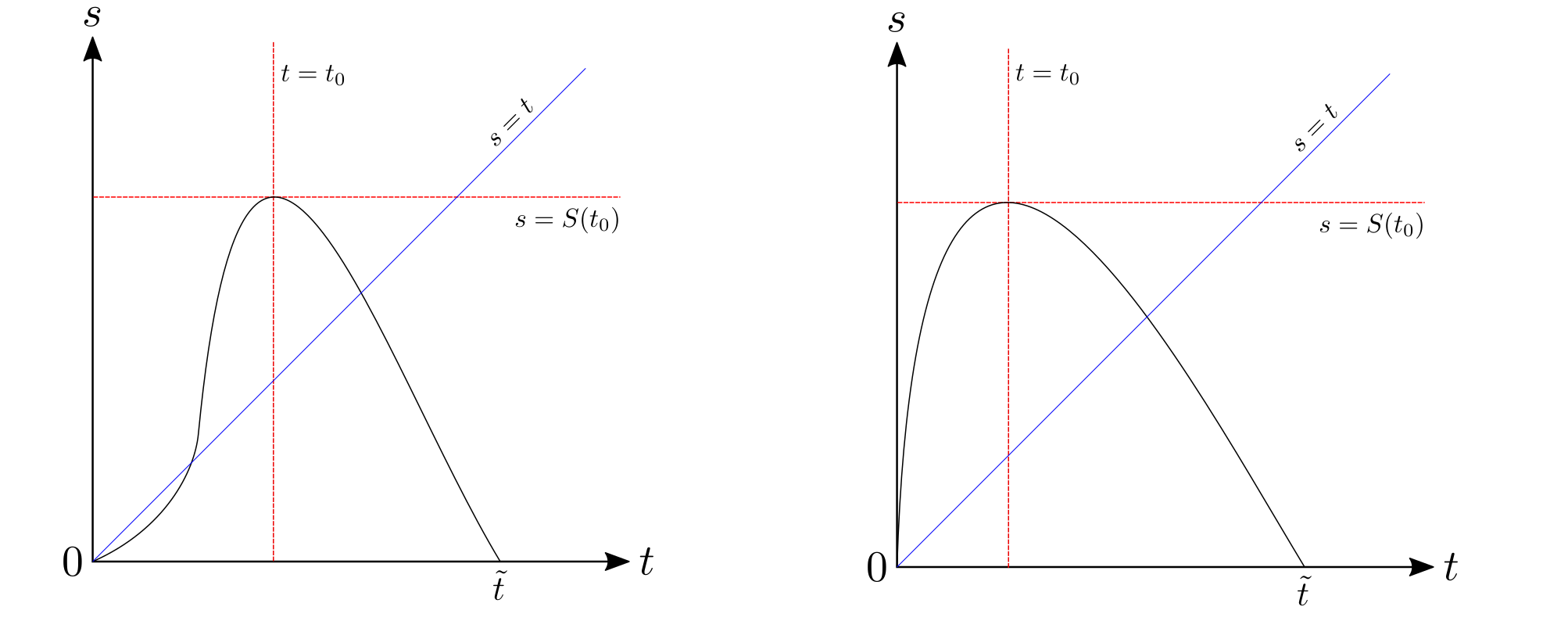}\caption{Zero set of $k(t,s)$ when $S(t_0)>t_0$.}\label{fig:nrcase34}
		\end{figure}
		Combining all the above observations, similar to the proof of Lemma~\ref{rr:lemma:uniquebulge}, we can conclude that when $\lvert ac\rvert>n_1$, $S$ has a unique maxima at $t_0\in(0,\tilde{t})$, say. When $ac>n_1$, the pair $(t_0,S(t_0))$ can be calculated as the unique nonzero solution $(t,s)\in\mathcal{C}$ satisfying $(\partial/\partial t)(R_+)(t,s)=0$, that is, \[
		n_1t \cosh\left(\left(\frac{q_2+p_2}{2}\right)s+n_1t\right)=\sinh\left(\left(\frac{q_2-p_2}{2}\right)s\right)+\sinh\left(\left(\frac{q_2+p_2}{2}\right)s+n_1t\right).
		\]
		When $ac<-n_1$, the pair $(t_0,S(t_0))$ can be calculated as the unique nonzero solution $(t,s)\in\mathcal{C}$ satisfying $(\partial/\partial t)(R_-)(t,s)=0$, that is, \[
		n_1t \cosh\left(\left(\frac{q_2+p_2}{2}\right)s+n_1t\right)=\sinh\left(\left(\frac{q_2+p_2}{2}\right)s+n_1t\right)-\sinh\left(\left(\frac{q_2-p_2}{2}\right)s\right).
		\]
		The above calculations show that the zero set of $k$ can be classified into four cases as shown in Figures~\ref{fig:nrcase12} and~\ref{fig:nrcase34}. If $S(t_0)\le t_0$, let $\tau=S(t_0)$; otherwise let $\tau>t_0$ be the unique solution of $k(t,t)=0$. Then taking $\lambda_2=\sqrt[4]{\frac{a^2}{c^2}{\rm{e}}^{-(q_2-p_2)\tau}}$, we have $f(\lambda_2,t,\tau)=k(t,\tau)$ for all $t>0$. Hence by the choice of $\tau$, we have $f(\lambda_2,t,\tau)<0$ for $t>\tau$ in each of the cases. Since, $f$ is decreasing in $s$, this implies $f(\lambda_2,t,s)<0$ for $t,s>\tau$.
	\end{proof}
	
	\begin{remark}[Optimization of $\tau_{2,1}$]\label{nr:optimization}
		Refer Remarks~\ref{nn:optimal} and \ref{nc:optimization} for notations. Let $P_2$ be the optimal choice that we obtained in Theorem~\ref{nr:thm21}. Since $\widetilde{M}(\epsilon_1)= P_2^{-1}\widetilde{P}_1^{(\epsilon_1)}=\begin{pmatrix}a &b+\epsilon_1 a\\c & d+\epsilon_1 c\end{pmatrix}$, $R=ac$ is invariant under choice of $P_1$. In Theorems~\ref{nr:thm12} and~\ref{nr:thm12<1/2}, for finding $K_{opt}$ we can minimize the function $\widetilde{K}(\epsilon_1)$ corresponding to $\widetilde{M}(\epsilon_1)$, which is given by $\widetilde{K}(\epsilon_1)=\left(a^2+(b+\epsilon_1 a)^2\right)\left(c^2+(d+\epsilon_1 c)^2\right)$.
	\end{remark}

	\section{Comparison with existing literature}\label{sec:comparison}
	In this section, we will present several examples of bimodal planar switched linear systems of the form~\eqref{eq:system}. In each of the examples, we will compare the dwell time bounds $\tau_{1,2},\tau_{2,1},\tau$ obtained in this paper with existing bounds given in~\cite{agarwal2018simple,morse1996supervisory,geromel2006stability,karabacak2009dwell}. The notations $\tau_{M}$, $\tau_{GC}$ and $\tau_{Kar}$ are used for the the dwell time bounds obtained in~\cite{geromel2006stability},~\cite[Theorem 2]{karabacak2009dwell}, and~\cite{morse1996supervisory}, respectively. \\
	The bound $\tau_M$ is given by
	\begin{eqnarray*}
		\tau_M &=& \max_{i=1,2}\ \ \inf_{\alpha>0, \beta>0}\ \ \{\alpha/\beta\ \vert \ \Vert e^{A_it}\Vert<e^{\alpha-\beta t}, \ t\ge 0\}.
	\end{eqnarray*}
For $i,j=1,2$, consider the following set of matrix inequalities
	\begin{eqnarray}\label{eq:GC}
		A_i^T P_i +P_iA_i &<& 0, \nonumber \\
		e^{A_i^T S}P_j e^{A_i S}-P_i &<& 0, \ \ \text{ for } i\ne j.
	\end{eqnarray}
	Then the dwell time bound $\tau_{GC}$ is given by
	\begin{eqnarray*}
		\tau_{GC} &=& \min_{S>0,\ P_1>0, P_2>0} \{S\ \vert\ \text{the set of inequalities~\eqref{eq:GC} are satisfied}\}.	
	\end{eqnarray*}
The dwell time bound $\tau_{GC}$ is the solution of an optimization problem constrained by certain linear matrix inequalities (LMI) and can be computed using the Robust Control Toolbox in MATLAB 2021a. The code in~\cite{horvath2018algorithm} can be designed to compute $\tau_{GC}$ by setting certain matrices in the LMI zero. \\
For the bound $\tau_{Kar}$, if none of the matrices $A_1,A_2$ are defective, let $J_i$ be the complex Jordan form of $A_i$ and $P_i$ be an invertible matrix with both the columns having unit norm satisfying $A_i=P_iJ_iP_i^{-1}$, for $i=1,2$. Then 
	\begin{eqnarray*}
		\tau_{Kar} &=& \dfrac{\ln\left(\Vert P_1^{-1}P_2\Vert\ \Vert P_2^{-1}P_1\Vert\right)}{\lambda_1^*+\lambda_2^*},
	\end{eqnarray*}
where $\lambda_i^*$ denotes the absolute of the real part of the eigenvalue of $A_i$ closer to the imaginary axis. \\
When one of the subsystem matrices, say $A_1$ is defective, $\tau_{Kar}$ is calculated by extending the result in \cite{karabacak2009dwell} using the inequality $\|{\rm{e}}^{J_1 t}\|\le c_{\theta_1} {\rm{e}}^{-\theta_1 t}$, where $0<\theta_1<\lambda_1^*$. The dwell time for any $0<\theta_1<\lambda_1^*$ is \[
\tau_{Kar}=\dfrac{\ln\left(c_{\theta_1}\Vert P_1^{-1}P_2\Vert\ \Vert P_2^{-1}P_1\Vert\right)}{\theta_1+\lambda_2^*}.
\] 
Similarly when both subsystems are defective, an expression for dwell time can be obtained for any $0<\theta_i<\lambda_i^*$ for $i=1,2$ as
\[
\tau_{Kar}=\dfrac{\ln\left(c_{\theta_1}c_{\theta_2}\Vert P_1^{-1}P_2\Vert\ \Vert P_2^{-1}P_1\Vert\right)}{\theta_1+\theta_2}.
\]	
Note that the switched system~\eqref{eq:system} is asymptotically stable for all signals from the collection $S_{\tau}$ for $\tau=\tau_{M}, \tau_{GC}$ or $\tau>\tau_{Kar}$ (see~\cite[Theorem 3]{karabacak2009dwell}). However the switched system~\eqref{eq:system} is only stable for signals from the collection $S_{\tau_{Kar}}$ (see~\cite[Theorem 2]{karabacak2009dwell}).
 
See Table~\ref{table:example} for examples and comparison of dwell time bounds. The following subsystems will be used in the examples:	
	\[
	X_1=\begin{pmatrix} -0.2 & -5\\ 1 & -0.3\end{pmatrix}, X_2=\begin{pmatrix}-0.4& -1\\5 & -0.6\end{pmatrix}, X_3=\begin{pmatrix} 0 & 1\\ -2 & -1\end{pmatrix}, 
	\]
\[
X_4=\begin{pmatrix} 0 & 1\\ -9 & -1\end{pmatrix}, X_5=\nobreak \begin{pmatrix} 0.01 & -0.1\\ 0.231 & -0.31\end{pmatrix}, X_6=\frac{1}{13}\begin{pmatrix} -18 & 25\\ -1 & -8 \end{pmatrix}, 
\]
\[
X_7=\begin{pmatrix} -0.2 & 0\\ 0 & -0.42\end{pmatrix}, X_8=\frac{1}{18}\begin{pmatrix}-113 & -25\\361 & 77\end{pmatrix}, X_9=\begin{pmatrix}
-0.08 & -0.06\\ 0.04 & -0.22\end{pmatrix}.
\]
The matrices $X_1,X_2,X_3$ and $X_4$ have complex eigenvalues, $X_5$, $X_7$ and $X_9$ are real diagonalizable, and the matrices $X_6$ and $X_8$ are defective. \\	
For the pair of subsystems $X_6,X_3$; $X_8,X_3$; and $X_8,X_9$ (one of the subsystem matrices is defective), computations show that the least value of $\tau_{Kar}$ is achieved at $\theta_1=0.666076$, $\theta_1=0.763393$ and $\theta_1=0.833934$, respectively. For the pair $X_6,X_8$ (both subsystem matrices defective), the least value of $\tau_{Kar}$ is obtained at $(\theta_1,\theta_2)=(0.846682,0.935004)$.
\begin{table}[h!]
	\centering
	\begin{tabular}{ |c|c|c|c|c|c|c|c|c| } 
		\hline
		Subsystems $A_1,A_2$& $\tau_{M}$ & $\tau_{GC}$  & $\tau_{Kar}$ &  $\tau_{1,2}$ &$\tau_{2,1}$&$\tau=\min\{\tau_{1,2},\tau_{2,1}\}$ \\ 
		\hline
		$X_1$, $X_2$ & 3.24 & 2.1384  & 2.1472 &  2.1472 & 2.1472 & 2.1472 \\ 
		$X_3, X_4$ & 2.23 & 0.6222  & 0.8047 & 0.8047 & 0.8047 & 0.8047\\
		$X_5$, $X_7$ & 12.41 & 0.0001  & 6.3739 &  0.4120 & 0.0774 & 0.0774\\  
		$X_5, X_3$ & 12.41 & 0.5587  & 3.6803 & 3.4673 & 0.4768 & 0.4768\\
		$X_8, X_3$ & 4.59 & 1.9666  & 3.7175 &  1.6046 & 2.2714 & 1.6046\\
		$X_6,X_8$ &	4.59 &2.9681	& 3.6366	& 3.0348 &3.0348	& 3.0348 \\
		$X_8,X_9$ & 4.59 & 2.5882 & 5.6798 & 2.3769 & 4.54305 & 2.3769 \\
		\hline
	\end{tabular}
	\caption{Comparison of dwell time bounds.}
	\label{table:example}
\end{table}

\noindent For the pair of subsystems $X_3, X_4$, a better bound 0.6073 can be obtained using homogeneous polynomials, as in~\cite{chesi2010computing}. Observe the relationship $\tau_{1,2}, \tau_{2,1}\le \tau_{Kar}$ and $\tau\le \tau_M$ in Table~\ref{table:example}. Note that among these examples, the bound $\tau_{GC}$ is greater than $\tau$ for the pair of subsystems $X_5, X_3$; $X_8,X_3$; and $X_8,X_9$, and is smaller in the other examples.

The concept of simple loop dwell time was introduced in~\cite{agarwal2018simple}. As per Figure~\ref{fig:signal}, simple loop dwell time is a lower bound on each $t_i+s_i$ which ensures stability of the switched system~\eqref{eq:system}. It is given by \[\tau_{loop}=\dfrac{\ln\left(\Vert P_1^{-1}P_2\Vert\ \Vert P_2^{-1}P_1\Vert\right)}{\min(\lambda_1^*,\lambda_2^*)},\] 
where the notations are the same as in $\tau_{Kar}$. In Theorems~\ref{cc:thm12} and \ref{cc:thm21}, we had proved that the line $s=-(\alpha_1/\alpha_2)t+c_0$ bounds the zero set, refer Figure~\ref{fig:cc}. This implies that the zero set lies below the line $s+t=c_0\cdot\max(1,\alpha_2/\alpha_1)$. The Schur's function, thus, is negative in the region $\{(t,s)\colon\ s+t>c_0\cdot\max(1,\alpha_2/\alpha_1)\}$ and hence $c_0\cdot\max(1,\alpha_2/\alpha_1)$ works as a simple loop dwell time. Moreover this value is bounded above by $\tau_{loop}$. A similar argument can be used for the case discussed in Theorem~\ref{rc:thm12}. For cases discussed in Theorem~\ref{rc:thm21} and \ref{rr:thm12}, let $D_1, D_2$ be the optimal choices as in these results, then \[
\left\|M_{D_1,D_2}^{-1}{\rm{e}}^{J_2 s}\,M_{D_1,D_2} {\rm{e}}^{J_1 t} \right\|\le \left\|M^{-1}{\rm{e}}^{J_2 s}\,M{\rm{e}}^{J_1 t} \right\|\le \left\|M^{-1}\right\| \left\|M\right\|{\rm{e}}^{\min(\lambda_1^*,\lambda_2^*)(t+s)}.\] Hence, the value $\tau_0=\max\{s+t\colon\ f(t,s)=0\}$ becomes the simple loop dwell time, where $f$ is the Schur's function of either of the cases. Moreover, $\tau_0\le\tau_{loop}$. As an example, consider the pair of subsystems $X_5, X_3$. Refer to Figure~\ref{fig:loopcomparison} for the zero set of Schur's function $f_{5,3}$ corresponding to Section~\ref{sub:rc21}. Here $\tau_{loop}=22.0819$ and $\tau_0=11.8556$. \begin{figure}[h!]
	\centering
	\includegraphics[width=.4\textwidth]{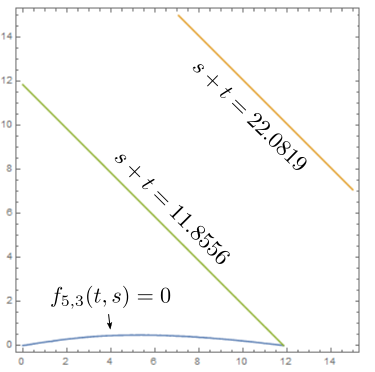}
	\caption{Comparison of simple loop dwell time for the pair $X_5, X_3$.}
	\label{fig:loopcomparison}
\end{figure}

		\section{Examples}\label{sec:examples}
			\begin{example}
			We will now show that when the subsystems $A_1$ and $A_2$ share a common eigenvector, then the switched system~\eqref{eq:system} is stable for any signal $\sigma$. Let us look at all the possibilities one by one.
			\begin{enumerate}
				\item \textit{Both $A_1$ and $A_2$ real diagonalizable}:\\ If these matrices share a common eigenvector, then one of the entries of transition matrix $M=P_2^{-1}P_1$ is zero for all possible Jordan matrices $P_1,P_2$. Hence the switched system~\eqref{eq:system} is stable for any signal $\sigma$ using  Proposition~\ref{rr:thm:ad=0,1}.
				\item \textit{One of the matrices has complex eigenvalues}: \\
				$A_1$ and $A_2$ can share a common eigenvector if both have complex eigenvalues. In this case $M=\pm \text{diag}(1,\pm 1)$ and hence the switched system~\eqref{eq:system} is stable for any signal $\sigma$ using Theorem~\ref{cc:thm12}. 
				\item \textit{$A_1$ defective and $A_2$ real diagonalizable}: \\
				If the subsystems share a common eigenvector, then either $a$ or $c$ entry is zero in the transition matrix  $M=P_2^{-1}P_1$ is zero for all possible Jordan matrices $P_1,P_2$. Hence the switched system~\eqref{eq:system} is stable for any signal $\sigma$ using Theorem~\ref{nr:thm21}.
				\item \textit{Both $A_1$ and $A_2$ defective}: \\
				The subsystem matrices $A_1$ and $A_2$ share a common eigenvector if and only if the entry $c$ is zero in the transition matrix  $M=P_2^{-1}P_1$ for all possible Jordan matrices $P_1,P_2$.\\				
				For $i=1,2$, let $P_i=\begin{bmatrix} \vec{u}_i & \vec{v}_i	\end{bmatrix}$, where $\vec{u}_1=\vec{u}_2$. For $\xi>0$, let $\widehat{P}_2=\begin{bmatrix} \vec{u}_2/\sqrt{\xi} & \sqrt{\xi}\, \vec{v}_2	\end{bmatrix}$. Then $A_2=\widehat{P}_2\widehat{J}_2\widehat{P}_2^{-1}$, where $\widehat{J}_2=\begin{pmatrix}
					-n_2 & \xi\\ 0 & -n_2
				\end{pmatrix}$. With $\widehat{P}_2$ as the eigenvector matrix, the new transition matrix $\widehat{M}$, say, has the $11$-entry as $a\sqrt{\xi}$ and the $21$-entry as 0. In this setting, $\Vert \widehat{M}e^{J_1t}\widehat{M}^{-1}e^{\widehat{J}_2s}\Vert \le \Vert \widehat{M}e^{J_1t}\widehat{M}^{-1}e^{-n_2s}\Vert \theta(\xi s)$. As argued in Section~\ref{sec:nnt21}, the final term is less than $1$ if and only if the conditions (C1) and (C2) are satisfied with $c=0$, $a$ replaced by $a\sqrt{\xi}$ and $n_2$ replaced by $n_2/\xi$. Now choose $\xi$ satisfying both $n_2/\xi>1/2$ and $a^2\xi/2\le n_1$. That is, $0<\xi<\min\{2n_2,2n_1/a^2\}$. Then by Theorem~\ref{nn:thm21>1/2}, the switched system~\eqref{eq:system} is stable for any signal $\sigma$.				
			\end{enumerate}
			
		\end{example}
		
	\begin{example}
		In this example, we give a generalization of the results discussed in this paper to a planar linear switched system with $N$ subsystems given by Hurwitz matrices.
		Let $A_1,\dots,A_N$ be planar Hurwitz matrices. Suppose the underlying graph determining the switching between subsystems is as shown in Figure~\ref{fig:petal}, where $A_1$ is at the center of the graph and $A_2,\dots,A_{N}$ are at the $N-1$ petals of the flower. Consider the switching signal $1 i_1 1 i_2 \dots 1 i_k \dots$, that is, for each $j\in\mathbb{N}$, $A_{i_j}$ is the $(2j)^{th}$ active subsystem where time $s_j>0$ is spent, and $A_1$ is active otherwise. The signal spends times $t_j$ in the subsystem $A_1$ at the $j^{th}$ instance. The flow of the switched system~\eqref{eq:system} is given by 
		\begin{eqnarray}
			x(t)=\prod_{j\in\mathbb{N}}\left({\rm{e}}^{A_{i_j}s_j}{\rm{e}}^{A_{1}t_j}\right)x(0)=\prod_{j\in\mathbb{N}}\left(P_1^{-1}P_{i_j}{\rm{e}}^{J_{i_j}s_j}P_{i_j}^{-1}P_1{\rm{e}}^{J_{1}t_j}\right)P_1^{-1}x(0).
		\end{eqnarray}
		\begin{figure}[h!]
			\centering
			\includegraphics[width=.4\textwidth]{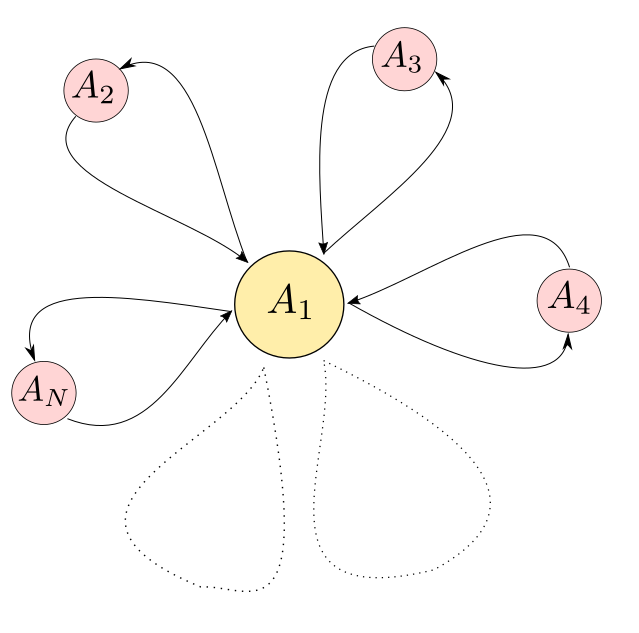}
			\caption{Underlying graph.}
			\label{fig:petal}
		\end{figure}
		In the right most expression, the terms are grouped in the same way as in our results for computing $\tau_{1,2}$. Let us look at the various possibilities for the subsystem matrix $A_1$.
		\begin{enumerate}
			\item \textit{$A_1$ defective}: \\
			For each $j=2,\dots,N$, calculate $\tau_{j}$ such that $\|P_1^{-1}P_{j}{\rm{e}}^{J_{j}s}P_{j}^{-1}P_1{\rm{e}}^{J_{1}t}\|<1$, for all $t,s>\tau_j$. For this, we can use relevant results from Sections~\ref{sec:NN},~\ref{sec:NC}, and~\ref{sec:NR} to compute $\tau_{1,2}$ with $A_1,A_j$ as the subsystems, denote it as $\tau_{j}$. Then for $\tau=\max_{j=2,\dots,N} \tau_{j}$, the switched system is stable for all signals $\sigma\in S_\tau$ and asymptotically stable for all signals $\sigma\in S'_{\tau}$.
			\item \textit{$A_1$ has complex eigenvalues}: \\
			For each $j=2,\dots,N$, calculate $\tau_{j}$ such that $\|P_1^{-1}P_{j}{\rm{e}}^{J_{j}s}P_{j}^{-1}P_1{\rm{e}}^{J_{1}t}\|<1$, for all $t,s>\tau_j$. For this, we can use relevant results from Sections~\ref{sec:CC},~\ref{sec:RC}, and~\ref{sec:NC} to compute $\tau_{1,2}$ with $A_1,A_j$ as the subsystems, denote it as $\tau_{j}$. Then for $\tau=\max_{j=2,\dots,N} \tau_{j}$, the switched system is stable for all signals $\sigma\in S_\tau$ and asymptotically stable for all signals $\sigma\in S'_{\tau}$.
			\item \textit{$A_1$ real diagonalizable}: \\
			Suppose $D_1=\text{diag}(\lambda_1,1/\lambda_1)$ be the scaling matrix corresponding to $A_1$. The curve $\|D_1^{-1}P_1^{-1}P_{j}{\rm{e}}^{J_{j}s}P_{j}^{-1}P_1D_1{\rm{e}}^{-p_1 t}\|=1$ bounds the zero set of the Schur's function since $\|D_1^{-1}P_1^{-1}P_{j}{\rm{e}}^{J_{j}s}P_{j}^{-1}P_1D_1{\rm{e}}^{J_{1}t}\|\le \|D_1^{-1}P_1^{-1}P_{j}{\rm{e}}^{J_{j}s}P_{j}^{-1}P_1D_1{\rm{e}}^{-p_1 t}\|$. Let $P_j^{-1}P_1=\begin{pmatrix}
				a_j & b_j\\
				c_j & d_j\\
			\end{pmatrix}$ has determinant $1$, for each $j=2,\ldots,n$. Then the curve $\|D_1^{-1}P_1^{-1}P_{j}{\rm{e}}^{J_{j}s}P_{j}^{-1}P_1D_1{\rm{e}}^{-p_1 t}\|=1$ is parametrized as follows: 
			\begin{enumerate}
				\item \textit{$A_j$ has complex eigenvalues}: Using results from Section~\ref{sec:RC}, $T_{\lambda_1,j}(s)=-\frac{\alpha_j}{p_1}s+\frac{1}{p_1}\cosh^{-1}\left(\frac{a_j^2+c_j^2}{2}\lambda_1^2+\frac{b_j^2+d_j^2}{2\lambda_1^2}\right)$. In this case, $\tau_j(\lambda_1)$ equals the unique fixed point of $T_{\lambda_1,j}$.
				\item \textit{$A_j$ defective}: Using results from Section~\ref{sec:NR}, $T_{\lambda_1,j}(s)=-\frac{n_j}{p_1} s+\frac{1}{p_1}\sinh^{-1}\left(\frac{d_j^2}{2\lambda_1^2}+\frac{c_j^2}{2}\lambda_1^2\right)s$.\\ 
				Let $K(\lambda_1)=\frac{d_j^2}{2\lambda_1^2}+\frac{c_j^2}{2}\lambda_1^2 $, \begin{enumerate}
					\item if $K(\lambda_1)\le n_j$, $\tau_j(\lambda_1)=0$.
					\item if $K(\lambda_1)>n_j$, let $s_0=\sqrt{(1/n_j^2)-1/(K(\lambda_1))^2}$.\\ 
					If $s_0\ge T_{\lambda_1,j}(s_0)$, $\tau_j(\lambda_1)=T_{\lambda_1,j}(s_0)$; otherwise $\tau_j(\lambda_1)$ is the unique fixed point of $T_{\lambda_1,j}$.
				\end{enumerate}
				\item \textit{$A_j$ real diagonalizable}: Using results from Section~\ref{sec:RR}, $T_{\lambda_1,j}(s)=-\left(\frac{q_2+p_2}{2p_1}\right)s+\frac{1}{p_1}\sinh^{-1}\left(\sqrt{R(\lambda_1)}\sinh\left(\frac{q_2-p_2}{2}\right)s\right)$, where $R(\lambda_1)=a_j^2d_j^2+b_j^2c_j^2+\frac{b_j^2d_j^2}{\lambda_1^4}+a_j^2c_j^2\lambda_1^4$. \begin{enumerate}
					\item if $R(\lambda_1)\le (q_2+p_2)/(q_2-p_2)$, $\tau_j(\lambda_1)=0$.
					\item If $R(\lambda_1)>(q_2+p_2)/(q_2-p_2)$, let $s_0$ be the unique solution $s>0$ of \[
					\tanh\left(\frac{q_2-p_2}{2}\right)s=\frac{R(\lambda_1)\left(\frac{q_2-p_2}{q_2+p_2}\right)-1}{R(\lambda_1)-1}.
					\]\\ If $s_0\ge T_{\lambda_1,j}(s_0)$, $\tau_j(\lambda_1)=T_{\lambda_1,j}(s_0)$; otherwise $\tau_j(\lambda_1)$ is the unique fixed point of $T_{\lambda_1,j}$.
				\end{enumerate}
			\end{enumerate}
			The proof of (a) is straightforward since $T_{\lambda_1,j}$ is linear in $s$ with a negative slope. The proofs of (b) and (c) rely on the fact that $T_{\lambda_1,j}$ is concave down in each of these cases and hence has a unique maxima.\\
			Thus we need to find a fixed scaling matrix $D_1$ (that is, a choice of $\lambda_1$) which solves $\tau=\min_{\lambda_1} \max_{j=2, \dots,N}\tau_j(\lambda_1)$. The inequality used in the beginning of this case,\\ $\|D_1^{-1}P_1^{-1}P_{j}{\rm{e}}^{J_{j}s}P_{j}^{-1}P_1D_1{\rm{e}}^{J_{1}t}\|\le \|D_1^{-1}P_1^{-1}P_{j}{\rm{e}}^{J_{j}s}P_{j}^{-1}P_1D_1{\rm{e}}^{-p_1 t}\|$, is strict since $p_1\ne q_1$, the switched system is asymptotically stable for all signals $\sigma\in S_\tau$.
		\end{enumerate}
	
\end{example}

\begin{appendix}
	
	\section{Proofs of some results stated in Section~\ref{sec:RR}}\label{ap:proof}
	
	\begin{proof}[Proof of Lemma~\ref{rr:lemma:ad:notinI}]\label{rr:lemma:ad:notinI:proof}
	When $ad>q_2/(q_2-p_2)$ for each value of $s$, say $s_1>0$, analyze the function $k(t,s_1)$. Observe that $k(t,s_1)=0$ if and only if $ad=\ell_-(t,s_1)$. Also		
		\begin{equation*}
			\frac{\partial \ell_-}{\partial t} (t,s)=\frac{p_1{\rm{e}}^{p_1 t}(1+{\rm{e}}^{p_2s+q_1 t})(1+{\rm{e}}^{q_1t+q_2s})+q_1 {\rm{e}}^{q_1 t}({\rm{e}}^{p_1 t+p_2 s}-1)({\rm{e}}^{p_1 t+q_2 s}-1)}{({\rm{e}}^{p_1t}+{\rm{e}}^{q_1t})^2\,({\rm{e}}^{q_2 s}-{\rm{e}}^{p_2 s})}.
		\end{equation*}
		Since $k(t,0)=-\left(1-{\rm e}^{-2p_1 t}\right)\left(1-{\rm e}^{-2q_1 t}\right)$, $k(t,0)<0$ for all $t>0$. From Lemma~\ref{rr:lemma:k0s}, $k(0,s)>0$ for all $s\in(0,s_0)$. Thus, for any $s_1\in(0,s_0)$, we have $ad>\ell_-(t,s_1)$. Further since $\ell_-$ is increasing in $t$, there exists a unique $t_1$ such that $ad=\ell_-(t_1,s_1)$, that is, $k(t_1,s_1)=0$.
		
		Similarly, it can be shown that $k(t,s_0)<0$ for all $t>0$; and $k(t,s_1)<0$ for all $t\ge 0$, $s_1>s_0$. Hence we obtain the result.\\
		
		When $ad<-p_2/(q_2-p_2)$, the proof is similar. For each value of $s$, say $s_1$, analyze the function $k(t,s_1)$. We get $k(t,s_1)=0$ if and only if $ad=\ell_+(t,s_1)$. For all $s>0$, $\ell_+(t,s)$ is strictly decreasing in $t$ since
		
		\begin{equation*}
			\frac{\partial \ell_+}{\partial t} (t,s)=-\frac{p_1{\rm{e}}^{p_1 t}({\rm{e}}^{p_2s+q_1 t}-1)({\rm{e}}^{q_1t+q_2s}-1)+q_1 {\rm{e}}^{q_1 t}({\rm{e}}^{p_1 t+p_2 s}+1)({\rm{e}}^{p_1 t+q_2 s}+1)}{({\rm{e}}^{p_1t}+{\rm{e}}^{q_1t})^2\,({\rm{e}}^{q_2 s}-{\rm{e}}^{p_2 s})}.
		\end{equation*}
		
		Now, note that for all $s_1\in (0,s_0)$, $\lim_{t\to 0}\,\ell_+(t,s)=-p_2/(q_2-p_2)$ and since $ad<-p_2/(q_2-p_2)$, there exists exactly one point $t_1$ such that $ad=\ell_+(t_1,s_1)$, since $\ell_+$ is decreasing in $t$. Again, for $s\in\{0,s_0\}$, there is exactly one solution. However for $s_1>s_0$, there is no solution.
	\end{proof}
	
	\begin{proof}[Intermediate details of proof of Lemma~\ref{rr:zero:IFT}]
		For each $s\in(0,s_0)$, there exists $t_s>0$ (depending on $s$) such that $ad=\ell_-(t_s,s)$. Since $\left(\partial \ell_-/\partial t\right)(t_s,s)>0$, by the implicit function theorem, there exists a neighborhood $U_s\subset (0,s_0)$ containing $s$ and a $\mathcal{C}^1$ function $\mathcal{O}_s\colon\, U_s\to \mathbb{R}$ such that $\ell_-(\mathcal{O}_s(y),y)=ad$, for all $y\in U_s$. 
		
		Moreover, by Lemma~\ref{rr:lemma:ad:notinI}, $\mathcal{O}_{s_1}(s)=\mathcal{O}_{s_2}(s)$ for all $s\in U_{s_1}\cap U_{s_2}$. Since $(0,s_0)=\cup_s\ U_s$, one can define a $\mathcal{C}^1$ function $\mathcal{O}\colon\, (0,s_0)\to\mathbb{R}$ as $\mathcal{O}(s)=\mathcal{O}_s(s)$. Clearly this function satisfies $\ell_-(\mathcal{O}(s),s)=ad$ and hence $k(\mathcal{O}(s),s)=0$, for all $s\in(0,s_0)$.
			\end{proof}

	\begin{lemma}\label{sinuv}
		If $0<v<u$, $\sinh(tu)/\sinh(tv)$ is increasing in $t$.
	\end{lemma}
	\begin{proof} Let $u>v>0$, then 
		\begin{align*}
			\left(\frac{\sinh (tu)}{\sinh (tv)}\right)'=\left(\frac{{\rm{e}}^{tu}-{\rm{e}}^{-tu}}{{\rm{e}}^{tv}-{\rm{e}}^{-tv}}\right)'=\frac{{\rm{e}}^{t(v-u)}({\rm{e}}^{2tv}-1)({\rm{e}}^{2tu}-1)}{2t({\rm{e}}^{2tv}-1)^2}\left[\frac{2tu\,({\rm{e}}^{2tu}+1)}{{\rm{e}}^{2tu}-1}-\frac{2tv\,({\rm{e}}^{2tv}+1)}{{\rm{e}}^{2tv}-1}\right].
		\end{align*}
		Consider the following function $q(x)=\frac{x({\rm{e}}^{x}+1)}{{\rm{e}}^{x}-1}$. Then $q'(x)=\frac{{\rm{e}}^{2x}-2x{\rm{e}}^{x}-1}{({\rm{e}}^{x}-1)^2}$. Note that $q'(0)=0$ and $q'(x)>0$ for all $x>0$. Hence $q(x)$ is an increasing function for $x>0$. Therefore for all $t>0$, $
		\left(\frac{\sinh (tu)}{\sinh (tv)}\right)'>0$.
	\end{proof}
	
	\begin{proof}[Proof of Theorem~\ref{rr:lemma:uniquebulge}]\label{rr:lemma:uniquebulge:proof} Recall the notation $\mathcal{C}=\{(t,s)\in Q_1\colon\ k(t,s)=0\}$. \\
		For $ad>q_2/(q_2-p_2)$, we have $\mathcal{C}=\{(t,s)\in Q_1\colon\ \ell_-(t,s)=ad\}$ since $\ell_+$ is negative on $Q_1$. Since 
		\[
		\frac{\partial \ell_-}{\partial s} (t,s)=\frac{-q_2\,{\rm e}^{q_2 s}\,({\rm e}^{p_2 s+p_1 t}-1)({\rm e}^{p_2 s+q_1 t}+1)+p_2\, {\rm e}^{p_2 s}\,({\rm e}^{q_2 s+p_1 t}-1)({\rm e}^{q_2 s+q_1 t}+1)}{({\rm e}^{p_2 s}-{\rm e}^{q_2 s})^2\,({\rm e}^{p_1 t}+{\rm e}^{q_1 t})},
		\]
		when $s\ne 0$, $(\partial \ell_-/\partial s) (t,s)<0$ if and only if		
		\begin{eqnarray}\label{ltsiff}
			\frac{q_2}{p_2}>\frac{({\rm e}^{p_1 t}-{\rm e}^{-q_2 s})({\rm e}^{q_1 t}+{\rm e}^{-q_2 s})}{({\rm e}^{p_1 t}-{\rm e}^{-p_2 s})({\rm e}^{q_1 t}+{\rm e}^{-p_2 s})}{\rm e}^{(q_2-p_2)s}=L(t,s).
		\end{eqnarray}
		The function $L(0,s)=\sinh q_2 s/\sinh p_2 s$ is strictly increasing by Lemma~\ref{sinuv}. We will now show that for any $t_0>0$, the function $L(t_0,s)$ is increasing in $s$. Further for $t_0>0$, we rewrite $L(t_0,s)=R_1(t_0,s) R_2(t_0,s)$, where
		\[
		R_1(t,s)=\frac{{\rm e}^{\frac{1}{2}q_2 s+p_1 t}-{\rm e}^{-\frac{1}{2}q_2 s}}{{\rm e}^{\frac{1}{2}p_2 s+p_1 t}-{\rm e}^{-\frac{1}{2}p_2 s}}, \ R_2(t,s)=\frac{{\rm e}^{\frac{1}{2}q_2 s+q_1 t}+{\rm e}^{-\frac{1}{2}q_2 s}}{{\rm e}^{\frac{1}{2}p_2 s+q_1 t}+{\rm e}^{-\frac{1}{2}p_2 s}}.
		\]
		We will now prove that both the functions $R_1$ and $R_2$ are increasing in $s>0$ and hence $L$ is increasing in $s>0$. It can be checked that $(\partial R_1/\partial s)(t,s)>0$ if and only if
		\begin{eqnarray*}
			\frac{q_2-p_2}{2}\left({\rm e}^{\frac{1}{2}(q_2+p_2)s+2p_1 t}-{\rm e}^{-\frac{1}{2}(q_2+p_2)s} \right)>\frac{q_2+p_2}{2}\left({\rm e}^{\frac{1}{2}(q_2-p_2)s+p_1 t}-{\rm e}^{-\frac{1}{2}(q_2-p_2)s+p_1 t} \right).
		\end{eqnarray*}
		Multiplying ${\rm e}^{-p_1 t}$ on both sides, we have
		\begin{eqnarray*}
			(q_2-p_2)\sinh\left(\frac{1}{2}(q_2+p_2)s+p_1 t\right)>(q_2+p_2)\sinh\left(\frac{1}{2}(q_2-p_2)s\right).
		\end{eqnarray*}
		Clearly, the inequality holds true for any $t>0$ and $s=0$ and the derivative with respect to $s$ of the left hand side is greater than that of the right hand side for all $s\ge 0$. This implies that the inequality holds true for all $t,s>0$. Hence, for all $t_0>0$, $(\partial R_1/\partial s)(t_0,s)$ is an increasing function of $s$.  
		Similarly, it can be shown that $(\partial R_2/\partial s)(t,s)>0$ if and only if		
		\begin{eqnarray*}
			(q_2-p_2)\sinh\left(\frac{1}{2}(q_2+p_2)s+q_1 t\right)>-(q_2+p_2)\sinh\left(\frac{1}{2}(q_2-p_2)s\right),
		\end{eqnarray*}
		which is true for all $t,s>0$. Thus, for all $t>0$, $(\partial R_2/\partial s)(t,s)$ is increasing in $s$. Thus $L(t,s)$, being a product of two positive increasing functions in $s$, is increasing in $s$; and $\lim_{s\to\infty} L(t,s)=\infty$.
		Now, for all $t_0>0$, $L(t_0,0)=1<q_2/p_2$ and since the function $L(t_0,s)$ is increasing in $s$, there exists a unique $r_0>0$ such that $L(t_0,r_0)=q_2/p_2$. Hence, by~\eqref{ltsiff}, we have 
		\begin{eqnarray}
			\frac{\partial \ell_-}{\partial s} (t_0,s)\begin{cases}<0,\hspace{10pt}s\in[0,r_0)\\
				=0,\hspace{10pt}s=r_0\\
				>0,\hspace{10pt}s\in(r_0,\infty). \end{cases}
		\end{eqnarray}
		Thus, for a fixed $t_0$, we have
		\begin{enumerate}
			\item If $\ell_-(t_0,r_0)=ad$, then unique root of $k(t_0,s)$.
			\item If $\ell_-(t_0,r_0)<ad$, then two roots of $k(t_0,s)$.
			\item If $\ell_-(t_0,r_0)>ad$, then no root of $k(t_0,s)$.
		\end{enumerate}
		This implies that $\mathcal{O}$, as defined in Lemma~\ref{rr:zero:IFT}, parametrizing the curve $\mathcal{C}$, has a unique local maxima, say at $\tilde{s}$. The point $(\mathcal{O}(\tilde{s}),\tilde{s})$ is given by the unique nonzero solution $(t,s)$ of the system
		\begin{eqnarray}
			L(t,s)=\frac{q_2	}{p_2}, \ \ k(t,s)=0.
		\end{eqnarray}
		When $ad<-p_2/(q_2-p_2)<0$, $k(t,s)=0$ if and only if $\ell_+(t,s)=ad$. Now
		$$\frac{\partial \ell_+}{\partial s} (t,s)=\frac{q_2\,{\rm e}^{q_2 s}\,({\rm e}^{p_2 s+p_1 t}+1)({\rm e}^{p_2 s+q_1 t}-1)-p_2\, {\rm e}^{p_2 s}\,({\rm e}^{q_2 s+p_1 t}+1)({\rm e}^{q_2 s+q_1 t}-1)}{({\rm e}^{p_2 s}-{\rm e}^{q_2 s})^2\,({\rm e}^{p_1 t}+{\rm e}^{q_1 t})},$$	
		when $s\ne 0$. Hence $(\partial \ell_+/\partial s) (t,s)>0$ if and only if
		\begin{eqnarray*}
			\frac{q_2}{p_2}>\frac{{\rm e}^{\frac{1}{2}q_2 s+q_1 t}-{\rm e}^{-\frac{1}{2}q_2 s}}{{\rm e}^{\frac{1}{2}p_2 s+q_1 t}-{\rm e}^{-\frac{1}{2}p_2 s}}\cdot \frac{{\rm e}^{\frac{1}{2}q_2 s+p_1 t}+{\rm e}^{-\frac{1}{2}q_2 s}}{{\rm e}^{\frac{1}{2}p_2 s+p_1 t}+{\rm e}^{-\frac{1}{2}p_2 s}}=\widetilde{L}(t,s).
		\end{eqnarray*}
	The function $\widetilde{L}(t,s)$ is a product of two positive increasing functions in $s$, hence is increasing in $s$. Also, for all $t_0>0$, $\widetilde{L}(t_0,0)=1<q_2/p_2$. Also note that $\lim_{s\to\infty}\,\widetilde{L}(t_0,s)=\infty$. Since the function $\widetilde{L}(t_0,s)$ is increasing in $s$, there exists a unique $s_0$ such that $\widetilde{L}(t_0,s_0)=q_2/p_2$. Hence, we have 		
		\begin{eqnarray}
			\frac{\partial \ell_+}{\partial s} (t_0,s)\begin{cases}>0,\hspace{10pt}s\in[0,s_0)\\
				=0,\hspace{10pt}s=s_0\\
				<0,\hspace{10pt}s\in(s_0,\infty). \end{cases}
		\end{eqnarray}
		Thus, for a fixed $t_0$, we have		
		\begin{enumerate}
			\item If $\ell_-(t_0,s_0)=ad$, then unique root of $k(t_0,s)$.
			\item If $\ell_-(t_0,s_0)>ad$, then two roots of $k(t_0,s)$.
			\item If $\ell_-(t_0,s_0)<ad$, then no root of $k(t_0,s)$.
		\end{enumerate}
		Again, this implies that $\mathcal{O}$ has a unique local maxima, say at $\tilde{s}$. The point $(\mathcal{O}(\tilde{s}),\tilde{s})$ is given by the unique nonzero solution $(t,s)$ of the system
		\begin{eqnarray}
			\widetilde{L}(t,s)=\frac{q_2}{p_2},\ \	k(t,s)=0.
		\end{eqnarray}
		
	\end{proof}

	\begin{proof}[Proof of Lemma~\ref{rr:direction:k}]\label{rr:direction:k:proof}
		Define $k^{(\gamma)}(t)=k(t,\gamma t)$ and $\ell^{(1)}_-(t)=\ell_-(t,t)$ for all $\gamma>0$ and $t>0$. Note that $\ell^{(1)}_-$ is increasing in $t$. Consider the function $\ell_-(t,s)$ for $(t,s)\ne(0,0)$, in the direction $s=\gamma t$. Denote this function by $\ell_-^{(\gamma)}(t)=\ell_-(t,\gamma t)$. This function is the same as $\ell^{(1)}_-(t)$ with just $p_2$ and $q_2$ replaced by $p_2\gamma$ and $q_2\gamma$, respectively. Since $\ell^{(1)}_-$ is increasing in $t$, so is $\ell_-^{(\gamma)}$. Observe that $k^{(\gamma)}(0)=0$ for all $\gamma$. Also for $t\ne 0$, $k^{(\gamma)}(t)=k(t,\gamma t)=0$ if and only if $ad=\ell_-^{(\gamma)}(t)$.\\
		If $ad\le\lim_{t\to 0}\ell_-^{(\gamma)}(t)$ for some $\gamma$, then since the function $\ell^{(\gamma)}_-$ is increasing in $t>0$, we get $ad<\ell_-^{(\gamma)}(t)$ for all $t>0$ and there is no positive root of the function $k^{(\gamma)}(t)$.\\ 
		On the other hand, if $ad>\lim_{t\to 0}\ell_-^{(\gamma)}(t)$ for some $\gamma$, then there is a unique positive root $t_{max}^{(\gamma)}$ of the function $k^{(\gamma)}(t)$. Also $k^{(\gamma)}(t)>0$ when $t<t_{max}^{(\gamma)}$ and $k^{(\gamma)}(t)<0$ when $t>t_{max}^{(\gamma)}$.
	\end{proof}

	\begin{proof}[Proof of Lemma \ref{rr:lemma:f}]\label{rr:lemma:f:proof}
Since
		\[
		f(\lambda_0,t,s)-k(t,s)=\left(\frac{b^2d^2}{\lambda_0^4}{\rm{e}}^{-2q_1 t}+a^2c^2\lambda_0^4{\rm{e}}^{-2p_1 t}-2\lvert ad(ad-1)\rvert {\rm{e}}^{-(q_1+p_1)t}\right)({\rm{e}}^{-p_2 s}-{\rm{e}}^{-q_2 s})^2,
		\]
		and $\lambda_0=\sqrt[8]{\frac{b^2d^2}{a^2c^2}\, {\rm{e}}^{-2(q_1-p_1)t_0}}$, using the fact that $ad-bc=1$, we have
		\[
		f(\lambda_0,0,s)=k(0,s)+\lvert ad(ad-1)\rvert \left( \frac{1}{\sqrt{{\rm{e}}^{-(q_1-p_1)t_0}}}-\sqrt{{\rm{e}}^{-(q_1-p_1)t_0}}\right)^2 ({\rm{e}}^{-p_2 s}-{\rm{e}}^{-q_2 s})^2.
		\]
		Since $ad(ad-1)>0$ when $ad\notin \mathcal{I}$, using the expression for $k(0,s)$ from Lemma~\ref{rr:lemma:k0s},
		\begin{align*}
		f(\lambda_0,0,s) = 	\left({\rm{e}}^{-p_2 s}-{\rm{e}}^{-q_2 s}\right)^2\left[ad(ad-1)\left\lbrace 4+ \left( \frac{1}{\sqrt{{\rm{e}}^{-(q_1-p_1)t_0}}}-\sqrt{{\rm{e}}^{-(q_1-p_1)t_0}}\right)^2\right\rbrace-m(0,s)\right],
		\end{align*}
		where $m(t,s)$ is as defined in Theorem~\ref{rr:thm:ad=0,1}.  We conclude that for all values of $ad\notin \mathcal{I}$, 		
		\begin{eqnarray*}
			\lim_{s\to 0} \dfrac{f(\lambda_0,0,s)}{\left({\rm{e}}^{-p_2 s}-{\rm{e}}^{-q_2 s}\right)^2}>\lim_{s\to 0}\left(4ad(ad-1)-m(0,s)\right)>0,
		\end{eqnarray*}
		where the last inequality is proved in Lemma~\ref{rr:lemma:k0s}. Also, since the function $m(0,s)$ is increasing with $\lim_{s\rightarrow\infty}m(0,s)=\infty$, there is a unique $\widetilde{s_0}>0$ such that $f(\lambda_0,0,\widetilde{s_0})=\nobreak 0$. Moreover $f(\lambda_0,0,s)>0$ for all $s\in(0,\widetilde{s_0})$.
	\end{proof}
	
\end{appendix}

	\bibliographystyle{siam}
	\bibliography{References}
\end{document}